\documentclass[12pt]{amsart}
\usepackage{amsfonts, amssymb, latexsym, graphicx, hyperref, amsmath,enumitem,amsthm}
\usepackage{color}
\usepackage{tikz}
\usepackage{mathtools}
\usepackage[vcentermath]{youngtab}
\usepackage{stackengine}
\usepackage{ytableau}
\ytableausetup{centertableaux}
\usepackage{algorithm,algorithmic}
\usetikzlibrary{calc, shapes, backgrounds,arrows,positioning,decorations.pathmorphing}

\newtheorem{theorem}{Theorem}[section]
\newtheorem{proposition}[theorem]{Proposition}
\newtheorem{lemma}[theorem]{Lemma}

\newtheorem{claim}[theorem]{Claim}
\newtheorem*{claim*}{Claim}

\newtheorem{corollary}[theorem]{Corollary}
\newtheorem{Main Conjecture}[theorem]{Main Conjecture}

\theoremstyle{remark}

\newtheorem{example}[theorem]{Example}

\newtheorem{remark}[theorem]{Remark}
\theoremstyle{plain}

\newtheorem*{nonvanishing*}{Nonvanishing}

\newtheorem*{theorem*}{Theorem}
\newtheorem*{lemma*}{Lemma}
\newtheorem*{corollary*}{Corollary}

\newcommand\polyName{indicator polytope}
\newcommand\tabFCI{{\sf PerfectTab}}
\newcommand\tabSort{{\sf PerfectTab}_{\downarrow}}

\newcommand{\TT}{\mathcal{T}}
\newcommand{\LL}{\mathcal{L}}

\newcommand{\cellsize}{19}
\newlength{\cellsz} 
\newsavebox{\cell}
\sbox{\cell}{\begin{picture}(\cellsize,\cellsize)
\put(0,0){\line(1,0){\cellsize}}
\put(0,0){\line(0,1){\cellsize}}
\put(\cellsize,0){\line(0,1){\cellsize}}
\put(0,\cellsize){\line(1,0){\cellsize}}
\end{picture}}
\newcommand\cellify[1]{\def\thearg{#1}\def\nothing{}%
\ifx\thearg\nothing
\vrule width0pt height\cellsz depth0pt\else
\hbox to 0pt{\usebox{\cell} \hss}\fi%
\vbox to \cellsz{
\vss
\hbox to \cellsz{\hss$#1$\hss}
\vss}}
\newcommand\tableau[1]{\vtop{\let\\\cr
\baselineskip -16000pt \lineskiplimit 16000pt \lineskip 0pt
\ialign{&\cellify{##}\cr#1\crcr}}}

\newcommand{\excise}[1]{}

\newcommand{\nolabel}{\circ}

\begin{document}
\pagestyle{plain}
\title{{An efficient algorithm for deciding vanishing of Schubert polynomial coefficients}}
\author{Anshul Adve}
\author{Colleen Robichaux}
\author{Alexander Yong}
\address{UCLA, Los Angeles, CA 90095}
\email{aadve@g.ucla.edu}
\address{Dept.~of Mathematics, U.~Illinois at Urbana-Champaign, Urbana, IL 61801, USA} 
\email{cer2@illinois.edu, ayong@illinois.edu}
\date{March 3, 2021}

\begin{abstract}
\emph{Schubert polynomials} form a basis of all polynomials and appear in the study of cohomology rings of flag manifolds. The {vanishing} problem for Schubert polynomials asks if a coefficient of a Schubert polynomial is zero. We give a tableau criterion to solve this problem, 
from which we deduce the first polynomial time algorithm. These results are obtained from new characterizations of the \emph{Schubitope}, a generalization of the  permutahedron defined for any
subset of the $n\times n$ grid. In contrast, we show that computing these coefficients explicitly is $\#${\sf P}-{\sf complete}. 
\end{abstract}
\maketitle

\vspace{-.2in}

\section{Introduction}\label{sec:intro}
\emph{Schubert polynomials} form a  linear basis of all polynomials ${\mathbb Z}[x_1,x_2,x_3,\ldots]$. They were introduced by A.~Lascoux--M.-P.~Sch\"{u}tzenberger \cite{LS1} to study the cohomology ring of the flag manifold. These polynomials
represent the Schubert classes under the Borel isomorphism. A reference is the textbook \cite{Fulton:YT}.

If $w_0=n \ n-1 \ \cdots 2 \ 1 $ is the longest length permutation 
in $S_n$, then 
\[{\mathfrak S}_{w_0}(x_1,\ldots,x_n):=x_1^{n-1}x_2^{n-2}\cdots x_{n-1}.\] 
Otherwise, $w\neq w_0$ and 
there exists $i$ such that $w(i)<w(i+1)$. Then one sets
\[{\mathfrak S_w}(x_1,\ldots,x_n)=\partial_i {\mathfrak S}_{ws_i}(x_1,\ldots,x_n),\] 
where $s_i$ is the transposition swapping $i$ and $i+1$ and
\[\partial_i f:= \frac{f(\ldots,x_i,x_{i+1},\ldots)-f(\ldots,x_{i+1},x_i,\ldots)}{x_i-x_{i+1}}.\] 
Since $\partial_i$ satisfies
\[\partial_i\partial_j=\partial_j\partial_i \text{\ for $|i-j|>1$, and \ } 
\partial_i\partial_{i+1}\partial_i=\partial_{i+1}\partial_i\partial_{i+1},\]
the above description of ${\mathfrak S}_w$ is well-defined. In addition,
under the inclusion 
$\iota: S_n\hookrightarrow S_{n+1}$ defined by
$w(1)\cdots w(n) \mapsto w(1) \ \cdots w(n) \ n+1$,
${\mathfrak S}_w={\mathfrak S}_{\iota(w)}$. 
Thus one unambiguously
refers to ${\mathfrak S}_w$ for each $w\in S_{\infty}=\bigcup_{n\geq 1} S_n$. 

The \emph{graph} $G(w)$ of a
permutation $w\in {S}_n$ is the $n\times n$ grid, with a $\bullet$ placed
in position $(i,w(i))$ (in matrix coordinates).
The \emph{Rothe diagram}
of $w$ is given by 
\[D(w)=\{(i,j): 1\leq i,j\leq n, j<w(i), i<w^{-1}(j)\}.\] 
This is pictorially described with rays that
strike out boxes south and east of each $\bullet$ in $G(w)$. $D(w)$ are the remaining boxes.

The \emph{code} of $w$, denoted ${\sf code}(w)$ is the vector $(c_1,c_2,\ldots,c_L)$ where $c_i$ is the number of boxes in the $i$-th row 
of $D(w)$ and $L$ indexes the southmost row with a positive number of boxes. 
To each $w\in S_{\infty}$ there is a unique associated code; see \cite[Proposition~2.1.2]{Manivel}.
\begin{example}\label{ex:Rothe}
If 
$w=53841267 \in {S}_{8}$ (in one line notation)
then $D(w)$ is depicted by: 
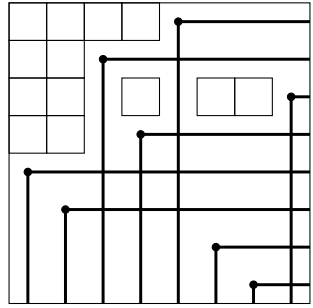
\[
\begin{tikzpicture}[scale=.5]
\draw (0,0) rectangle (8,8);

\draw (0,7) rectangle (1,8);
\draw (1,7) rectangle (2,8);
\draw (2,7) rectangle (3,8);
\draw (3,7) rectangle (4,8);

\draw (0,6) rectangle (1,7);
\draw (1,6) rectangle (2,7);

\draw (0,5) rectangle (1,6);
\draw (1,5) rectangle (2,6);

\draw (0,4) rectangle (1,5);
\draw (1,4) rectangle (2,5);

\draw (3,5) rectangle (4,6);

\draw (5,5) rectangle (6,6);

\draw (6,5) rectangle (7,6);

\filldraw (4.5,7.5) circle (.5ex);
\draw[line width = .2ex] (4.5,0) -- (4.5,7.5) -- (8,7.5);
\filldraw (2.5,6.5) circle (.5ex);
\draw[line width = .2ex] (2.5,0) -- (2.5,6.5) -- (8,6.5);
\filldraw (7.5,5.5) circle (.5ex);
\draw[line width = .2ex] (7.5,0) -- (7.5,5.5) -- (8,5.5);
\filldraw (3.5,4.5) circle (.5ex);
\draw[line width = .2ex] (3.5,0) -- (3.5,4.5) -- (8,4.5);
\filldraw (0.5,3.5) circle (.5ex);
\draw[line width = .2ex] (0.5,0) -- (0.5,3.5) -- (8,3.5);
\filldraw (1.5,2.5) circle (.5ex);
\draw[line width = .2ex] (1.5,0) -- (1.5,2.5) -- (8,2.5);
\filldraw (5.5,1.5) circle (.5ex);
\draw[line width = .2ex] (5.5,0) -- (5.5,1.5) -- (8,1.5);
\filldraw (6.5,0.5) circle (.5ex);
\draw[line width = .2ex] (6.5,0) -- (6.5,0.5) -- (8,0.5);
\end{tikzpicture}
\]
Here, ${\sf code}(w)=(4,2,5,2)$.
\end{example}

Consider the monomial expansion
\[{\mathfrak S}_w=\sum_{\alpha\in\mathbb{Z}_{\geq0}^n}c_{\alpha,w}x^{\alpha}.\]
Now, $c_{\alpha,w}=0$ unless $\alpha_i=0$ for $i>L$, and moreover, $c_{\alpha,w}\in {\mathbb Z}_{\geq 0}$.
Let {\sf Schubert} be the problem of deciding $c_{\alpha,w}\neq 0$,
as measured in the input size of $\alpha$ and $w$ (under the assumption that arithmetic operations take constant time).
The {\tt INPUT} is ${\sf code}=(c_1,\ldots c_L)\in {\mathbb Z}_{\geq 0}^L$ with $c_L>0$
and $\alpha\in {\mathbb Z}_{\geq 0}^L$. ${\sf Schubert}$ returns {\tt YES} if $c_{\alpha,w}>0$ and {\tt NO} otherwise.

\begin{theorem}
\label{thm:SchubertinP}
${\sf Schubert}\in {\sf P}$.
\end{theorem}

We prove Theorem~\ref{thm:SchubertinP} using another result. Fix $n \in \mathbb{Z}_{>0}$ and let $D \subseteq [n]^2$. We call $D$ a \emph{diagram} and visualize $D$ as a subset of an $n \times n$ grid of boxes, oriented so that $(r,c) \in [n]^2$ represents the box in the $r$th row from the top and the $c$th column from the left. Let ${\tabFCI}(D,\alpha)$ 
be the fillings of $D$ with $\alpha_k$ many $k$'s, where entries 
in each column are distinct, any entry in row $i$ is $\leq i$, and each box contains exactly one entry. Let 
${\tabSort}(D,\alpha) \subseteq {\tabFCI}(D,\alpha)$ 
be fillings where entries in each column increase from top to bottom.

\begin{theorem}
\label{thm:Advetableau}
$c_{\alpha,w}>0\iff {\tabFCI}(D(w),\alpha)\neq \emptyset 
\iff {\tabSort}(D(w),\alpha)\neq \emptyset$
\end{theorem}

In general $\#{\tabFCI}(D(w),\alpha)\neq c_{\alpha,w}$ but rather 
$\#{\tabFCI}(D(w),\alpha)\geq c_{\alpha,w}$ (cf.~\cite{FGRS}). 

\begin{example}
Here are the tableaux in  $\bigcup_{\alpha} {\tabSort}(D(31524),\alpha)$:
\begin{center}
\begin{tikzpicture}[scale=.4]
\draw (0,0) rectangle (5,5);
\draw (0,4) rectangle (1,5) node[pos=.5] {$1$};
\draw (1,4) rectangle (2,5) node[pos=.5] {$1$};
\draw (1,2) rectangle (2,3) node[pos=.5] {$2$};
\draw (3,2) rectangle (4,3) node[pos=.5] {$2$};
\filldraw (2.5,4.5) circle (.5ex);
\draw[line width = .2ex] (2.5,0) -- (2.5,4.5) -- (5,4.5);
\filldraw (0.5,3.5) circle (.5ex);
\draw[line width = .2ex] (0.5,0) -- (0.5,3.5) -- (5,3.5);
\filldraw (4.5,2.5) circle (.5ex);
\draw[line width = .2ex] (4.5,0) -- (4.5,2.5) -- (5,2.5);
\filldraw (1.5,1.5) circle (.5ex);
\draw[line width = .2ex] (1.5,0) -- (1.5,1.5) -- (5,1.5);
\filldraw (3.5,0.5) circle (.5ex);
\draw[line width = .2ex] (3.5,0) -- (3.5,0.5) -- (5,0.5);

\end{tikzpicture}\qquad
\begin{tikzpicture}[scale=.4]
\draw (0,0) rectangle (5,5);
\draw (0,4) rectangle (1,5) node[pos=.5] {$1$};
\draw (1,4) rectangle (2,5) node[pos=.5] {$1$};
\draw (1,2) rectangle (2,3) node[pos=.5] {$2$};
\draw (3,2) rectangle (4,3) node[pos=.5] {$1$};
\filldraw (2.5,4.5) circle (.5ex);
\draw[line width = .2ex] (2.5,0) -- (2.5,4.5) -- (5,4.5);
\filldraw (0.5,3.5) circle (.5ex);
\draw[line width = .2ex] (0.5,0) -- (0.5,3.5) -- (5,3.5);
\filldraw (4.5,2.5) circle (.5ex);
\draw[line width = .2ex] (4.5,0) -- (4.5,2.5) -- (5,2.5);
\filldraw (1.5,1.5) circle (.5ex);
\draw[line width = .2ex] (1.5,0) -- (1.5,1.5) -- (5,1.5);
\filldraw (3.5,0.5) circle (.5ex);
\draw[line width = .2ex] (3.5,0) -- (3.5,0.5) -- (5,0.5);
\end{tikzpicture}\qquad
\begin{tikzpicture}[scale=.4]
\draw (0,0) rectangle (5,5);

\draw (0,4) rectangle (1,5) node[pos=.5] {$1$};
\draw (1,4) rectangle (2,5) node[pos=.5] {$1$};

\draw (1,2) rectangle (2,3) node[pos=.5] {$3$};

\draw (3,2) rectangle (4,3) node[pos=.5] {$1$};
\filldraw (2.5,4.5) circle (.5ex);
\draw[line width = .2ex] (2.5,0) -- (2.5,4.5) -- (5,4.5);
\filldraw (0.5,3.5) circle (.5ex);
\draw[line width = .2ex] (0.5,0) -- (0.5,3.5) -- (5,3.5);
\filldraw (4.5,2.5) circle (.5ex);
\draw[line width = .2ex] (4.5,0) -- (4.5,2.5) -- (5,2.5);
\filldraw (1.5,1.5) circle (.5ex);
\draw[line width = .2ex] (1.5,0) -- (1.5,1.5) -- (5,1.5);
\filldraw (3.5,0.5) circle (.5ex);
\draw[line width = .2ex] (3.5,0) -- (3.5,0.5) -- (5,0.5);
\end{tikzpicture}

\vspace*{0.5cm}
\begin{tikzpicture}[scale=.4]

\draw (0,0) rectangle (5,5);
\draw (0,4) rectangle (1,5) node[pos=.5] {$1$};
\draw (1,4) rectangle (2,5) node[pos=.5] {$1$};
\draw (1,2) rectangle (2,3) node[pos=.5] {$2$};
\draw (3,2) rectangle (4,3) node[pos=.5] {$3$};
\filldraw (2.5,4.5) circle (.5ex);
\draw[line width = .2ex] (2.5,0) -- (2.5,4.5) -- (5,4.5);
\filldraw (0.5,3.5) circle (.5ex);
\draw[line width = .2ex] (0.5,0) -- (0.5,3.5) -- (5,3.5);
\filldraw (4.5,2.5) circle (.5ex);
\draw[line width = .2ex] (4.5,0) -- (4.5,2.5) -- (5,2.5);
\filldraw (1.5,1.5) circle (.5ex);
\draw[line width = .2ex] (1.5,0) -- (1.5,1.5) -- (5,1.5);
\filldraw (3.5,0.5) circle (.5ex);
\draw[line width = .2ex] (3.5,0) -- (3.5,0.5) -- (5,0.5);
\end{tikzpicture}\qquad
\begin{tikzpicture}[scale=.4]

\draw (0,0) rectangle (5,5);

\draw (0,4) rectangle (1,5) node[pos=.5] {$1$};
\draw (1,4) rectangle (2,5) node[pos=.5] {$1$};
\draw (1,2) rectangle (2,3) node[pos=.5] {$3$};
\draw (3,2) rectangle (4,3) node[pos=.5] {$2$};

\filldraw (2.5,4.5) circle (.5ex);
\draw[line width = .2ex] (2.5,0) -- (2.5,4.5) -- (5,4.5);
\filldraw (0.5,3.5) circle (.5ex);
\draw[line width = .2ex] (0.5,0) -- (0.5,3.5) -- (5,3.5);
\filldraw (4.5,2.5) circle (.5ex);
\draw[line width = .2ex] (4.5,0) -- (4.5,2.5) -- (5,2.5);
\filldraw (1.5,1.5) circle (.5ex);
\draw[line width = .2ex] (1.5,0) -- (1.5,1.5) -- (5,1.5);
\filldraw (3.5,0.5) circle (.5ex);
\draw[line width = .2ex] (3.5,0) -- (3.5,0.5) -- (5,0.5);
\end{tikzpicture}\qquad
\begin{tikzpicture}[scale=.4]

\draw (0,0) rectangle (5,5);

\draw (0,4) rectangle (1,5) node[pos=.5] {$1$};
\draw (1,4) rectangle (2,5) node[pos=.5] {$1$};

\draw (1,2) rectangle (2,3) node[pos=.5] {$3$};

\draw (3,2) rectangle (4,3) node[pos=.5] {$3$};

\filldraw (2.5,4.5) circle (.5ex);
\draw[line width = .2ex] (2.5,0) -- (2.5,4.5) -- (5,4.5);
\filldraw (0.5,3.5) circle (.5ex);
\draw[line width = .2ex] (0.5,0) -- (0.5,3.5) -- (5,3.5);
\filldraw (4.5,2.5) circle (.5ex);
\draw[line width = .2ex] (4.5,0) -- (4.5,2.5) -- (5,2.5);
\filldraw (1.5,1.5) circle (.5ex);
\draw[line width = .2ex] (1.5,0) -- (1.5,1.5) -- (5,1.5);
\filldraw (3.5,0.5) circle (.5ex);
\draw[line width = .2ex] (3.5,0) -- (3.5,0.5) -- (5,0.5);
\end{tikzpicture}
\end{center}

Hence, for instance, $c_{(2,1,1),31524}>0$ but $c_{(4),31524}=0$.
\end{example}

To prove Theorems~\ref{thm:SchubertinP} and~\ref{thm:Advetableau} we establish results about the \emph{Schubitope} \cite{MTY}. 
This polytope ${\mathcal S}_D$ is defined with a halfspace description for any
$D\subseteq [n]^2$. We prove (Theorem~\ref{thm:independent_characterization}) that a lattice point $\alpha$ is in ${\mathcal S}_D$ if and only if ${\tabFCI}(D,\alpha)\neq {\emptyset}$ where $D$ is any diagram. 

We then introduce the \emph{\polyName} ${\mathcal P}(D,\alpha)$ whose lattice points ${\mathcal P}(D,\alpha)_{\mathbb Z}$ are in bijection with 
${\tabFCI}(D,\alpha)$. We prove  that ${\mathcal P}(D,\alpha)\neq \emptyset\iff {\mathcal P}(D,\alpha)_{\mathbb Z}\neq \emptyset$ (Theorem~\ref{thm:relaxation_equivalence}). 
Thus determining ${\mathcal P}(D,\alpha)_{\mathbb Z}\neq \emptyset$ (and equivalently $\alpha\in {\mathcal S}_{D}$) is in {\sf P} using L.~Khachiyan's ellipsoid method for linear programming, see \cite{Schrijver}.
We give two proofs of Theorem~\ref{thm:relaxation_equivalence}. The first shows 
${\mathcal P}(D,\alpha)$ is totally unimodular. Hence ${\mathcal P}(D,\alpha)\neq\emptyset$ implies 
${\mathcal P}(D,\alpha)$
 has integral vertices. Our second proof obviates total unimodularity and is potentially adaptable to problems lacking that property. 
However, only the high-level structure of the second proof is easily generalizable --- the rest is necessarily \emph{ad hoc}.

For the case of Rothe diagrams $D=D(w)$, using A.~Fink-K.~M\'esz\'aros-A.~St.~Dizier \cite[Corollary 12 and Theorem 14]{Fink} (conjectured in \cite[Conjectures 5.1 and 5.13]{MTY}), 
\begin{equation}\label{eq:schubSchub}
\alpha \in \mathcal{S}_{D(w)} \iff c_{\alpha,w}>0.
\end{equation}
This, combined with our results on the Schubitope, proves Theorems~\ref{thm:SchubertinP} and~\ref{thm:Advetableau}.

The class $\#{\sf P}$ in L.~Valiant's complexity theory of counting problems are those that count the number of accepting paths of a nondeterministic Turing machine running in polynomial time. A problem ${\mathcal P}\in \#{\sf P}$ is \emph{complete} if for any problem ${\mathcal Q}\in \#{\sf P}$
there exists a polynomial-time counting reduction from ${\mathcal Q}$ to ${\mathcal P}$. These are the hardest of the problems in $\#{\sf P}$. There does not exist a polynomial time algorithm for such problems unless
${\sf P}={\sf NP}$. 

In contrast with Theorem~\ref{thm:SchubertinP}, we prove:
\begin{theorem}
\label{thm:second}
 Counting $c_{\alpha,w}$ is $\#{\sf P}$-{\sf complete}.\end{theorem}

Given $\{c_{\alpha,w}\in {\mathbb Z}_{\geq 0} \}$ it is standard to ask for a counting rule for $c_{\alpha,w}$.
A complexity motivation is an \emph{appropriate} rule that establishes a counting problem is in $\#{\sf P}$ with respect to given input (length).
The rule of \cite{BJS} establishes that counting $c_{\alpha,w}$ is in $\#{\sf P}$ if the input is 
$(w,\alpha)$ but not if the input is
$({\sf code}(w),\alpha)$. For the latter input assumption, we use the transition algorithm of
\cite{LS:transition} and its \emph{graphical} reformulation from \cite{Knutson.Yong}. This allows us
to give a polynomial time counting reduction to the $\#{\sf P}$-{\sf complete} 
problem of counting Kostka coefficients
\cite{Narayanan}, (see Section \ref{sec:schubinsharpP}).

Related discussion of complexity and the ``nonvanishing problem'' in algebraic combinatorics appears 
in the conference proceedings \cite{FPSAC} and preprint \cite{earlierversion} versions of this paper.

\section{The Schubitope}\label{sec:SchubChar}
Consider a diagram $D\subseteq[n]^2$.
	Given $S \subseteq [n]$ and a column $c \in [n]$, construct a string denoted ${\sf word}_{c,S}(D)$ by reading column $c$ from top to bottom and recording
	\begin{itemize}
		\item $($ if $(r,c) \not\in D$ and $r \in S$,
		\item $)$ if $(r,c) \in D$ and $r \not\in S$, and
		\item $\star$ if $(r,c) \in D$ and $r \in S$.
	\end{itemize}

	Let $\theta_D^c(S) = \#\{\star\text{'s in } {\sf word}_{c,S}(D)\} + \#\{\text{paired } ()\text{'s in } {\sf word}_{c,S}(D)\}$ and
	\begin{align*}
	\theta_D(S) = \sum_{c=1}^{n} \theta_D^c(S).
	\end{align*}
	
\begin{example}\label{ex:schubWord} In the diagram $D$ below, we labelled the corresponding strings for ${\sf word}_{c,S}(D)$ for $S=\{1,3\}$. For instance, we see ${\sf word}_{5,\{1,3\}}(D)=(\star)$.
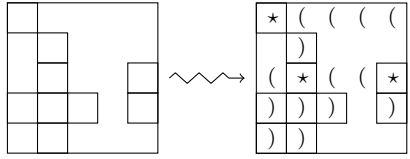
\begin{center}
\begin{tikzpicture}
\node (root) {\begin{tikzpicture}[scale=.4]
\draw (0,0) rectangle (5,5);

\draw (0,4) rectangle (1,5);

\draw (1,3) rectangle (2,4);

\draw (1,2) rectangle (2,3);
\draw (4,2) rectangle (5,3);

\draw (0,1) rectangle (1,2);
\draw (1,1) rectangle (2,2);
\draw (2,1) rectangle (3,2);
\draw (4,1) rectangle (5,2);

\draw (1,0) rectangle (2,1);
\end{tikzpicture}};

\node[right=1 of root] (temp) {\begin{tikzpicture}[scale=.4]
\draw (0,0) rectangle (5,5);

\draw (0,4) rectangle (1,5);

\draw (1,3) rectangle (2,4);

\draw (1,2) rectangle (2,3);
\draw (4,2) rectangle (5,3);

\draw (0,1) rectangle (1,2);
\draw (1,1) rectangle (2,2);
\draw (2,1) rectangle (3,2);
\draw (4,1) rectangle (5,2);

\draw (0,0) rectangle (1,1);
\draw (1,0) rectangle (2,1);

\node at (0,4.5) {\tiny$\star$};
\node at (1,4.5) {\tiny$($};
\node at (2,4.5) {\tiny$($};
\node at (3,4.5) {\tiny$($};
\node at (4,4.5) {\tiny$($};

\node at (0,2.5) {\tiny$($};
\node at (1,2.5) {\tiny$\star$};
\node at (2,2.5) {\tiny$($};
\node at (3,2.5) {\tiny$($};
\node at (4,2.5) {\tiny$\star$};

\node at (1,3.5) {\tiny$)$};

\node at (0,1.5) {\tiny$)$};
\node at (1,1.5) {\tiny$)$};
\node at (2,1.5) {\tiny$)$};
\node at (4,1.5) {\tiny$)$};

\node at (0,0.5) {\tiny$)$};
\node at (1,0.5) {\tiny$)$};

\end{tikzpicture}};

\draw [->,
line join=round,
decorate, decoration={
    zigzag,
    segment length=9,
    amplitude=2,post=lineto,
    post length=2pt
}] (root) -- (temp);
\path (temp);

\end{tikzpicture}
\end{center}
\end{example}

	The \emph{Schubitope} $\mathcal{S}_D$, as defined in \cite{MTY}, is the polytope
	\begin{align}\label{eqn:ogschubitope}
	\left\{(\alpha_1,\dots,\alpha_n) \in \mathbb{R}_{\geq 0}^n : \alpha_1 + \dots + \alpha_n = \#D \text{ and } \sum_{i \in S} \alpha_i \leq \theta_D(S) \text{ for all } S \subseteq [n]\right\}.
	\end{align}
	\subsection{Characterizations via tableaux}
	
	A \emph{tableau} of \emph{shape} $D$ is a map 
\[\tau : D \rightarrow [n] \cup \{\nolabel\},\] 
where $\tau(r,c) = \nolabel$ indicates that the box $(r,c)$ is unlabelled. Let ${\sf Tab}(D)$ denote the set of such tableaux.
	
	It will be useful to reformulate the original definition of $\theta_D(S)$ into the language of tableaux. Given $S \subseteq [n]$, define $\pi_{D,S} \in {\sf Tab}(D)$ by
	\begin{align}\label{eqn:pi_definition}
	\pi_{D,S}(r,c) =
	\begin{cases}
	r &\text{if } (r,c) \text{ contributes a ``$\star$''  to } {\sf word}_{c,S}(D), \\
	s &\mbox{if } (r,c) \text{ contributes a  ``)'' to } {\sf word}_{c,S}(D) \mbox{ which is} \\
	&\text{paired with an ``('' from } (s,c), \\
	\nolabel &\mbox{otherwise}.
	\end{cases}
	\end{align}
	In (\ref{eqn:pi_definition}) and throughout, we pair by the standard ``inside-out'' convention.
\begin{example} \label{ex:schubTab} Continuing Example \ref{ex:schubWord}, below is $\pi_{D,\{1,3\}}(D)$
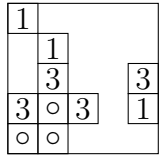
\[\begin{tikzpicture}[scale=.4]
\draw (0,0) rectangle (5,5);

\draw (0,4) rectangle (1,5)  node[pos=.5] {$1$};

\draw (1,3) rectangle (2,4)  node[pos=.5] {$1$};

\draw (1,2) rectangle (2,3)  node[pos=.5] {$3$};
\draw (4,2) rectangle (5,3)  node[pos=.5] {$3$};

\draw (0,1) rectangle (1,2)  node[pos=.5] {$3$};
\draw (1,1) rectangle (2,2)  node[pos=.5] {$\circ$};
\draw (2,1) rectangle (3,2)  node[pos=.5] {$3$};
\draw (4,1) rectangle (5,2)  node[pos=.5] {$1$};

\draw (1,0) rectangle (2,1)  node[pos=.5] {$\circ$};

\draw (0,0) rectangle (1,1)  node[pos=.5] {$\circ$};

\end{tikzpicture}\]
\end{example}	
	\begin{proposition}\label{prop:theta_equivalence}
		For all $D \subseteq [n]^2$ and $S \subseteq [n]$, we have $\theta_D(S) = \#\pi_{D,S}^{-1}(S)$.
	\end{proposition}
	
	\begin{proof} $\pi_{D,S}(r,c) \in S$ if and only if $(r,c)$ falls into one of the first two cases in \eqref{eqn:pi_definition}.
	\end{proof}
	
	Say $\tau \in {\sf Tab}(D)$ is \emph{flagged} if $\tau(r,c) \leq r$ whenever $\tau(r,c) \neq \nolabel$. It is \emph{column-injective} if $\tau(r,c) \neq \tau(r',c)$ whenever $r \neq r'$ and $\tau(r,c) \neq \nolabel$. 
Let ${\sf FCITab}(D) \subseteq {\sf Tab}(D)$ be the set of tableaux of shape $D$ which are flagged and column-injective.

\begin{example}
Of the tableaux of shape $D$ below, only the second and fourth are flagged, and only the third and fourth are column-injective.
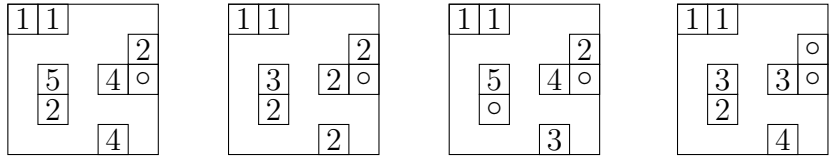
\begin{center}
\begin{tikzpicture}[scale=.4]

\draw (0,0) rectangle (5,5);
\draw (0,4) rectangle (1,5) node[pos=.5] {$1$};
\draw (1,4) rectangle (2,5) node[pos=.5] {$1$};
\draw (4,3) rectangle (5,4) node[pos=.5] {$2$};
\draw (1,2) rectangle (2,3) node[pos=.5] {$5$};
\draw (3,2) rectangle (4,3) node[pos=.5] {$4$};
\draw (4,2) rectangle (5,3) node[pos=.5] {$\circ$};

\draw (1,1) rectangle (2,2) node[pos=.5] {$2$};

\draw (3,0) rectangle (4,1) node[pos=.5] {$4$};
\end{tikzpicture}\qquad
\begin{tikzpicture}[scale=.4]

\draw (0,0) rectangle (5,5);
\draw (0,4) rectangle (1,5) node[pos=.5] {$1$};
\draw (1,4) rectangle (2,5) node[pos=.5] {$1$};
\draw (4,3) rectangle (5,4) node[pos=.5] {$2$};
\draw (1,2) rectangle (2,3) node[pos=.5] {$3$};
\draw (3,2) rectangle (4,3) node[pos=.5] {$2$};
\draw (4,2) rectangle (5,3) node[pos=.5] {$\circ$};
\draw (1,1) rectangle (2,2) node[pos=.5] {$2$};

\draw (3,0) rectangle (4,1) node[pos=.5] {$2$};

\end{tikzpicture}\qquad
\begin{tikzpicture}[scale=.4]

\draw (0,0) rectangle (5,5);

\draw (0,4) rectangle (1,5) node[pos=.5] {$1$};
\draw (1,4) rectangle (2,5) node[pos=.5] {$1$};

\draw (4,3) rectangle (5,4) node[pos=.5] {$2$};
\draw (1,2) rectangle (2,3) node[pos=.5] {$5$};
\draw (3,2) rectangle (4,3) node[pos=.5] {$4$};
\draw (4,2) rectangle (5,3) node[pos=.5] {$\circ$};
\draw (1,1) rectangle (2,2) node[pos=.5] {$\circ$};

\draw (3,0) rectangle (4,1) node[pos=.5] {$3$};

\end{tikzpicture}
\qquad
\begin{tikzpicture}[scale=.4]

\draw (0,0) rectangle (5,5);

\draw (0,4) rectangle (1,5) node[pos=.5] {$1$};
\draw (1,4) rectangle (2,5) node[pos=.5] {$1$};

\draw (4,3) rectangle (5,4) node[pos=.5] {$\circ$};

\draw (1,2) rectangle (2,3) node[pos=.5] {$3$};
\draw (3,2) rectangle (4,3) node[pos=.5] {$3$};
\draw (4,2) rectangle (5,3) node[pos=.5] {$\circ$};

\draw (1,1) rectangle (2,2) node[pos=.5] {$2$};

\draw (3,0) rectangle (4,1) node[pos=.5] {$4$};
\end{tikzpicture}
\end{center}
\end{example}	
	\begin{proposition}\label{prop:pi_in_FCITab}
		$\pi_{D,S} \in {\sf FCITab}(D)$ for all $D \subseteq [n]^2$ and $S \subseteq [n]$.
	\end{proposition}
	
	\begin{proof}
		This is immediate from \eqref{eqn:pi_definition}.
	\end{proof}
	
	A simple consequence of being flagged and column-injective is the following.
	
	\begin{proposition}\label{prop:FCI_consequence}
		Let $\tau \in {\sf FCITab}(D)$. Then for all $(r,c) \in [n]^2$ and $S\subseteq [n]$, we have
		\begin{align}\label{eqn:FCI_consequence}
		\#\{(i,c) \in \tau^{-1}(S) : i < r\} \leq \#\{i \in S : i \leq r\},
		\end{align}
		with strict inequality whenever $(r,c) \in \tau^{-1}(S)$.
	\end{proposition}
	
	\begin{proof}
		The map $(i,c) \mapsto \tau(i,c)$ from $\{(i,c) \in \tau^{-1}(S) : i \leq r\}$ to $\{i \in S : i \leq r\}$ is well-defined since $\tau$ is flagged. It is injective since $\tau$ is column-injective. Thus \eqref{eqn:FCI_consequence} holds, and
		\begin{align*}
		\#\{(i,c) \in \tau^{-1}(S) : i < r\} < \#\{(i,c) \in \tau^{-1}(S) : i \leq r\} \leq \#\{i \in S : i \leq r\}
		\end{align*}
		whenever $(r,c) \in \tau^{-1}(S)$, establishing the strict inequality assertion.
	\end{proof}
	
	In fact, a stronger assertion holds when $\tau = \pi_{D,S}$.
	
	\begin{proposition}\label{prop:pi_characterization}
		If $(r,c) \in D \subseteq [n]^2$ and $S \subseteq [n]$, then
		\begin{align*}
		(r,c) \in \pi_{D,S}^{-1}(S) \iff \#\{(i,c) \in \pi_{D,S}^{-1}(S) : i < r\} < \#\{i \in S : i \leq r\}.
		\end{align*}
	\end{proposition}
	
	\begin{proof}
		($\Rightarrow$) This direction follows from Propositions \ref{prop:pi_in_FCITab} and \ref{prop:FCI_consequence}.
		
		\noindent
		($\Leftarrow$) If $r \in S$, then $(r,c)$ contributes a ``$\star$'' to ${\sf word}_{c,S}(D)$, so $\pi_{D,S}(r,c) = r \in S$, as desired. 
Thus we assume $r\not\in S$. The hypothesis combined with this assumption says
		\begin{align*}
		\#\{(i,c) \in \pi_{D,S}^{-1}(S) : i < r\} < \#\{i \in S : i \leq r\} = \#\{i \in S : i < r\}.
		\end{align*}
		Thus, there is a maximal $s \in S$ with $s < r$ such that $\pi_{D,S}(r',c) \neq s$ whenever $r' < r$. If $(s,c) \in D$, then $(s,c)$ contributes a ``$\star$'' to ${\sf word}_{c,S}(D)$, so $\pi_{D,S}(s,c) = s$, contradicting our choice of $s$. Therefore, $(s,c)$ contributes an ``$($''
to ${\sf word}_{c,S}(D)$. If this ``$($'' is paired by a ``$)$'' contributed by $(r',c) \in D$ with $r' < r$, then $\pi_{D,S}(r',c) = s$, again a contradiction. Thus, this ``$($'' pairs the ``$)$'' from $(r,c)$, so $\pi_{D,S}(r,c) = s \in S$. Hence, $(r,c) \in \pi_{D,S}^{-1}(S)$ as desired.
	\end{proof}
	
	The previous two propositions combined assert that $\{(r,c)\in \pi^{-1}_{D,S}(S)\}$ is characterized by greedy selection as one moves down each column $c$. The next
	proposition shows that this greedy algorithm maximizes $\#\tau^{-1}(S)$ among all $\tau \in {\sf FCITab}(D)$. 
	
	\begin{proposition}\label{prop:pi_optimality}
		Let $D \subseteq [n]^2$ and $S \subseteq [n]$. Then $\#\pi_{D,S}^{-1}(S) \geq \#\tau^{-1}(S)$ for all $\tau \in {\sf FCITab}(D)$.
	\end{proposition}
	
	\begin{proof}
		If not, then there exist $\tau \in {\sf FCITab}(D)$ and $(r,c) \in [n]^2$ satisfying
		\begin{align*}
		\#\{(i,c) \in \pi_{D,S}^{-1}(S) : i \leq r\} < \#\{(i,c) \in \tau^{-1}(S) : i \leq r\}
		\end{align*}
		and we can choose these such that $r$ is minimized. Then because $r$ is minimal,
		\begin{align*}
		\#\{(i,c) \in \pi_{D,S}^{-1}(S) : i < r\} = \#\{(i,c) \in \tau^{-1}(S) : i < r\}
		\end{align*}
                and $(r,c) \in \tau^{-1}(S) \smallsetminus \pi_{D,S}^{-1}(S)$, so in particular $(r,c) \in D$.
		Thus Proposition \ref{prop:FCI_consequence} implies
		\begin{align*}
		\#\{(i,c) \in \pi_{D,S}^{-1}(S) : i < r\} = \#\{(i,c) \in \tau^{-1}(S) : i < r\} < \#\{i \in S : i \leq r\}.
		\end{align*}
		But then we must have $(r,c) \in \pi_{D,S}^{-1}(S)$ by Proposition \ref{prop:pi_characterization}, a contradiction.
	\end{proof}
	
	If $\tau$ has shape a subset of $[n]^2$ and $\alpha = (\alpha_1,\dots,\alpha_n) \in \mathbb{R}_{\geq 0}^n$, say $\tau$ \emph{exhausts} $\alpha$ \emph{over} $S$ if
	\begin{align*}
	\sum_{i \in S} \alpha_i \leq \#\tau^{-1}(S).
	\end{align*}

\begin{example}\label{ex:exhaust}
Only the left tableau below exhausts $\alpha=(3,2,2,4)$ over $S=\{1,3\}$.

\begin{center}
\begin{tikzpicture}[scale=.4]
\draw (0,0) rectangle (5,5);

\draw (0,4) rectangle (1,5)  node[pos=.5] {$1$};

\draw (1,3) rectangle (2,4)  node[pos=.5] {$1$};
\draw (2,3) rectangle (3,4)  node[pos=.5] {$1$};

\draw (2,2) rectangle (3,3)  node[pos=.5] {$3$};

\draw (2,1) rectangle (3,2)  node[pos=.5] {$4$};
\draw (4,1) rectangle (5,2)  node[pos=.5] {$4$};

\draw (0,0) rectangle (1,1)  node[pos=.5] {$4$};
\draw (2,0) rectangle (3,1)  node[pos=.5] {$\circ$};
\draw (3,0) rectangle (4,1)  node[pos=.5] {$4$};
\draw (4,0) rectangle (5,1)  node[pos=.5] {$3$};
\end{tikzpicture}\qquad
\begin{tikzpicture}[scale=.4]

\draw (0,0) rectangle (5,5);

\draw (0,4) rectangle (1,5)  node[pos=.5] {$1$};

\draw (1,3) rectangle (2,4)  node[pos=.5] {$1$};
\draw (2,3) rectangle (3,4)  node[pos=.5] {$2$};

\draw (2,2) rectangle (3,3)  node[pos=.5] {$3$};

\draw (2,1) rectangle (3,2)  node[pos=.5] {$4$};
\draw (4,1) rectangle (5,2)  node[pos=.5] {$4$};

\draw (0,0) rectangle (1,1)  node[pos=.5] {$4$};
\draw (2,0) rectangle (3,1)  node[pos=.5] {$\circ$};
\draw (3,0) rectangle (4,1)  node[pos=.5] {$4$};
\draw (4,0) rectangle (5,1)  node[pos=.5] {$2$};
\end{tikzpicture}
\end{center}
\end{example}

	\begin{theorem}\label{thm:dependent_characterization}
		Let $D \subseteq [n]^2$ and $\alpha = (\alpha_1,\dots,\alpha_n) \in \mathbb{Z}_{\geq 0}^n$ with $\alpha_1 + \dots + \alpha_n = \#D$. Then $\alpha \in \mathcal{S}_D$ if and only if for each $S \subseteq [n]$, there exists $\tau_{D,S} \in {\sf FCITab}(D)$ which exhausts $\alpha$ over $S$.
	\end{theorem}

	\begin{proof}[Proof of Theorem \ref{thm:dependent_characterization}.]
		$(\Rightarrow)$ The inequalities in \eqref{eqn:ogschubitope} combined with Proposition \ref{prop:theta_equivalence} imply
		\begin{align*}
		\sum_{i \in S} \alpha_i \leq \theta_D(S) = \#\pi_{D,S}^{-1}(S).
		\end{align*}
		Thus, $\tau_{D,S} := \pi_{D,S}$ exhausts $\alpha$ over $S$.
		
		\noindent
		$(\Leftarrow)$ By Propositions \ref{prop:pi_optimality} and \ref{prop:theta_equivalence},
		\begin{align*}
		\sum_{i \in S} \alpha_i \leq \#\tau_{D,S}^{-1}(S) \leq \#\pi_{D,S}^{-1}(S) = \theta_D(S),
		\end{align*}
		so the inequalities in \eqref{eqn:ogschubitope} hold.
	\end{proof}

	\begin{remark}\label{rmk:tau=pi}
		The proof of ($\Rightarrow$) shows that we can take $\tau_{D,S} = \pi_{D,S}$ in Theorem \ref{thm:dependent_characterization}.
	\end{remark}
	
	It would be nice if $\tau_{D,S}$ did not depend on $S$, i.e., if some $\tau_D$ exhausted $\alpha$ over all $S \subseteq [n]$, so we could take $\tau_{D,S} = \tau_D$ in Theorem \ref{thm:dependent_characterization}. Indeed, this is shown in Theorem \ref{thm:independent_characterization}.
	
	Say $\tau \in {\sf Tab}(D)$ has \emph{content} $\alpha$ if $\#\tau^{-1}(\{i\}) = \alpha_i$ for each $i \in [n]$. Let ${\sf Tab}(D,\alpha)$ and ${\sf FCITab}(D,\alpha)$ be the subsets of ${\sf Tab}(D)$ and ${\sf FCITab}(D)$, respectively, of those tableaux which have content $\alpha$. In addition, call a tableau $\tau \in {\sf Tab}(D)$ \emph{perfect} if $\tau \in {\sf FCITab}(D)$, and if no boxes are left unlabelled, i.e., $\tau^{-1}(\{\nolabel\}) = \emptyset$. Thus, the set of perfect tableaux of content $\alpha$ is precisely ${\sf PerfectTab}(D,\alpha) \subseteq {\sf FCITab}(D,\alpha)$ introduced in Section \ref{sec:intro}.
	
	\begin{proposition}\label{prop:proper_content}
		Let $D \subseteq [n]^2$ and $\alpha = (\alpha_1,\dots,\alpha_n) \in \mathbb{Z}_{\geq 0}^n$. Then ${\sf PerfectTab}(D,\alpha) \neq \emptyset$ if and only if $\alpha_1 + \dots + \alpha_n = \#D$ and ${\sf FCITab}(D,\alpha) \neq \emptyset$.
	\end{proposition}
	
	\begin{proof}
		($\Rightarrow$) Let $\tau \in {\sf PerfectTab}(D,\alpha)$. Then $\tau \in {\sf FCITab}(D,\alpha)$, and since $\tau$ has content $\alpha$ and satisfies $\tau^{-1}(\{\nolabel\}) = \emptyset$,
		\begin{align*}
		\alpha_1 + \dots + \alpha_n = \#\tau^{-1}(\{1\}) + \dots + \#\tau^{-1}(\{n\}) = \#D.
		\end{align*}
		
		\noindent
		($\Leftarrow$) Let $\tau \in {\sf FCITab}(D,\alpha)$. Then since $\tau$ has content $\alpha$,
		\begin{align*}
		\#\tau^{-1}(\{\nolabel\}) = \#D - \#\tau^{-1}(\{1\}) - \dots - \#\tau^{-1}(\{n\}) = \#D - \alpha_1 - \dots - \alpha_n = 0.
		\end{align*}
		Thus, $\tau \in {\sf PerfectTab}(D,\alpha)$.
	\end{proof}
	
	\begin{theorem}\label{thm:independent_characterization}
		Let $D \subseteq [n]^2$ and $\alpha = (\alpha_1,\dots,\alpha_n) \in \mathbb{Z}_{\geq 0}^n$. Then $\alpha \in \mathcal{S}_D$ if and only if ${\sf PerfectTab}(D,\alpha) \neq \emptyset$.
	\end{theorem}
	
	The proof will require a lemma regarding tableaux of the form $\tau = \pi_{D,S}$.
	
	\begin{lemma}\label{lemma:pi_union}
		Let $D \subseteq [n]^2$, and $S,T \subseteq [n]$ be disjoint. Set 
		\[\tilde{D} = D \smallsetminus \pi_{D,S}^{-1}(S) \text{\ and \ $U = S \cup T$.}\] 
		Then
		\begin{align*}
		\pi_{D,U}^{-1}(U) = \pi_{D,S}^{-1}(S) \cup \pi_{\tilde{D},T}^{-1}(T).
		\end{align*}
	\end{lemma}
	
	\begin{proof}
		Let $(r,c) \in D$, and assume by induction on $r$ that
		\begin{align}\label{eqn:partial_characterization}
		(i,c) \in \pi_{D,U}^{-1}(U) \iff (i,c) \in \pi_{D,S}^{-1}(S) \cup \pi_{\tilde{D},T}^{-1}(T) 
		\end{align}
		whenever $i < r$. This clearly holds in the base case $r=1$. By Proposition \ref{prop:pi_characterization}, $(r,c) \in \pi_{D,U}^{-1}(U)$ if and only if
		\begin{align}\label{eqn:union_characterization}
		\#\{(i,c) \in \pi_{D,U}^{-1}(U) : i < r\} < \#\{i \in U : i \leq r\}.
		\end{align}
		By \eqref{eqn:partial_characterization} and the fact that 
		\[\pi_{D,S}^{-1}(S) \cap \tilde{D} = \emptyset = S \cap T,\] 
		\eqref{eqn:union_characterization} is equivalent to
		\begin{align*}
		\#\{(i,c) \in \pi_{D,S}^{-1}(S) : i < r\} + \#\{(i,c) \in \pi_{\tilde{D},T}^{-1}(T) : i < r\} < \#\{i \in S : i \leq r\} + \#\{i \in T : i \leq r\}.
		\end{align*}
		By applying Proposition \ref{prop:FCI_consequence} twice, we see that this holds if and only if at least one of (i) and (ii) below hold.
		\begin{enumerate}
			\item[(i)] $\#\{(i,c) \in \pi_{D,S}^{-1}(S) : i < r\} < \#\{i \in S : i \leq r\}$
			\item[(ii)] $\#\{(i,c) \in \pi_{\tilde{D},T}^{-1}(T) : i < r\} < \#\{i \in T : i \leq r\}$
		\end{enumerate}
		By Proposition \ref{prop:pi_characterization}, (i) is equivalent to $(r,c) \in \pi_{D,S}^{-1}(S)$.
If indeed $(r,c) \in \pi_{D,S}^{-1}(S)$ holds, then our induction step is complete. Otherwise, 
$(r,c) \not\in \pi_{D,S}^{-1}(S)$, so by definition, $(r,c)\in \tilde{D}$. Thus, applying Proposition~\ref{prop:pi_characterization}
to $\tilde{D}$, $T\subseteq [n]$ and $(r,c)\in \tilde{D}$, 
(ii) is equivalent to $(r,c) \in \pi_{\tilde{D},T}^{-1}(T)$. Hence, \eqref{eqn:partial_characterization} holds for all $i \leq r$.
	\end{proof}
	
	\begin{corollary}\label{cor:pi_subset}
		Let $D \subseteq [n]^2$ and $S \subseteq U \subseteq [n]$. Then $\pi_{D,S}^{-1}(S) \subseteq \pi_{D,U}^{-1}(U)$.
	\end{corollary}
	
	\begin{proof}
		Take $T = U \smallsetminus S$ in Lemma \ref{lemma:pi_union}.
	\end{proof}
	
	Finally, we are ready to prove Theorem \ref{thm:independent_characterization}.
	
	\begin{proof}[Proof of Theorem \ref{thm:independent_characterization}]
		($\Leftarrow$) Let $\tau_D \in {\sf PerfectTab}(D,\alpha)$. Then $\alpha_1 + \dots + \alpha_n = \#D$ by Proposition \ref{prop:proper_content}. Also, for each $S \subseteq [n]$,
		\begin{align*}
		\sum_{i \in S} \alpha_i = \sum_{i \in S} \#\tau_D^{-1}(\{i\}) = \#\tau_D^{-1}(S),
		\end{align*}
		so $\tau_D$ exhausts $\alpha$ over $S$. Thus, $\alpha \in \mathcal{S}_D$ by Theorem \ref{thm:dependent_characterization}.
		
		\noindent
		($\Rightarrow$) We induct on the sum of the row indices of each box in $D$, i.e., $\sum_{(i,j) \in D} i$. The base case of an empty diagram is trivial, so we may assume $D \neq \emptyset$. Then since $\alpha \in \mathcal{S}_D$, \eqref{eqn:ogschubitope} implies $\alpha_1 + \dots + \alpha_n = \#D > 0$, so we can choose $m$ maximal such that $\alpha_m > 0$.
		
		\noindent
		{\sf Case 1:} ($D$ contains boxes below row $m$). Pick $(r,c) \in D$ below row $m$ (so $r > m$).
		
		\begin{claim}\label{claim:diagram_gap}
			There exists $r_1 < r$ such that $(r_1,c) \not\in D$.
		\end{claim}
		
		\begin{proof}[Proof of Claim \ref{claim:diagram_gap}.]
			By Theorem \ref{thm:dependent_characterization}, there exists $\tau_{D,[m]} \in {\sf FCITab}(D)$ such that
			\begin{align}\label{eqn:image_in_[m]}
			\#\tau_{D,[m]}^{-1}([m]) \geq \alpha_1 + \dots + \alpha_m = \alpha_1 + \dots + \alpha_n = \#D.
			\end{align}
			Thus, $\tau_{D,[m]}(D) \subseteq [m]$. Consequently, by column-injectivity of $\tau_{D,[m]}$,
there can be at most $m$ boxes in each column of $D$. Since $(r,c) \in D$ with $r > m$, there are more than $m$ boxes in column $c$ 
if $(r_1,c)\in D$ for all $r_1 < r$. Hence there must be some $r_1 < r$ for which $(r_1,c) \not\in D$, as asserted.
		\end{proof}
		
		By Claim \ref{claim:diagram_gap}, we can choose $r_1 < r$ maximal such that $(r_1,c) \not\in D$. Let
		\begin{align*}
		\tilde{D} = (D \smallsetminus \{(r,c)\}) \cup \{(r_1,c)\}.
		\end{align*}
		
		\begin{claim}\label{claim:shifted_box}
			$\alpha \in \mathcal{S}_{\tilde{D}}$.
		\end{claim}
		
		\begin{proof}[Proof of Claim \ref{claim:shifted_box}.]
			Since $\alpha \in \mathcal{S}_D$, $(r,c) \in D$, and $(r_1,c) \not\in D$, we have
			\begin{align*}
			\alpha_1 + \dots + \alpha_n = \#D = \#\tilde{D}.
			\end{align*}
			Let $S \subseteq [n]$ and $T = S \cap [m]$. Then define $\tau_{\tilde{D},S} \in {\sf Tab}(\tilde{D})$ by
			\begin{align*}
			\tau_{\tilde{D},S}(i,j) =
			\begin{cases}
			\pi_{D,T}(r,c) &\text{ if } (i,j) = (r_1,c), \\
			\pi_{D,T}(i,j) &\text{ otherwise}.
			\end{cases}
			\end{align*}

			If $\pi_{D,T}(r,c) = \nolabel$, then certainly $\tau_{\tilde{D},S} \in {\sf FCITab}(\tilde{D})$. 
Otherwise, let $s = \pi_{D,T}(r,c)$. Since $(r,c) \in D$ but $r \not\in T$, $(r,c)$ contributes a ``$)$'' to ${\sf word}_{c,S}(D)$. Thus, by \eqref{eqn:pi_definition}, $(s,c)$ contributes an 
``$($'', so in particular $(s,c) \not\in D$. From our choice of $r_1$, we must therefore have $s \leq r_1$, so $\tau_{\tilde{D},S}$ is flagged. Hence, $\tau_{\tilde{D},S} \in {\sf FCITab}(\tilde{D})$. 

By construction, 
\[\#\tau_{\tilde{D},S}^{-1}(\{i\}) = \#\pi_{D,T}^{-1}(\{i\})\] 
for each $i \in [n]$, so $\tau_{\tilde{D},S}$ exhausts $\alpha$ over $T$ by Theorem \ref{thm:dependent_characterization} and in particular Remark \ref{rmk:tau=pi}. Since $\alpha_i = 0$ for all $i > m$, we can write
			\begin{align*}
			\sum_{i \in S} \alpha_i = \sum_{i \in T} \alpha_i \leq \#\tau_{\tilde{D},S}^{-1}(T) \leq \#\tau_{\tilde{D},S}^{-1}(S).
			\end{align*}
			Therefore, $\tau_{\tilde{D},S}\in  {\sf FCITab}(\tilde{D})$ exhausts $\alpha$ over $S$, so $\alpha \in \mathcal{S}_{\tilde{D}}$ by Theorem \ref{thm:dependent_characterization}.
		\end{proof}
		
		Since $r_1 < r$,
		\begin{align*}
		\sum_{(i,j) \in \tilde{D}} i < \sum_{(i,j) \in D} i.
		\end{align*}
		Thus, Claim \ref{claim:shifted_box} and induction yields $\tau_{\tilde{D}} \in {\sf PerfectTab}(\tilde{D},\alpha)$. Define $\tau_D \in {\sf Tab}(D)$ by
		\begin{align*}
		\tau_D(i,j) =
		\begin{cases}
		\tau_{\tilde{D}}(r_1,c) &\text{if } (i,j) = (r,c), \\
		\tau_{\tilde{D}}(i,j) &\text{otherwise}.
		\end{cases}
		\end{align*}
		Then it is easy to check that $\tau_D \in {\sf PerfectTab}(D,\alpha)$, so {\sf Case 1} is complete.
		
		\noindent
		{\sf Case 2:} ($D$ does not contain boxes below row $m$). We say an inequality
                $
                \sum_{i \in S}\alpha_i \leq \theta_D(S)
                $
                from \eqref{eqn:ogschubitope} is \emph{nontrivial} if
		\begin{align}\label{eqn:nontrivSchub}
		\sum_{i \in S} \alpha_i > 0
		\qquad \text{and} \qquad
		\theta_D(S) < \#D.
		\end{align}

		\noindent
		{\sf Case 2a:} (All nontrivial inequalities from \eqref{eqn:ogschubitope} are strict). 
		Thus if \eqref{eqn:nontrivSchub} holds, then
		\begin{align}\label{eqn:schubitope_slack}
		\sum_{i \in S} \alpha_i < \theta_D(S).
		\end{align}
		
		\begin{claim}\label{claim:box_in_row_m}
			There exists $c \in [n]$ such that $(m,c) \in D$.
		\end{claim}
		
		\begin{proof}[Proof of Claim \ref{claim:box_in_row_m}.]
			By Theorem \ref{thm:dependent_characterization}, there exists some $\tau_{D,\{m\}} \in {\sf FCITab}(D)$ which exhausts $\alpha$ over $\{m\}$. Then 
			\[\#\tau_{D,\{m\}}^{-1}(\{m\}) \geq \alpha_m > 0,\] 
			so $\tau_{D,\{m\}}(r,c) = m$ for some $(r,c) \in D$. Since $\tau_{D,\{m\}}$ is flagged, we must have $r \geq m$. But by the assumption of {\sf Case 2}, there are no boxes below row $m$, so $r = m$.
		\end{proof}
		
		Pick $c \in [n]$ as in Claim \ref{claim:box_in_row_m}. Then let $\tilde{D} = D \smallsetminus \{(m,c)\}$ and $\tilde{\alpha} = (\tilde{\alpha}_1,\dots,\tilde{\alpha}_n) := (\alpha_1,\dots,\alpha_{m-1},\alpha_m-1,0,\dots,0)$.
		
		\begin{claim}\label{claim:deleted_box}
			$\tilde{\alpha} \in \mathcal{S}_{\tilde{D}}$.
		\end{claim}
		
		\begin{proof}[Proof of Claim \ref{claim:deleted_box}.]
			Since $\alpha_i = 0$ for all $i > m$, and $(m,c) \in D$, we have
			\begin{align}\label{eqn:deleted_content}
			\tilde{\alpha}_1 + \dots + \tilde{\alpha}_n = \alpha_1 + \dots + \alpha_n - 1 = \#D - 1 = \#\tilde{D}.
			\end{align}
			For each $S \subseteq [n]$, let 
			\[\tau_{\tilde{D},S} = \pi_{D,S}|_{\tilde{D}} \in {\sf FCITab}(\tilde{D})\] 
			be the restriction of $\pi_{D,S}$ to $\tilde{D}$. Then by Proposition \ref{prop:theta_equivalence},
			\begin{align}\label{eqn:theta_bound}
			\#\tau_{\tilde{D},S}^{-1}(S) \geq \#\pi_{D,S}^{-1}(S) - 1 = \theta_D(S) - 1.
			\end{align}
			If $\sum_{i \in S} \alpha_i = 0$, then
			\begin{align*}
			\sum_{i \in S} \tilde{\alpha}_i = 0 \leq \#\tau_{\tilde{D},S}^{-1}(S).
			\end{align*}
			If $\theta_D(S) = \#D$, then by \eqref{eqn:deleted_content} and \eqref{eqn:theta_bound},
			\begin{align*}
			\sum_{i \in S} \tilde{\alpha}_i \leq \tilde{\alpha}_1 + \dots + \tilde{\alpha}_n = \#D - 1 = \theta_D(S) - 1 \leq \#\tau_{\tilde{D},S}^{-1}(S).
			\end{align*}
			Finally, if $\sum_{i \in S} \alpha_i > 0$ and $\theta_D(S) < \#D$, then \eqref{eqn:schubitope_slack} must hold, so by \eqref{eqn:schubitope_slack} and \eqref{eqn:theta_bound},
			\begin{align*}
			\sum_{i \in S} \tilde{\alpha}_i \leq \sum_{i \in S} \alpha_i \leq \theta_D(S) - 1 \leq \#\tau_{\tilde{D},S}^{-1}(S).
			\end{align*}
			In all three cases, $\tau_{\tilde{D},S}$ exhausts $\tilde{\alpha}$ over $S$, so $\tilde{\alpha} \in \mathcal{S}_{\tilde{D}}$ by Theorem \ref{thm:dependent_characterization}.
		\end{proof}
		
		By construction,
		\begin{align*}
		\sum_{(i,j) \in \tilde{D}} i < \sum_{(i,j) \in D} i.
		\end{align*}
		Thus, Claim \ref{claim:deleted_box} and induction yield $\tau_{\tilde{D}} \in {\sf PerfectTab}(\tilde{D},\tilde{\alpha})$. Define $\tau_D \in {\sf Tab}(D)$ by
		\begin{align*}
		\tau_D(i,j) =
		\begin{cases}
		m &\text{if } (i,j) = (m,c), \\
		\tilde{\tau}(i,j) &\text{otherwise}.
		\end{cases}
		\end{align*}
		Clearly, $\tau_D$ is flagged, has content $\alpha$, and satisfies $\tau_D^{-1}(\{\nolabel\}) = \emptyset$. The only potential obstruction to column-injectivity is that there could be some $r \neq m$ for which $\tau_D(r,c) = m$. This is impossible, since $\tau_D$ is flagged, so such an $r$ must be greater than $m$, but by the assumption of {\sf Case 2} there are no boxes below row $m$. Thus, $\tau_D \in {\sf PerfectTab}(D,\alpha)$, so {\sf Case 2a} is complete.
		
		\noindent
		{\sf Case 2b:} (There exists a tight, nontrivial inequality in \eqref{eqn:ogschubitope}). Thus, there exists $A \subseteq [n]$ satisfying
		\begin{align}\label{eqn:tight_inequality}
		0 < \sum_{i \in A} \alpha_i = \theta_D(A) < \#D.
		\end{align}
		Let $D^{(1)} = \pi_{D,A}^{-1}(A)$ and $D^{(2)} = D \smallsetminus D^{(1)}$. Then for each $i \in [n]$, set
		\begin{align*}
		\alpha^{(1)}_i =
		\begin{cases}
		\alpha_i &\text{if } i \in A, \\
		0 &\text{if } i \not\in A
		\end{cases}
		\qquad \text{and} \qquad
		\alpha^{(2)}_i =
		\begin{cases}
		\alpha_i &\text{if } i \not\in A, \\
		0 &\text{if } i \in A.
		\end{cases}
		\end{align*}
		
		\begin{claim}\label{claim:decomposition1}
			$\alpha^{(1)} := (\alpha^{(1)}_1,\dots,\alpha^{(1)}_n) \in \mathcal{S}_{D^{(1)}}$.
		\end{claim}
		
		\begin{proof}[Proof of Claim \ref{claim:decomposition1}.]
			By \eqref{eqn:tight_inequality} and Proposition \ref{prop:theta_equivalence}, we have
			\begin{align*}
			\alpha^{(1)}_1 + \dots + \alpha^{(1)}_n = \sum_{i \in A} \alpha_i = \theta_D(A) = \#\pi_{D,A}^{-1}(A) = \#D^{(1)}.
			\end{align*}
			Let $S \subseteq [n]$ and $T = S \cap A$. Then set 
			\[\tau_{D^{(1)},S} = \pi_{D,T}|_{D^{(1)}} \in {\sf FCITab}(D^{(1)}).\] 
			By Corollary \ref{cor:pi_subset}, 
			$\pi_{D,T}^{-1}(T) \subseteq D^{(1)}$, 
			so $\tau_{D^{(1)},S}^{-1}(T) = \pi_{D,T}^{-1}(T)$. 
			Thus, by 
Remark \ref{rmk:tau=pi}, $\tau_{D^{(1)},S}$ exhausts $\alpha$ over $T$. Hence,
			\begin{align*}
			\sum_{i \in S} \alpha^{(1)}_i = \sum_{i \in T} \alpha_i \leq \#\tau_{D^{(1)},S}^{-1}(T) \leq \#\tau_{D^{(1)},S}^{-1}(S),
			\end{align*}
			so $\tau_{D^{(1)},S}$ exhausts $\alpha^{(1)}$ over $S$, and consequently $\alpha^{(1)} \in \mathcal{S}_{D^{(1)}}$ by Theorem \ref{thm:dependent_characterization}.
		\end{proof}
		
		\begin{claim}\label{claim:decomposition2}
			$\alpha^{(2)} := (\alpha^{(2)}_1,\dots,\alpha^{(2)}_n) \in \mathcal{S}_{D^{(2)}}$.
		\end{claim}
		
		\begin{proof}[Proof of Claim \ref{claim:decomposition2}.]
			By \eqref{eqn:tight_inequality} and Proposition \ref{prop:theta_equivalence},
			\begin{align*}
			\alpha^{(2)}_1 + \dots + \alpha^{(2)}_n = \alpha_1 + \dots + \alpha_n - \sum_{i \in A} \alpha_i = \#D - \theta_D(A) = \#D - \#\pi_{D,A}^{-1}(A) = \#D^{(2)}.
			\end{align*}
			Let $S \subseteq [n]$, $T = S \smallsetminus A$, and $U = A \cup T$. Then by Theorem \ref{thm:dependent_characterization}, Remark \ref{rmk:tau=pi}, \eqref{eqn:tight_inequality}, Proposition \ref{prop:theta_equivalence}, and Lemma \ref{lemma:pi_union}, we can write
			\begin{align*}
			\sum_{i \in S} \alpha^{(2)}_i = \sum_{i \in U} \alpha_i - \sum_{i \in A} \alpha_i &\leq \#\pi_{D,U}^{-1}(U) - \theta_D(A)
			\\&= \#\pi_{D,U}^{-1}(U) - \#\pi_{D,A}^{-1}(A) = \#\pi_{D^{(2)},T}^{-1}(T) \leq \#\pi_{D^{(2)},T}^{-1}(S).
			\end{align*}
			Thus, $\tau_{D^{(2)},S} := \pi_{D^{(2)},T}$ exhausts $\alpha^{(2)}$ over $S$, so $\alpha^{(2)} \in \mathcal{S}_{D^{(2)}}$ by Theorem \ref{thm:dependent_characterization}.
		\end{proof}
		
		By \eqref{eqn:tight_inequality} and Proposition \ref{prop:theta_equivalence}, we have 
		\[0 < \#\pi_{D,A}^{-1}(A) < \#D,\] 
		so $D^{(1)},D^{(2)} \subsetneq D$. Thus, by Claims \ref{claim:decomposition1} and \ref{claim:decomposition2} and induction, there exist 
		\[\tau_{D^{(1)}} \in {\sf PerfectTab}(D^{(1)},\alpha^{(1)}) \text{\ and \ $\tau_{D^{(2)}} \in {\sf PerfectTab}(D^{(2)},\alpha^{(2)})$}.\] 
		Define $\tau_D = \tau_{D^{(1)}} \cup \tau_{D^{(2)}} \in {\sf Tab}(D)$ by
		\begin{align*}
		\tau_D(i,j) =
		\begin{cases}
		\tau_{D^{(1)}}(i,j) &\text{if } (i,j) \in D^{(1)}, \\
		\tau_{D^{(2)}}(i,j) &\text{if } (i,j) \in D^{(2)}.
		\end{cases}
		\end{align*}
		Clearly $\tau_D$ is flagged and satisfies $\tau_D^{-1}(\{\nolabel\}) = \emptyset$. It has content $\alpha$ because $\alpha = \alpha^{(1)} + \alpha^{(2)}$, and it is column-injective because the images of $\tau_{D^{(1)}}$ and $\tau_{D^{(2)}}$ are disjoint. Therefore, $\tau_D \in {\sf PerfectTab}(D,\alpha)$ and {\sf Case 2b} is complete.
		
		This completes the proof of Theorem \ref{thm:independent_characterization}. 
	\end{proof}
	
	\subsection{Polytopal descriptions of perfect tableaux}\label{subsctn:polytopes}
	
	Given $D \subseteq [n]^2$ and $\alpha = (\alpha_1,\dots,\alpha_n) \in \mathbb{Z}_{\geq 0}^n$, define the \polyName
	\[\mathcal{P}(D,\alpha) \subseteq \mathbb{R}^{n^2}\] 
	to be the polytope with points of the form
	$
	(\alpha_{ij})_{i,j\in[n]} = (\alpha_{11},\dots,\alpha_{n1},\dots,\alpha_{1n},\dots,\alpha_{nn})
	$
	governed by the inequalities (A)-(C) below.
	\begin{enumerate}
		\item[(A)] Column-Injectivity Conditions: For all $i,j \in [n]$,
		\begin{align*}
		0 \leq \alpha_{ij} \leq 1.
		\end{align*}
		
		\item[(B)] Content Conditions: For all $i \in [n]$,
		\begin{align*}
		\sum_{j=1}^n \alpha_{ij} = \alpha_i.
		\end{align*}
		
		\item[(C)] Flag Conditions: For all $s,j \in [n]$,
		\begin{align*}
		\sum_{i=1}^s \alpha_{ij} \geq \#\{(i,j) \in D : i \leq s\}.
		\end{align*}
	\end{enumerate}
	
	\begin{proposition}\label{prop:column_content}
		Let $D \subseteq [n]^2$ and $\alpha = (\alpha_1,\dots,\alpha_n) \in \mathbb{Z}^n_{\geq 0}$ with $\alpha_1 + \dots + \alpha_n = \#D$. If $(\alpha_{ij}) \in \mathcal{P}(D,\alpha)$, then for each $j \in [n]$, we have
		\begin{align*}
		\sum_{i=1}^{n} \alpha_{ij} = \#\{(i,j) \in D : i \in [n]\}.
		\end{align*}
	\end{proposition}
	
	\begin{proof}
		From the flag conditions (C) where $s=n$, we have that
		\begin{align*}
		\sum_{i=1}^{n} \alpha_{ij} \geq \#\{(i,j) \in D : i \in [n]\}.
		\end{align*}
		If this inequality is strict for any $j$, then using the content conditions (B), we can write
		\begin{align*}
		\#D = \alpha_1 + \dots + \alpha_n = \sum_{i=1}^{n} \sum_{j=1}^{n} \alpha_{ij} = \sum_{j=1}^{n} \sum_{i=1}^{n} \alpha_{ij} > \sum_{j=1}^{n} \#\{(i,j) \in D : i \in [n]\} = \#D,
		\end{align*}
		a contradiction.
	\end{proof}
	
	\begin{theorem}\label{thm:int_pt_characterization}
		Let $D \subseteq [n]^2$ and $\alpha = (\alpha_1,\dots,\alpha_n) \in \mathbb{Z}_{\geq 0}^n$. Then ${\sf PerfectTab}(D,\alpha) \neq \emptyset$ if and only if $\alpha_1 + \dots + \alpha_n = \#D$ and $\mathcal{P}(D,\alpha) \cap \mathbb{Z}^{n^2} \neq \emptyset$.
	\end{theorem}
	
	\begin{proof}
		($\Rightarrow$) By Proposition \ref{prop:proper_content}, we have $\alpha_1 + \dots + \alpha_n = \#D$.
		Let $\tau \in {\sf PerfectTab}(D,\alpha)$. Then for each $i,j \in [n]$, set
		\begin{align*}
		\alpha_{ij} = \#\{r \in [n] : \tau(r,j) = i\} =
		\begin{cases}
		1 &\text{if } \tau(r,j) = i \text{ for some } r \in [n], \\
		0 &\text{otherwise},
		\end{cases}
		\end{align*}
		where the second equality follows from the fact that $\tau$ is column-injective.
		
		\begin{claim}\label{claim:tableau_to_int_pt}
			$(\alpha_{ij}) \in \mathcal{P}(D,\alpha) \cap \mathbb{Z}^{n^2}$.
		\end{claim}
		
		\begin{proof}[Proof of Claim \ref{claim:tableau_to_int_pt}.]
			Clearly $(\alpha_{ij}) \in \mathbb{Z}^{n^2}$ and the column-injectivity conditions (A) hold. Since $\tau$ has content $\alpha$,
			\begin{align*}
			\sum_{j=1}^{n} \alpha_{ij} = \sum_{j=1}^{n} \#\{r \in [n] : \tau(r,j) = i\} = \#\tau^{-1}(\{i\}) = \alpha_i
			\end{align*}
			for each $i \in [n]$, so the content conditions (B) hold. Finally, for each $s,j \in [n]$, we have
			\begin{align*}
			\sum_{i=1}^{s} \alpha_{ij} = \#\{r \in [n] : \tau(r,j) \leq s\} \geq \#\{(r,j) \in D : r \leq s\}
			\end{align*}
			since $\tau$ is flagged. Thus, the flag conditions (C) also hold.
		\end{proof}
		
		\noindent
		($\Leftarrow$) Let $(\alpha_{ij}) \in \mathcal{P}(D,\alpha) \cap \mathbb{Z}^{n^2}$. By the column-injectivity conditions (A), $\alpha_{ij} \in \{0,1\}$. Thus, by Proposition \ref{prop:column_content}, there exists for each $j \in [n]$ a bijection
		\begin{align*}
		\varphi_j : \{i \in [n] : (i,j) \in D\} \rightarrow \{i \in [n] : \alpha_{ij} = 1\}
		\end{align*}
		that is order-preserving, i.e., $\varphi_j$ satisfies $\varphi_j(i) < \varphi_j(i')$ whenever $i < i'$. Define $\tau \in {\sf Tab}(D)$ by $\tau(i,j) = \varphi_j(i)$.
		
		\begin{claim}\label{claim:increasing_tableau}
			$\tau \in {\sf PerfectTab}(D,\alpha)$.
		\end{claim}
		
		\begin{proof}[Proof of Claim \ref{claim:increasing_tableau}]
			By construction, $\tau^{-1}(\{\nolabel\}) = \emptyset$. Since $\varphi_j$ is injective and order-preserving, $\tau$ is strictly increasing along columns, hence column-injective. For each $i \in [n]$, the content conditions (B) imply
			\begin{align*}
			\tau^{-1}(\{i\}) = \sum_{j=1}^{n} \#\varphi_j^{-1}(\{i\}) = \sum_{j=1}^{n} \alpha_{ij} = \alpha_i,
			\end{align*}
			so $\tau$ has content $\alpha$. Finally, the flag conditions (C) show that for each $s,j \in [n]$,
			\begin{align*}
			\#\{i \leq s : (i,j) \in D\} \leq \sum_{i=1}^{s} \alpha_{ij} = \#\{i \leq s : \alpha_{ij} = 1\},
			\end{align*}
			so $\varphi_j(i) \leq i$ for each $(i,j) \in D$ since $\varphi_j$ is order-preserving. Thus, 
			$\tau(i,j) = \varphi_j(i) \leq i$ 
			and $\tau$ is flagged. Hence, $\tau \in {\sf PerfectTab}(D,\alpha)$.
		\end{proof}
		
		This shows that ${\sf PerfectTab}(D,\alpha) \neq \emptyset$ and completes the proof of the theorem.
	\end{proof}
	
	\begin{remark}\label{rmk:injective_iff_increasing}
		The proof of Claim \ref{claim:increasing_tableau} shows that if ${\sf PerfectTab}(D,\alpha) \neq \emptyset$, then we can find $\tau \in {\sf PerfectTab}(D,\alpha)$ which is not only column-injective, but also strictly increasing along columns, so $\tau(i,j) < \tau(i',j)$ whenever $i < i'$. Thus ${\sf PerfectTab}(D,\alpha) \neq \emptyset$ if and only if ${\sf PerfectTab}(D,\alpha)_{\downarrow} \neq \emptyset$.
	\end{remark}
	
	Theorem \ref{thm:int_pt_characterization} formulates the problem of determining if
	${\sf PerfectTab}(D,\alpha)\neq \emptyset$ in terms of feasibility of an integer linear programming problem. 
	In general, integral feasibility is ${\sf NP}$-{\sf complete}. We now show that in our case,
	feasibility of the problem is equivalent to feasibility of its LP-relaxation:
	
	\begin{theorem}\label{thm:relaxation_equivalence}
		Let $D \subseteq [n]^2$ and $\alpha = (\alpha_1,\dots,\alpha_n) \in \mathbb{Z}^n$ with $\alpha_1 + \dots + \alpha_n = \#D$. Then $\mathcal{P}(D,\alpha) \cap \mathbb{Z}^{n^2} \neq \emptyset$ if and only if $\mathcal{P}(D,\alpha) \neq \emptyset$.
	\end{theorem}

	For reasons given in the Introduction, we provide two proofs of this fact. 
		
	\begin{proof}[Proof 1 of Theorem \ref{thm:relaxation_equivalence}.]

We write the constraints (A)-(C) in the form $M\vec{x} \leq \vec{b}$ where $M$ is
a $(3n^2+n)\times n^2$ block matrix and $\vec{b}$ is a vector of length $3n^2+n$ of the form
\[M=\left(\begin{matrix}
M_{A_1} \\
M_{A_2} \\
M_{B} \\
M_{C}
\end{matrix}
\right) \mbox{ and } \vec{b}=(b_i)_{i=1}^{3n^2+n}.\] Let $\vec{b}_I$ denote the subvector of $\vec{b}$ containing those $b_i$ with $i\in I\subseteq[3n^2+n]$. 
Also, we use the following coordinatization: 
\[\vec{x}=(\alpha_{11},\ldots,\alpha_{n1},\alpha_{12},\ldots,\alpha_{n2},\ldots,\alpha_{nn})^T.\]  

\begin{itemize}
\item  $M_{A_1}$ is the $n^2\times n^2$ block corresponding to the condition $0\leq\alpha_{ij}$ from (A). Thus, $M_{A_1}=-I_{n^2}$ and $b_r=0$ for $r\in[1,n^2]$. 
\item $M_{A_2}$ is  the $n^2\times n^2$ block corresponding to $\alpha_{ij}\leq 1$ from (A). Hence, 
$M_{A_2}=I_{n^2}$ and $b_r=1$ for $r\in[n^2+1,2n^2]$. 
\item $M_{C}$ is the $n^2\times n^2$ matrix for (C). Thus, 
\[M_{C}=\left(\begin{matrix}
M_{C_T} & 0 & \ldots & 0\\
0 & M_{C_T} & \ldots & 0\\
\vdots & \vdots & \ddots &\vdots \\
0 & 0 & \ldots & M_{C_T} 
\end{matrix}
\right)\] where $M_{C_T}=(c_{ij})_{1\leq i,j\leq n}$ is lower triangular such that $c_{ij}=-1$ for $i\geq j$. 
Also, \[b_{(2n^2+n)+n(j-1)+s}=-\#\{(i,j) \in D : i \leq s\}, 
\text{ \ for $s,j\in[n]$.}\]
\item $M_{B}$ is the $n\times n^2$ block encoding (B). Take $M_{B}=\left(\begin{matrix}
I_n & I_n & \ldots & I_n
\end{matrix}
\right)$ and $\vec{b}_{[2n^2+1,2n^2+n]}=(\alpha_i)_{i\in[n]}$. 
Clearly $M_{B}\vec{x}\leq (\alpha_i)_{i\in[n]}$ encodes the inequalities $\sum_{j=1}^n \alpha_{ij}\leq \alpha_i$.
Now, (B) requires $\sum_{j=1}^n \alpha_{ij}= \alpha_i$. However,
$\alpha_1 + \dots + \alpha_n = \#D$ ensures that \[\left(\begin{matrix}
M_{B} \\
M_{C}
\end{matrix}
\right)\vec{x} \leq  \vec{b}_{[2n^2+1,3n^2+n]}
\mbox{  \ only if \  }
M_{B}\vec{x}= (\alpha_i)_{i\in[n]}.\]
\end{itemize}
 Summarizing, $M\vec{x}\leq\vec{b}$ indeed encodes (A)-(C).

\begin{example}
For $n=2$ consider $\vec{x}=(\alpha_{11},\alpha_{21},\alpha_{12},\alpha_{22})^T$ with $D=\{(1,1),(1,2),(2,2)\}\subset [2]\times[2]$ and $\alpha=(2,1)$.

We have
\[M_{A_1}\vec{x}=\left(\begin{matrix}
-1 & 0 & 0 & 0\\
0 & -1 & 0 & 0\\
0 & 0 & -1 & 0\\
0 & 0 & 0 & -1
\end{matrix} 
\right)\left(\begin{matrix}
\alpha_{11} \\
\alpha_{21} \\
\alpha_{12} \\
\alpha_{22} 
\end{matrix} 
\right)\leq \left(\begin{matrix}
0 \\
0 \\
0 \\
0 
\end{matrix} 
\right)\]
\[M_{A_2}\vec{x}=\left(\begin{matrix}
1 & 0 & 0 & 0\\
0 & 1 & 0 & 0\\
0 & 0 & 1 & 0\\
0 & 0 & 0 & 1
\end{matrix} 
\right)\left(\begin{matrix}
\alpha_{11} \\
\alpha_{21} \\
\alpha_{12} \\
\alpha_{22} 
\end{matrix} 
\right)\leq \left(\begin{matrix}
1 \\
1 \\
1 \\
1 
\end{matrix} 
\right)\]
\[M_{B}\vec{x}=\left(\begin{matrix}
1 & 0 & 1 & 0\\
0 & 1 & 0 & 1
\end{matrix} 
\right)\left(\begin{matrix}
\alpha_{11} \\
\alpha_{21} \\
\alpha_{12} \\
\alpha_{22} 
\end{matrix} 
\right)\leq \left(\begin{matrix}
\alpha_1 \\
\alpha_2 
\end{matrix} 
\right)
= \left(\begin{matrix}
2 \\
1 
\end{matrix} 
\right)\]
\[M_{C}\vec{x}=\left(\begin{matrix}
-1 & 0 & 0 & 0\\
-1 & -1 & 0 & 0\\
0 & 0 & -1 & 0\\
0 & 0 & -1 & -1
\end{matrix} 
\right)\left(\begin{matrix}
\alpha_{11} \\
\alpha_{21} \\
\alpha_{12} \\
\alpha_{22} 
\end{matrix} 
\right)\leq \left(\begin{matrix}
-\#\{(i,1) \in D : i \leq 1\}\\
-\#\{(i,1) \in D : i \leq 2\}\\
-\#\{(i,2) \in D : i \leq 1\}\\
-\#\{(i,2) \in D : i \leq 2\}
\end{matrix} 
\right)
= \left(\begin{matrix}
-1 \\
-1 \\
-1 \\
-2 
\end{matrix} 
\right)\]
\end{example}

\begin{theorem}\label{thm:unimod}
$M$ is a totally unimodular matrix; that is, every minor of $M$ equals $0,1$, or $-1$.
\end{theorem}

\begin{proof}
Suppose $M'$ is a square submatrix of $M$ with $k$ rows from $M_{A_1}$ or $M_{A_2}$. We show by induction on $k$ that $\det(M')\in\{0,\pm 1\}$.

For the base case $k=0$, consider $M'$ an $\ell\times \ell$ submatrix of $M$ with only rows from $M_{B}$ and $M_{C}$. Let $M'_B,M'_C$ be the corresponding blocks of $M'$, i.e. $M'=\left(\begin{matrix}
M'_{B} \\
M'_{C}
\end{matrix}
\right)$ where $M'_B$, or $M'_C$, is the submatrix of $M_B$, or $M_C$ respectively, using the rows and columns of $M'$. 
Since ${M}_B$ has one $1$ per column, $M'_B$ has at most one $1$ per column. 
By the form of ${M}_C$, it is straightforward to row reduce $M'_C$ to obtain a $(0,-1)$-matrix $M''_{C}$ 
with at most one $-1$ in each column.
Let $M''=\left(\begin{matrix}
M'_{B} \\
M''_{C}
\end{matrix}
\right)$, an $\ell\times \ell$ matrix. 
It is textbook (see \cite[Theorem~13.3]{Combopt}) that if a $(0,\pm1)$-matrix $N$
has at most one $1$ and at most one $-1$ in each column, $N$ is totally unimodular; hence  
$\det(M')=\pm\det(M'')\in \{0,-1,1\}$ as desired. Thus the base case holds.

Now suppose $M'$ is a square submatrix of $M$ that contains $k\geq1$ rows from $M_{A_1}$ or $M_{A_2}$. Let $R$ be such a row from $M_{A_1}$ or $M_{A_2}$. If $R$ contains all $0$'s, $\det(M')=0$, and we are done.
Otherwise $R$ contains a single $\pm1$. 
Hence the cofactor expansion for $\det(M')$ along $R$ gives $\det(M')=\pm \det(M'')$ where $M''$ is a submatrix of $M$ with $k-1$ rows from $M_{A_1}$ or $M_{A_2}$. So by induction, $\det(M')\in\{0,\pm 1\}$, as required.
\end{proof} 

Since $M$ is totally unimodular then any vertices of $M\vec{x}\leq \vec{b}$ are integral
\cite[Theorem~13.2]{Combopt}.
Thus, if ${\mathcal P}(D,\alpha)\neq \emptyset$ then its vertices are integral, i.e.,
${\mathcal P}(D,\alpha)\cap {\mathbb Z}^{n^2}\neq \emptyset$.
\end{proof}
	
	\begin{proof}[Proof 2 of Theorem \ref{thm:relaxation_equivalence}.]
		Given a point $(\alpha_{ij}) \in \mathcal{P}(D,\alpha)$, we say a pair of sequences
		\begin{align*}
		(r_1,\dots,r_{m+1};c_1,\dots,c_m) \in [n]^{m+1} \times [n]^m,
		\end{align*}
		for some $m \in \mathbb{Z}_{>0}$, is \emph{stable} at $(\alpha_{ij})$ if the properties (i)-(iv) below hold. The purpose of each property will become clear later.
		\begin{enumerate}
			\item[(i)] $r_{m+1} = r_1$.
			
			\item[(ii)] For all $k \in [m]$, $\alpha_{r_kc_k},\alpha_{r_{k+1}c_k} \not\in \mathbb{Z}$.
			
			\item[(iii)] For all $k \in [m]$, if $i > r_{k+1}$ and $\alpha_{ic_k} \not\in \mathbb{Z}$, then $i = r_k$.
			
			\item[(iv)] There exists $(r,c) \in [n]^2$ such that
			\begin{align*}
			\#\{k \in [m] : (r,c) = (r_k,c_k)\} \neq \#\{k \in [m] : (r,c) = (r_{k+1},c_k)\}.
			\end{align*}
		\end{enumerate}
		
		\begin{claim}\label{claim:stable_exists}
			For any $(\alpha_{ij}) \in \mathcal{P}(D,\alpha) \smallsetminus \mathbb{Z}^{n^2}$, there exists $(r_1,\dots,r_{m+1};c_1,\dots,c_m)$ stable at $(\alpha_{ij})$.
		\end{claim}
		
		\begin{proof}[Proof of Claim \ref{claim:stable_exists}.]
			Choose $r_1,c_1$ such that $\alpha_{r_1c_1} \not\in \mathbb{Z}$, and assume that we have fixed $r_k,c_k$ such that $\alpha_{r_kc_k} \not\in \mathbb{Z}$. By Proposition \ref{prop:column_content}, we have
			\begin{align*}
			\sum_{i=1}^{n} \alpha_{ic_k} = \#\{(i,c_k) \in D : i \in [n]\} \in \mathbb{Z}.
			\end{align*}
			Thus, as $\alpha_{r_kc_k} \not\in \mathbb{Z}$, it makes sense to set
			\begin{align}\label{eqn:next_stable_row}
			r_{k+1} = \max\{i \neq r_k : \alpha_{ic_k} \not\in \mathbb{Z}\}.
			\end{align}
			If $r_{k+1} = r_{\ell}$ for some $\ell \in [k]$, then end the construction of these sequences. Otherwise, the content conditions (B) say that
			\begin{align*}
			\sum_{j=1}^{n} \alpha_{r_{k+1}j} = \alpha_{r_{k+1}} \in \mathbb{Z},
			\end{align*}
			and since $\alpha_{r_{k+1}c_k} \not\in \mathbb{Z}$, we can choose $c_{k+1} \neq c_k$ such that $\alpha_{r_{k+1}c_{k+1}} \not\in \mathbb{Z}$, completing the recursive definition. By the pigeonhole principle, this process must halt, yielding sequences $r_1,\dots,r_{\ell},\dots,r_{m+1}$ and $c_1,\dots,c_{\ell},\dots,c_m$ with $r_{m+1} = r_{\ell}$.
			
			By disregarding the first $\ell-1$ terms of each sequence, we may assume $\ell = 1$ without loss of generality. Then we assert that $(r_1,\dots,r_{m+1};c_1,\dots,c_m)$ is stable at $(\alpha_{ij})$. Indeed, (i) and (ii) are immediate from the construction, (iii) follows from \eqref{eqn:next_stable_row}, and (iv) holds because $(r,c) := (r_2,c_2)$ exists and satisfies
			\begin{align*}
			\#\{k \in [m] : (r,c) = (r_k,c_k)\} = 1
			\qquad \text{and} \qquad
			\#\{k \in [m] : (r,c) = (r_{k+1},c_k)\} = 0,
			\end{align*}
			since $c_2 \neq c_1$ and $r_2 \neq r_k$ for all $k \neq 2$.
		\end{proof}
		
		We now fix a pair of sequences $(r_1,\dots,r_{m+1};c_1,\dots,c_m)$. Given $(\alpha_{ij})$ and $\delta > 0$, set
		\begin{align}
                  \label{eqn:alphadeltadef}
		\alpha^{\delta}_{ij} = \alpha_{ij} + \delta[\#\{k \in [m] : (i,j) = (r_k,c_k)\} - \#\{k \in [m] : (i,j) = (r_{k+1},c_k)\}].
		\end{align}
		
		\begin{claim}\label{claim:is_stable}
			If $(r_1,\dots,r_{m+1};c_1,\dots,c_m)$ is stable at $(\alpha_{ij}) \in \mathcal{P}(D,\alpha)$, then $(\alpha^{\delta}_{ij}) \in \mathcal{P}(D,\alpha)$ for some $\delta > 0$.
		\end{claim}
		
		\begin{proof}[Proof of Claim \ref{claim:is_stable}.]
			First, note that the content conditions (B) are preserved regardless of our choice of $\delta$. Indeed, for each $i \in [n]$,
			\begin{align*}
			\sum_{j=1}^{n} \alpha^{\delta}_{ij} &= \sum_{j=1}^{n} [\alpha_{ij} + \delta[\#\{k \in [m] : (i,j) = (r_k,c_k)\} - \#\{k \in [m] : (i,j) = (r_{k+1},c_k)\}]]
			\\&= \alpha_i + \delta[\#\{k \in [m] : i = r_k\} - \#\{k \in [m] : i = r_{k+1}\}],
			\end{align*}
			and the term in brackets vanishes by (i). 
			
			We next check the flag conditions (C). For each $s,j \in [n]$, we can write
			\begin{align}\label{eqn:delta_flag_bound}
			\nonumber\sum_{i=1}^s \alpha^{\delta}_{ij} &= \sum_{i=1}^{s} [\alpha_{ij} + \delta[\#\{k \in [m] : (i,j) = (r_k,c_k)\} - \#\{k \in [m] : (i,j) = (r_{k+1},c_k)\}]]
			\\\nonumber&= \sum_{i=1}^{s} \alpha_{ij} + \delta[\#\{k \in [m] : s \geq r_k \text{ and } j = c_k\} - \#\{k \in [m] : s \geq r_{k+1} \text{ and } j = c_k\}]
			\\&\geq \sum_{i=1}^{s} \alpha_{ij} - \delta[\#\{k \in [m] : r_{k+1} \leq s < r_k \text{ and } j = c_k\}].
			\end{align}
			Thus, if $\#\{k \in [m] : r_{k+1} \leq s < r_k \text{ and } j = c_k\} = 0$, then the flag condition (C) for these $s,j$ is preserved. 

Otherwise, $r_{k+1} \leq s < r_k$ and $j = c_k$ for some $k \in [m]$, so (ii) and (iii) tell us that there is exactly one $i > s$ for which $\alpha_{ij} \not\in \mathbb{Z}$, namely $i = r_k$. This, combined with Proposition \ref{prop:column_content}, shows that
			\begin{align}\label{eqn:delta_flag_strict}
			\sum_{i=1}^{s} \alpha_{ij} = \sum_{i=1}^{n} \alpha_{ij} - \sum_{i=s+1}^{n} \alpha_{ij} = \#\{(i,j) \in D : i \in [n]\} - \sum_{i=s+1}^{n} \alpha_{ij} \not\in \mathbb{Z}.
			\end{align}
			By the nonintegrality from (\ref{eqn:delta_flag_strict}),
the flag inequalities (C) for $(\alpha_{ij}) \in \mathcal{P}(D,\alpha)$ are strict:
			\begin{align}\label{eqn:oct18abc}
			\sum_{i=1}^{s} \alpha_{ij} > \#\{(i,j) \in D : i \leq s\}.
			\end{align}
			Hence, by taking $\delta$ sufficiently small and applying \eqref{eqn:delta_flag_bound} and \eqref{eqn:oct18abc}, we can ensure
			\begin{align*}
						\sum_{i=1}^s \alpha^{\delta}_{ij} \geq \sum_{i=1}^{s} \alpha_{ij} - \delta[\#\{k \in [m] : r_{k+1} \leq s < r_k \text{ and } j = c_k\}] \geq \#\{(i,j) \in D : i \leq s\}
			\end{align*}
			for all $s,j \in [n]$, so the flag conditions (C) will be preserved. If $\alpha_{ij} \neq \alpha^{\delta}_{ij}$ 
then by (\ref{eqn:alphadeltadef}) we must have $(i,j)=(r_k,c_k)$ or $(i,j)=(r_{k+1},c_k)$ for some $k$, which by (ii) implies 
$0 < \alpha_{ij} < 1$. So we can require in addition that $\delta$ be small enough that $0 \leq \alpha^{\delta}_{ij} \leq 1$ for all $i,j \in [n]$. For such $\delta$, the conditions (A)-(C) all hold, so $(\alpha^{\delta}_{ij}) \in \mathcal{P}(D,\alpha)$.
		\end{proof}
		
		Finally, choose a point $(\alpha_{ij}) \in \mathcal{P}(D,\alpha)$ with the maximum number of integer coordinates. If $(\alpha_{ij}) \in \mathbb{Z}^{n^2}$, then we are done. Otherwise, there exists $(r_1,\dots,r_{m+1};c_1,\dots,c_m)$ that is stable at $(\alpha_{ij})$ by Claim \ref{claim:stable_exists}. By (iv), there exists $(r,c) \in [n]^2$ such that $|\alpha^{\delta}_{rc}| \rightarrow \infty$ as $\delta \rightarrow \infty$, so $\alpha_{rc}^{\delta}$ violates the column-injectivity conditions (A) for large $\delta$. This, combined with Claim \ref{claim:is_stable}, shows that the set $S = \{\delta > 0 : (\alpha^{\delta}_{ij}) \in \mathcal{P}(D,\alpha)\}$ is nonempty and bounded above. Thus, we can define $\eta = \sup S$ and set $(\tilde{\alpha}_{ij}) = (\alpha^{\eta}_{ij})$. Since $\mathcal{P}(D,\alpha)$ is closed and the map $\delta \mapsto (\alpha_{ij}^{\delta})$ from $S$ to $\mathcal{P}(D,\alpha)$ is continuous, this supremum is in fact a maximum, and $(\tilde{\alpha}_{ij}) \in \mathcal{P}(D,\alpha)$. By our choice of $(\alpha_{ij})$, we cannot have $\tilde{\alpha}_{r_kc_k} \in \mathbb{Z}$ or $\tilde{\alpha}_{r_{k+1}c_k} \in \mathbb{Z}$ for any $k \in [m]$, since then $(\tilde{\alpha}_{ij})$ has more integer coordinates than $(\alpha_{ij})$. Thus, $(r_1,\dots,r_{m+1};c_1,\dots,c_m)$ is stable at $(\tilde{\alpha}_{ij})$, so by Claim \ref{claim:is_stable}, there exists $\delta > 0$ for which $(\tilde{\alpha}^{\delta}_{ij}) \in \mathcal{P}(D,\alpha)$. But then $(\alpha^{\eta+\delta}_{ij}) = (\tilde{\alpha}^{\delta}_{ij}) \in \mathcal{P}(D,\alpha)$, contradicting the maximality of $\eta$.
	\end{proof}
	
In summary, applying the results of this section to $D=D(w)$,
\begin{equation}\label{eq:mainEquiv}
c_{\alpha,w}>0\hspace*{-0.08cm}
\overset{\text{\cite{Fink}}}{{\iff}}\hspace*{-0.08cm}
\alpha \in \mathcal{S}_{D}\hspace*{-0.08cm}
\overset{\text{\stackanchor[3pt]{Thm}{\ref{thm:independent_characterization}}} }{{\iff}}\hspace*{-0.08cm}
{\sf PerfectTab}(D,\alpha)\neq\emptyset \hspace*{-0.08cm}
\overset{\text{\stackanchor[3pt]{Thm}{\ref{thm:int_pt_characterization}}}}{{\iff}}\hspace*{-0.08cm}
\mathcal{P}(D,\alpha) \cap \mathbb{Z}^{n^2} \neq \emptyset\hspace*{-0.08cm}
\overset{\text{\stackanchor[3pt]{Thm}{\ref{thm:relaxation_equivalence}}} }{{\iff}}\hspace*{-0.08cm}
\mathcal{P}(D,\alpha) \neq \emptyset.
\end{equation}
	
	If $D\subseteq [n]^2$ has many identical columns, then many of the flag conditions (C) will look essentially the same. Thus, for efficiency of computation, we construct a ``compressed" version of $\mathcal{P}(D,\alpha)$ that removes some of the repetitive inequalities.
				
		A tuple ${\mathcal C}=(m,\{P_k\}_{k=1}^\ell,\{p_k\}_{k=1}^\ell,\{\lambda_k\}_{k=1}^\ell)$ is a \emph{compression} of $D \subseteq [n]^2$ if:
	\begin{itemize}
		\item $m \leq n$ is a nonnegative integer such that $(r,p) \not\in D$ whenever $r > m$ and $p \in [n]$,
		
		\item $P=P_1\dot{\cup}\cdots\dot{\cup}P_{\ell} \subseteq [n]$ such that if  $p,p' \in P_k$ then
	\[\{r \in [n] : (r,p) \in D\} = \{r \in [n] : (r,p') \in D\},\]
and moreover if $D$ is nonempty in column $p$ then $p\in P_k$ for some $k\in [\ell]$.
	
		\item $p_k \in P_k$ a representative for each $k \in [\ell]$, and
		
		\item $\lambda_k = \#P_k$ for each $k \in {\ell}$.
	\end{itemize}

	For $D\subseteq [n]^2$, a compression ${\mathcal C}$ of $D$, and 
$\tilde{\alpha} = (\tilde{\alpha}_1,\dots,\tilde{\alpha}_m) \in \mathbb{Z}_{\geq 0}^m$ define
	\begin{align}\label{eqn:compressed_arguments}
		\mathcal{Q}(D,{\mathcal C},\tilde{\alpha}) \subseteq \mathbb{R}^{m\ell}
	\end{align}
	to be the polytope with points of the form $(\tilde{\alpha}_{ik})_{i\in[m],k\in[\ell]}$ satisfying (A')-(C') below.
	
	\begin{enumerate}
		\item[(A')] Column-Injectivity Conditions: For all $i \in [m], k \in [\ell]$,
		\begin{align*}
		0 \leq \tilde{\alpha}_{ik} \leq 1.
		\end{align*}
		
		\item[(B')] Content Conditions: For all $i \in [m]$,
		\begin{align*}
		\sum_{k=1}^{\ell} \lambda_k\tilde{\alpha}_{ik} = \alpha_i.
		\end{align*}
		
		\item[(C')] Flag Conditions: For all $s \in [m], k \in [\ell]$,
		\begin{align*}
		\sum_{i=1}^s \tilde{\alpha}_{ik} \geq \#\{(i,p_k) \in D : i \leq s\}.
		\end{align*}
	\end{enumerate}
	
	\begin{remark}\label{rmk:trivial_compression}
		We can always take $m = \ell = n$ and $P_k = \{k\}$ for each $k \in [\ell]$, in which case 
$\mathcal{Q}(D,{\mathcal C},\tilde{\alpha}) = \mathcal{P}(D,\alpha) \subseteq \mathbb{R}^{n^2}$.
	\end{remark}
	
	\begin{theorem}\label{thm:compression_equivalence} Let $\alpha = (\alpha_1,\dots,\alpha_n) \in \mathbb{Z}_{\geq 0}^n$ and $\tilde{\alpha} = (\tilde{\alpha}_1,\dots,\tilde{\alpha}_m) := (\alpha_1,\dots,\alpha_m)$. Then
		$\alpha_1 + \dots + \alpha_n = \#D$ and $\mathcal{P}(D,\alpha) \neq \emptyset$ if and only if $\alpha_1 + \dots + \alpha_m = \#D$, $\alpha_{m+1} = \dots = \alpha_n = 0$, and $\mathcal{Q}(D,{\mathcal C},\tilde{\alpha}) \neq \emptyset$.
	\end{theorem}
	
	\begin{proof}
		($\Rightarrow$) Let $(\alpha_{ij}) \in \mathcal{P}(D,\alpha)$. Then by the content and flag conditions (B) and (C),
		\begin{align*}
		\#D = \alpha_1 + \dots + \alpha_n &\geq \alpha_1 + \dots + \alpha_m = \sum_{i=1}^{m} \sum_{j=1}^{n} \alpha_{ij}
		\\&= \sum_{j=1}^{n} \sum_{i=1}^{m} \alpha_{ij} \geq \sum_{j=1}^{n} \#\{(i,j) \in D : i \leq m\} = \#D.
		\end{align*}
		Thus, $\alpha_1 + \dots + \alpha_m = \#D$ and $\alpha_{m+1} = \dots = \alpha_n = 0$. Now, for each $i \in [m]$ and $k \in [\ell]$, set
		\begin{align*}
		\tilde{\alpha}_{ik} = \frac{1}{\lambda_k} \sum_{j \in P_k} \alpha_{ij}.
		\end{align*}
		We claim that $(\tilde{\alpha}_{ik}) \in \mathcal{Q}(D,{\mathcal C},\alpha)$. First, for each $i \in [m]$ and $k \in [\ell]$, we have
		\begin{align*}
		0 \leq \tilde{\alpha}_{ik} = \frac{1}{\lambda_k} \sum_{j \in P_k} \alpha_{ij} \leq \frac{1}{\lambda_k} \sum_{j \in P_k} 1 = 1,
		\end{align*}
		so the column-injectivity conditions (A') are satisfied. Next, for each $i \in [m]$, (B) implies
		\begin{align*}
		\sum_{k=1}^{\ell} \lambda_k \tilde{\alpha}_{ik} = \sum_{k=1}^{\ell} \sum_{j \in P_k} \alpha_{ij} = \sum_{j=1}^{n} \alpha_{ij} = \alpha_i,
		\end{align*}
		so the content conditions (B') are satisfied. Finally, for each $s \in [m]$ and $k \in [\ell]$, (C) implies
		\begin{align*}
		\sum_{i=1}^{s} \tilde{\alpha}_{ik} = \frac{1}{\lambda_k} \sum_{j \in P_k} \sum_{i=1}^{s} \alpha_{ij} \geq \frac{1}{\lambda_k} \sum_{j \in P_k} \#\{(i,j) \in D : i \leq s\} = \#\{(i,p_k) \in D : i \leq s\},
		\end{align*}
		so the flag conditions (C') are satisfied.
		
		\noindent
		($\Leftarrow$) Clearly $\alpha_1 + \dots + \alpha_n = \#D$. Let $(\tilde{\alpha}_{ik}) \in 
\mathcal{Q}(D,{\mathcal C},\tilde{\alpha})$. For each $i,j \in [n]$, set
		\begin{align*}
		\alpha_{ij} =
		\begin{cases}
		0 &\text{if } i > m, \\
		\tilde{\alpha}_{ik} &\text{if } i \leq m \text{ and } j \in P_k.
		\end{cases}
		\end{align*}
		We claim that $(\alpha_{ij}) \in \mathcal{P}(D,\alpha)$. The column-injectivity conditions (A) are clear. If $i > m$,
		\begin{align*}
		\sum_{j=1}^{n} \alpha_{ij} = 0 = \alpha_i.
		\end{align*}
		Otherwise $i \leq m$, and (B') implies
		\begin{align*}
		\sum_{j=1}^{n} \alpha_{ij} = \sum_{k=1}^{\ell} \sum_{j \in P_k} \tilde{\alpha}_{ik} = \sum_{k=1}^{\ell} \lambda_k\tilde{\alpha}_{ik} = \alpha_i.
		\end{align*}
		Thus, the content conditions (B) hold. Finally, if $s \in [n]$ and $j \in P_k$, then (C') implies
		\begin{align*}
		\sum_{i=1}^{s} \alpha_{ij} = \sum_{i=1}^{\min\{s,m\}} \tilde{\alpha}_{ik} \geq \#\{(i,p_k) \in D: i \leq \min\{s,m\}\} = \#\{(i,j) \in D: i \leq s\}.
		\end{align*}
		Hence, the flag conditions (C) hold as well.
	\end{proof}
	
	\subsection{Deciding membership in the Schubitope}
	
	We use the above results of this section to give a polynomial time algorithm to check if a lattice point is in the Schubitope.
	
		Let $D \subseteq [n]^2$, and fix a compression ${\mathcal C}=(m,\{P_k\}_{k=1}^{\ell},\{p_k\}_{k=1}^{\ell},
\{\lambda_k\}_{k=1}^{\ell})$ of $D$ (as in Section \ref{subsctn:polytopes}).

	\begin{theorem}\label{thm:lp_characterization}
		Let $\alpha = (\alpha_1,\dots,\alpha_n) \in \mathbb{Z}_{\geq 0}^n$. Then $\alpha \in \mathcal{S}_D$ if and only if $\alpha_1 + \dots + \alpha_m = \#D$, $\alpha_{m+1} = \dots = \alpha_n = 0$, and $\mathcal{Q}(D,{\mathcal C},\tilde{\alpha}) \neq \emptyset$, where $\tilde{\alpha} = (\tilde{\alpha}_1,\dots,\tilde{\alpha}_m) := (\alpha_1,\dots,\alpha_m)$.
	\end{theorem}
	
	\begin{proof}
		This follows from Theorems \ref{thm:independent_characterization}, \ref{thm:int_pt_characterization}, \ref{thm:relaxation_equivalence}, and \ref{thm:compression_equivalence}.
	\end{proof}
	
	For each $k \in [\ell]$, let $R_k({\mathcal C})= \{r \in [n] : (r,p_k) \in D\} \subseteq [m]$.

	\begin{theorem}\label{thm:poly_time_schubitope}
		Given as input $\{R_k({\mathcal C})\}_{k=1}^{\ell}$, $\{\lambda_k\}_{k=1}^{\ell}$, and $\tilde{\alpha} = (\tilde{\alpha}_1,\dots,\tilde{\alpha}_m) \in \mathbb{Z}_{\geq 0}^m$ satisfying $\tilde{\alpha}_1 + \dots + \tilde{\alpha}_m = \#D$, one can decide if $\alpha := (\tilde{\alpha}_1,\dots,\tilde{\alpha}_m,0,\dots,0) \in \mathbb{Z}_{\geq 0}^n$ lies in $\mathcal{S}_D$ in polynomial time in $m$ and $\ell$.
	\end{theorem}
	
	\begin{remark}\label{rmk:natural_input}
		In view of Theorem \ref{thm:lp_characterization}, this input is most natural, because the conditions $\alpha_1 + \dots + \alpha_m = \#D$ and $\alpha_{m+1} = \dots = \alpha_n = 0$ are clearly necessary, and it contains the minimum amount of information we need to compute $\mathcal{Q}(D,{\mathcal C},\tilde{\alpha})$.
	\end{remark}
	
	\begin{remark}\label{rmk:poly_in_n}
		As in Remark \ref{rmk:trivial_compression}, we can take $m = \ell = n$ and $P_k = \{k\}$ for each $k \in [\ell]$, so we can check if $\alpha$ is in $\mathcal{S}_D$ in polynomial time in $n$ regardless of the structure of $D$.
	\end{remark}
	
	\begin{proof}[Proof of Theorem \ref{thm:poly_time_schubitope}]
		Since  $R_k({\mathcal C})$ takes $m$ bits to encode for each $k \in [\ell]$, and $\mathcal{Q}(D,{\mathcal C},\tilde{\alpha}) \subseteq \mathbb{R}^{m\ell}$ is governed by $O(m\ell)$ constraints, $\mathcal{Q}(D,{\mathcal C},\tilde{\alpha})$ can be constructed in polynomial time in $m$ and $\ell$. By Theorem \ref{thm:lp_characterization}, we are done using L.~Khachiyan's ellipsoid method \cite{Schrijver}.
	\end{proof}
	
\section{Computing Rothe diagrams}
We will repeatedly use the following to establish the complexity of computing preliminary data of $D(w)$ given ${\sf code}(w)$.

\begin{proposition}\label{prop:wPoly}
There exists an $O(L^{2})$-time algorithm to compute $(w(1),\ldots,w(L))$ from the input ${\sf code}(w)=(c_1,\ldots,c_L)$.
\end{proposition}

\begin{proof}
Clearly $w(1)=c_1+1$. After determining $w(1),\ldots, w(i-1)$, we determine (in $O(L)$-time) $\pi:=\pi^{(i)}\in S_{i-1}$ such that
$(w(\pi(1))<w(\pi(2))<\ldots<w(\pi(i-1)))$. 
Next, set
\[B:=(w(\pi(1)), w(\pi(2))-w(\pi(1)), w(\pi(3))-w(\pi(2)), \ldots, w(\pi(i-1))-w(\pi(i-2)))).\]
Let 
\[V_t:=\sum_{j=1}^{t} (B_j-1), \text{ \ for $0\leq t\leq i-1$}.\] 
Set $w(i):=c_i+T+1$ where
 $T:=\max_{t\in[0,i-1]}\{t \ : \ c_i\geq V_t\}$. By construction, $w(1),\ldots,w(i)$ is a partial permutation with code $(c_1,\ldots,c_{i-1},c_i)$.
Each stage $1\leq i\leq L$ takes $O(i)$-time. 
\end{proof}

The \emph{essential set} of $w$ 
consists of the maximally southeast boxes of each
connected component of $D(w)$, i.e.,
\begin{equation}\label{eqn:August2abc}
{\sf Ess}(w)=\{(i,j)\in D(w): (i+1,j),(i,j+1)\not\in D(w)\}.
\end{equation}
If it exists, we call the connected component of $D(w)$ involving $(1,1)$ the \emph{dominant component} and denote it by ${\sf Dom}(w)$. For instance, in Example \ref{ex:Rothe}, ${\sf Dom}(w)$ has shape $(4,2,2,2)$.
Further, if it exists, the \emph{accessible box} $\mathbf{z}_w$ is the southmost then eastmost box in ${\sf Ess}(w)\smallsetminus{\sf Dom}(w)$. In Example \ref{ex:Rothe},
\[{\sf Ess}(w)=\{(1,4),(3,4),(3,7),(4,2)\} 
\text{\ and $\mathbf{z}_w=(3,7)$.}\] 
(Although $(4,2)$ is the southmost box of ${\sf Ess}(w)$,
it is in ${\sf Dom}(w)$, and hence not the accessible.) 

We will need the following in Section~\ref{sec:schubinsharpP}:

\begin{proposition}\label{prop:accBoxPoly}
Given ${\sf code}(w)$, there exists an $O(L^{2})$-time algorithm to compute $\mathbf{z}_w=(r,c)$ or determine it does not exist.
\end{proposition}
\begin{proof}
Use Proposition \ref{prop:wPoly} to find $(w(1),\ldots,w(L))$ in $O(L^2)$-time. 
Next, compute \[w_{NW}(i):=\{w(j):w(j)\leq w(i), \ 
j\leq i\}.\]
Then take \[Y(i):=\{q-1 \ : \ q\in w_{NW}(i)\}\smallsetminus w_{NW}(i), \text{ \ for $i\in [L]$}.\]  

Compute $k_i:=\max Y(i)$ for $i\in [L]$ in $O(L^2)$-time
(if $k_i\geq 1$, then $k_i$ is the column index of the eastmost box of $D(w)$ in row $i$).  
In $O(L^2)$-time, calculate
 \[I:=\{i\in[2,\ldots,L] \ : \ k_i>\min_{j<i}w(j)\}.\]
Let $Y:=\{(i,k_i) \ : \  i\in I\}$. Hence, $Y\cap {\sf Dom}(w)=\emptyset$. Thus,
if $Y=\emptyset$, $\mathbf{z}_w$ does not exist. Otherwise, $\mathbf{z}_w\in Y$. Thus, in $O(L)$-time, determine
$r:=\max\{i:(i,k_i)\in Y\}$. Output $\mathbf{z}_w=(r,k_r)$. 
\end{proof}

The \emph{pivots} of $\mathbf{z}_w$ denoted ${\sf Piv}(\mathbf{z}_w)$ are the $\bullet$'s of $D(w)$ that are maximally southeast, among those northwest of $\mathbf{z}_w$.  In Example \ref{ex:Rothe},
${\sf Piv}((3,7))=\{(2,3),(1,5)\}$.

\section{Proofs of Theorems \ref{thm:SchubertinP} and \ref{thm:Advetableau}}
\subsection{Proof of Theorem \ref{thm:SchubertinP}}\label{sec:inPsec}
By (\ref{eq:schubSchub}) combined with Theorem \ref{thm:poly_time_schubitope}, it remains to establish the complexity of computing a compression of $D(w)$. For this, we need the following lemmas and propositions. Fix $w\in S_\infty$ with ${\sf code}(w)=(c_1,\ldots, c_L)$.
Let $\sigma\in S_L$ be such that 
$\{w(\sigma(1))<w(\sigma(2))<\ldots< w(\sigma(L))\}$. For convenience, set $w(\sigma(0)):=0$.

\begin{lemma}\label{prop:colTypeRothe}
For $1\leq h\leq L$, and for all 
\[j_1,j_2\in\{w(\sigma(h-1))+1,w(\sigma(h-1))+2,\ldots,w(\sigma(h))-1\},\]
we have $(i,j_1)\in D(w)$ if and only if $(i,j_2)\in D(w)$.
\end{lemma}
\begin{proof}
For each $k$, let 
$u_1^{(k)}<\ldots<u_k^{(k)}$ 
be $w(1),w(2),\ldots,w(k)$ sorted in increasing order. Set $u_0^{(k)}:=0$. 
The lemma follows from the inductive claim 
that in the first $k$ rows of $D(w)$, the columns
$u_{h-1}^{(k)}+1,u_{h-1}^{(k)}+2,\ldots,u_{h}^{(k)}-1$
are the same. The base case $k=1$ is clear. The inductive step is straightforward by considering how, in row $k+1$ of $D(w)$, the $\bullet$ and 
its ray emanating east affects the columns.
\end{proof}

Define a collection of intervals in $[n]$ by
\[P'_{2k-1}:=[w(\sigma(k-1))+1,w(\sigma(k))-1]
\text{ \ and \ } P'_{2k}:=\{w(\sigma(k))\}, \mbox{ for } 1\leq k\leq L.\]

Let $1\leq h_1<h_2<\ldots<h_\ell\leq 2L$ be indices of the intervals $P_h'$ that are nonempty. Set $P_i:=P_{h_i}'$. 

\begin{lemma}\label{prop:colPartRothe}
If $j_1,j_2\in P_k$ for some $k$, then $(i,j_1)\in D(w) \iff (i,j_2)\in D(w)$. 
\end{lemma}
\begin{proof}
This follows by the definition of $\{P_k\}_{k=1}^{\ell}$ together with Lemma \ref{prop:colTypeRothe}.
\end{proof}

Let $p_k:=\min\{p\in P_k\}$ for each $k\in[\ell]$. 

\begin{proposition}\label{prop:rothePartitionPoly}
There exists an $O(L^{2})$-time algorithm to compute $\{P_k\}_{k=1}^{\ell},\{p_k\}_{k=1}^\ell,$ and $\{\#P_k\}_{k=1}^\ell$ from the input ${\sf code}(w)=(c_1,\ldots,c_L)$.
\end{proposition}
\begin{proof}
Proposition \ref{prop:wPoly} computes $(w(1),\ldots,w(L))$ in $O(L^{2})$-time. It takes $O(L\log(L))$-time to sort $(w(1),\ldots,w(L))$, i.e., to compute $\sigma\in S_L$. Computing
the endpoints, and thus cardinalities, of the $P'_k$ takes $O(L)$-time as there are at most $2L$ of them. Then we reindex $\{\#P'_k\}_{k=1}^{2L}$ to obtain $\{\#P_k\}_{k=1}^\ell$ in $O(L)$-time.
\end{proof}

For each $k \in [\ell]$, let \[R_k := \{r \in [L] : (r,p_k) \in D(w)\}.\]
	
\begin{proposition}\label{prop:rothePoly}
Computing $\{R_k\}_{k=1}^\ell$ from ${\sf code}(w)$ takes $O(L^{2})$-time.
\end{proposition}
\begin{proof}
By $D(w)$'s definition, $r\in R_k$ if and only if $w(r)> p_k$ and $p_k\not\in\{w(i)  :  i<r\}$.
Propositions~\ref{prop:rothePartitionPoly} and~\ref{prop:wPoly}
give $\{P_k\}_{k=1}^{\ell}$, $\{p_k\}_{k=1}^\ell$ and 
$\{w(1),\ldots,w(L)\}$ in $O(L^2)$-time.
\end{proof}

\noindent\emph{Conclusion of proof of Theorem \ref{thm:SchubertinP}:}
Proposition \ref{prop:rothePartitionPoly} computes $\{P_k\}_{k=1}^\ell, \{p_k\}_{k=1}^\ell,$ and $\{\#P_k\}_{k=1}^\ell$ in $O(L^2)$-time. 
Proposition \ref{prop:rothePoly} finds $\{R_k\}_{k=1}^\ell$ in $O(L^2)$-time. 
One checks, using Lemma~\ref{prop:colPartRothe}, that  ${\mathcal C}=(L,\{P_k\}_{k=1}^\ell,\{p_k\}_{k=1}^\ell,\{\#P_k\}_{k=1}^\ell)$ is a 
compression of $D(w)$. Hence we may apply Theorem \ref{thm:poly_time_schubitope} by taking $D:=D(w)$, $R_k({\mathcal C}):=R_k$, $\lambda_k:=\# P_k$ for $k\in[\ell]$ and $m:=L$. Thus the result follows by (\ref{eq:schubSchub}).
\qed

\subsection{Proof of Theorem \ref{thm:Advetableau}; an application}
Remark~\ref{rmk:injective_iff_increasing} combined with (\ref{eq:mainEquiv}) proves the theorem.
\qed

Let $n_{132}(w)$ be the number of $132$-patterns in $w\in S_n$, that is,
\[n_{132}(w)=\#\{(i,j,k):1\leq i<j<k\leq n, w(i)<w(k)<w(j)\}.\]

\begin{corollary}
\label{cor:atleast}
There are at least $n_{132}(w)+1$ distinct vectors $\alpha$ such that $c_{\alpha,w}>0$.
\end{corollary}
\begin{proof}
Suppose $i<j<k$ index a $132$ pattern in $w$. There is a box ${\sf b}$ of $D(w)$ in row $j$ and column $w(k)$. There are $N:=n_{132}(w)$ many such boxes, ${\sf b}_1,\ldots, {\sf b}_{N}$ (all distinct),
listed in English language reading order. Let $M_i$ be boxes in the same column and connected component as ${\sf b}_i$
that are weakly north of ${\sf b}_i$ and strictly south of any ${\sf b}_j$, where $j<i$. Iteratively define fillings $F_0,F_1,F_2,\ldots,F_N$ of $D(w)$:
\begin{itemize}
\item[($F_0$)] Fill each box ${\sf c}$ of $D(w)$ with the row number of ${\sf c}$.
\item[($F_i$)] For $1\leq i\leq N$, $F_i$ is the same as $F_{i-1}$ except that 
$F_i({\sf c}):=F_{i-1}({\sf c})-1$ if ${\sf c}\in M_i$.
\end{itemize}

Clearly, $F_0\in \tabSort(D(w)):=\bigcup_{\alpha} \tabSort(D(w),\alpha)$.
 Inductively assume $F_{i-1}\in \tabSort(D(w))$. 
 Since labels only decrease, $F_i$ satisfies the row bound condition. 
 Next we check that each column is strictly increasing.
 Let ${\sf m}_i$ be the northmost box of $M_i$. If ${\sf m}_i$
is adjacent and directly below some ${\sf b}_j$ (for a $j<i$) then
\[F_{i}({\sf b}_j)=F_{0}({\sf b}_j)-1<F_0({\sf m}_i)-1=F_i({\sf m}_i),\] 
as needed. Otherwise suppose ${\sf m}_i$ is adjacent and south of  a non-diagram position. Let ${\sf d}_i$ (if it exists) be
the first diagram box directly north of ${\sf m}_i$. Then $F_0({\sf d}_i)< F_{0}({\sf m}_i)-1$.
Hence 
\[
F_i({\sf d}_i)\leq F_{0}({\sf d}_i)<F_{0}({\sf m}_i)-1=F_i({\sf m}_i),
\]
verifying column increasingness here as well. That $F_i$ is column increasing elsewhere is clear since
$F_{i-1}$ is column increasing (by induction) and only labels of $M_i$ are changed.

It remains to check that every label of $F_i$ is in ${\mathbb Z}_{>0}$. Since
each box of $D(w)$ is decremented at most once, the only concern is there is a box 
${\sf x}$ in the first row  that appears in some $M_i$, since then $F_0({\sf x})=1$
and $F_{i}({\sf x})=0$. However, in this case ${\sf b}_i$ must be in ${\sf Dom}(w)$, which implies that the ``$1$'' in the $132$-pattern associated to ${\sf b}_i$ could not exist, a contradiction. Thus $F_i\in \tabSort(D(w))$, completing the induction.

Finally, under Theorem~\ref{thm:Advetableau},
each $F_i$ corresponds to a distinct exponent vector since
the sum of the labels is strictly decreasing at each step $F_{i-1}\mapsto F_i$. 
\end{proof}

From Corollary~\ref{cor:atleast}, this result of A.~Weigandt \cite{Weigandt} is immediate:
\begin{corollary}[A.~Weigandt's $132$-bound]
\label{cor:Weigandt}
${\mathfrak S}_w(1,1,1,\ldots,1)\geq n_{132}(w)+1$.
\end{corollary}

As shown in  \cite{Weigandt}, Corollary~\ref{cor:Weigandt} in turn implies ${\mathfrak S}_w(1,1,\ldots,1)\geq 3$ if $n_{132}(w)\geq 2$,
a recent conjecture of R.~P.~Stanley \cite{Shanan}. 

\section{Counting $c_{\alpha,w}$ is in $\#{\sf P}$}
\label{sec:schubinsharpP}

\subsection{Vexillary permutations}\label{sec:vex}
A permutation $w\in S_n$ is \emph{vexillary} if there does not exist a \emph{$2143$ pattern}, i.e., indices
$i<j<k<l$ such that $w$ has the pattern $w(j)<w(i)<w(l)<w(k)$. 
For example, $w=\underline{5}\underline{3}\underline{8}412\underline{6}7$ is not vexillary; we underlined the positions of a $2143$ pattern. \emph{Fulton's criterion} states that $w$ is vexillary if and only if there do not exist $(a,b),(c,d)\in {\sf Ess}(w)$ such that $a<c$ and $b<d$. In Example \ref{ex:Rothe}, $w$ is not vexillary due to $(1,4)$ and $(3,7)$. 
Our main reference for this subsection is \cite[Chapter~2]{Manivel}.

We will also use this characterization of vexillary permutations:

\begin{theorem}\cite{LS:transition}\label{thm:vexCode}
Given ${\sf code}(w)=(c_1,\ldots,c_L)\in\mathbb{Z}_{\geq 0}^n$, $w$ vexillary if and only if
\begin{itemize}
\item[\emph{(i)}] if $i$ is such that $c_i>c_{i+1}$, then $c_i>c_j$ for any $j>i$, and
\item[\emph{(ii)}] if $i, h$ are such that $c_i\geq c_h$, then $\#\{j  \ : \ i<j<h,  \ c_j<c_h\}\leq c_i-c_h$.
\end{itemize}
\end{theorem}

The \emph{shape} of a vexillary permutation $v$ is the partition $\lambda(v)$ formed by sorting 
${\sf code}(v)=(c_1,c_2,\ldots)$ into decreasing order. Now, 
if $c_i\neq 0$, let $e_i$ be the greatest integer $j\geq i$ such that $c_j\geq c_i$.  The \emph{flag} 
\[\phi(v)=(\phi_1\leq \phi_2\leq\ldots\leq \phi_m)\] 
for $v$ is the sequence of $e_i$'s sorted into increasing order; see, e.g., \cite[Definition~2.2.9]{Manivel}. 

\begin{example}\label{ex:flag}  Consider ${\sf code}(v)=(5,1,3,1,2)$ for the vexillary $v=6253714$. Here
\[e=(1,5,3,5,5), \phi(v)=(1,3,5,5,5) \text{\ and $\lambda(v)=(5,3,2,1,1)$.}\]

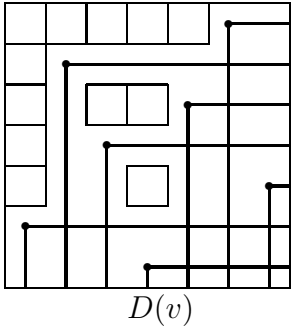
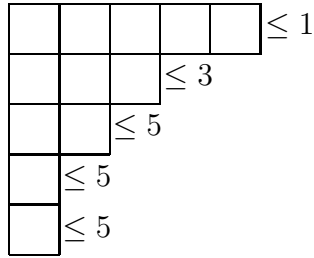
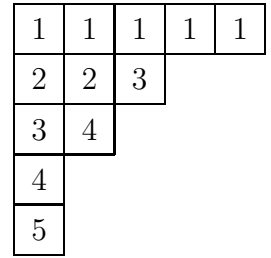
\setlength{\unitlength}{.27mm}
\[
\begin{picture}(570,140)
\put(60,0){$D(v)$}
\put(0,15){\framebox(140,140)}
\thicklines
\put(110,145){\circle*{4}}
\put(110,145){\line(1,0){30}}
\put(110,145){\line(0,-1){130}}
\put(30,125){\circle*{4}}
\put(30,125){\line(1,0){110}}
\put(30,125){\line(0,-1){110}}
\put(90,105){\circle*{4}}
\put(90,105){\line(1,0){50}}
\put(90,105){\line(0,-1){90}}
\put(50,85){\circle*{4}}
\put(50,85){\line(1,0){90}}
\put(50,85){\line(0,-1){70}}
\put(130,65){\circle*{4}}
\put(130,65){\line(1,0){10}}
\put(130,65){\line(0,-1){50}}
\put(10,45){\circle*{4}}
\put(10,45){\line(1,0){130}}
\put(10,45){\line(0,-1){30}}
\put(70,25){\circle*{4}}
\put(70,25){\line(1,0){70}}
\put(70,25){\line(0,-1){10}}

\thinlines
\put(0,135){\framebox(100,20)}
\put(20,135){\line(0,1){20}}
\put(40,135){\line(0,1){20}}
\put(60,135){\line(0,1){20}}
\put(80,135){\line(0,1){20}}
\put(20,55){\line(0,1){100}}

\put(0,55){\line(1,0){20}}
\put(0,75){\line(1,0){20}}
\put(0,95){\line(1,0){20}}
\put(0,115){\line(1,0){20}}

\put(40,95){\line(0,1){20}}
\put(40,95){\line(1,0){40}}
\put(60,95){\line(0,1){20}}
\put(80,95){\line(0,1){20}}
\put(40,115){\line(1,0){40}}

\put(60,55){\line(1,0){20}}
\put(60,75){\line(1,0){20}}
\put(60,55){\line(0,1){20}}
\put(80,55){\line(0,1){20}}

\put(230,0){$\lambda(v)$ flagged by $\phi(v)$}

\put(230,130){
$\tableau{{\ }&{\ }&{ \  }&{ \ }&{\ }\\{\ }&{\ }&{\ }\\{ \ }&{\ }\\{ \ }\\{ \ }}$}
\put(360,140){$\leq 1$}
\put(310,115){$\leq 3$}
\put(285,90){$\leq 5$}
\put(260,65){$\leq 5$}
\put(260,40){$\leq 5$}

\put(450,130){
$\tableau{{1 }&{1}&{1}&{1}&{1}\\{2}&{2}&{3}\\{3}&{4}\\{4 }\\{5}}$}
\put(440,0){$T\in{\sf SSYT}(\lambda(v), \phi(v))$}
\end{picture}
\]
\end{example}

For a partition $\lambda=(\lambda_1\geq \lambda_2\geq\ldots\geq \lambda_m \geq 0)$ and a flag $\phi=(\phi_1\leq \phi_2\leq\ldots\leq \phi_m)$ of positive integers, 
define the \emph{flagged Schur function}  
\[S_{\lambda}(\phi)=\det|h_{\lambda_i-i+j}(\phi_i)|_{i,j=1,\ldots,m},\] where \[h_k(n)=\sum_{1\leq i_1\leq\ldots\leq i_k\leq n}x_{i_1}\cdots x_{i_k}\] is the complete homogeneous symmetric polynomial of degree $k$. 
Furthermore,
\begin{equation}
\label{eqn:Sept28abc}  
{\mathfrak S}_v=S_{\lambda(v)}(\phi(v)), \text{\ for $v$ vexillary}.
\end{equation}

A semistandard Young tableau of shape $\lambda$ is \emph{flagged} by $\phi$ if its 
entries in row $i$ are $\leq \phi_i$; see Example \ref{ex:flag}. Denote the set of such tableaux by ${\sf SSYT}(\lambda,\phi)$.
Then
\begin{equation}
\label{eqn:flaggen}
S_{\lambda}(\phi)=\sum_{T\in{\sf SSYT}(\lambda,\phi)}x^{{\sf content}(T)}.
\end{equation}
where ${\sf content}(T)=(\mu_1,\ldots,\mu_{\ell(\lambda)})$ such that $\mu_i$ is the number of $i$'s in $T$.

\subsection{Graphical transition}
 The transition recurrence for ${\mathfrak S}_w$ was found by A.~Lascoux and M.-P. Sch\"{u}tzenberger \cite{LS:transition}. This is transition for the case discussed in \cite{Knutson.Yong}:

\begin{theorem}[\cite{LS:transition}, cf.~\cite{Knutson.Yong}]\label{thm:transition}
Let $\mathbf{z}_w=(r,c)$ and $w'=w\cdot (r\ k)$ where $k=w^{-1}(c)$.
Then 
\begin{equation}
\label{eqn:thetrans}
{\mathfrak S}_w=x_r {\mathfrak S}_{w'}+\sum_{w''=w'\cdot (i\ k)} {\mathfrak S}_{w''},
\end{equation}
where the summation is over $\{i:(i,w(i))\in {\sf Piv}(\mathbf{z}_w)\}$.
\end{theorem}

We will use the \emph{graphical} transition tree $\TT(w)$ of \cite{Knutson.Yong}. This reformulates (\ref{eqn:thetrans}) in terms of
Rothe diagrams and certain moves on these diagrams.
By definition,
$D(w)$ (equivalently $w$) will label the root of $\TT(w)$. If $w$ is vexillary, stop. Otherwise, there exists an accessible box
$\mathbf{z}_w=(r,c)\in D(w)$
(if not, $D(w)={\sf Dom}(w)$, contradicting $w$ is not vexillary). 

The children of $D(w)$ are Rothe diagrams resulting from two types of moves:

\begin{itemize}
\item[(T.1)] \emph{Deletion moves}: remove $\mathbf{z}_w$ from $D(w)$. The resulting diagram is $D(w')$. Add an edge $D(w)\stackrel{x_r}{\longrightarrow} D(w')$.

\item[(T.2)] 
\emph{March moves}: There is a move for each
$\mathbf{x}^{(i)}=(i,w(i))\in {\sf Piv}(\mathbf{z}_w)$. 
Let ${\mathcal R}$ be the rectangle with corners $\mathbf{z}_w$ and $\mathbf{x}^{(i)}$. Remove $\mathbf{x}^{(i)}$ and its rays from 
$G(w)$ to form $G^{(i)}(w)$. Order the boxes $\{{\sf b}_i\}_{i=1}^r$ in ${\mathcal R}$ in English reading order. Move ${\sf b}_1$ strictly north and strictly west to the closest position not occupied by other boxes of $D(w)$ or rays from $G^{(i)}(w)$. 
Repeat with ${\sf b}_2,{\sf b}_3,\ldots$ where ${\sf b}_j$ may move into a square left unoccupied by earlier moves.
The resulting diagram will be $D(w'')$ where $w''=w'\cdot (i \ k)$. 
Add an edge $D(w)\stackrel{i}{\longrightarrow} D(w'')$.
\end{itemize}

Repeat for each child $D(u)$. Stop when $u$ vexillary; these permutations are the leaves
$\LL(w)$ of $\TT(w)$. (Multiple leaves may be labelled by the same permutation.)

\begin{example}\label{ex:marchMove} Let $w=53841267$. We compute the march move $2$ for the pivot $(2,3)$:
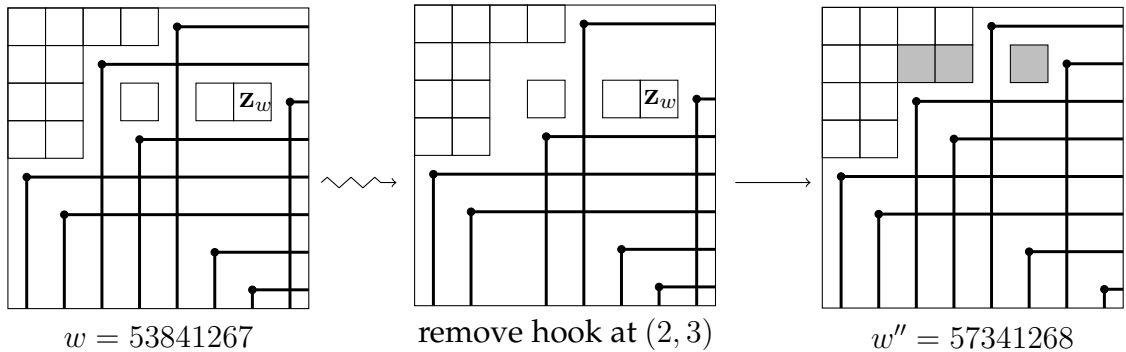
\begin{center}

\begin{tikzpicture}
\node (root) {\begin{tikzpicture}[scale=.5]

\draw (0,0) rectangle (8,8);

\draw (0,7) rectangle (1,8);
\draw (1,7) rectangle (2,8);
\draw (2,7) rectangle (3,8);
\draw (3,7) rectangle (4,8);

\draw (0,6) rectangle (1,7);
\draw (1,6) rectangle (2,7);

\draw (0,5) rectangle (1,6);
\draw (1,5) rectangle (2,6);

\draw (0,4) rectangle (1,5);
\draw (1,4) rectangle (2,5);

\draw (3,5) rectangle (4,6);

\draw (5,5) rectangle (6,6);

\draw (6,5) rectangle (7,6);

\node at (4,-.75) {$w=53841267$};
\node at (6.6,5.5) {$\mathbf{z}_w$};

\filldraw (4.5,7.5) circle (.5ex);
\draw[line width = .2ex] (4.5,0) -- (4.5,7.5) -- (8,7.5);
\filldraw (2.5,6.5) circle (.5ex);
\draw[line width = .2ex] (2.5,0) -- (2.5,6.5) -- (8,6.5);
\filldraw (7.5,5.5) circle (.5ex);
\draw[line width = .2ex] (7.5,0) -- (7.5,5.5) -- (8,5.5);
\filldraw (3.5,4.5) circle (.5ex);
\draw[line width = .2ex] (3.5,0) -- (3.5,4.5) -- (8,4.5);
\filldraw (0.5,3.5) circle (.5ex);
\draw[line width = .2ex] (0.5,0) -- (0.5,3.5) -- (8,3.5);
\filldraw (1.5,2.5) circle (.5ex);
\draw[line width = .2ex] (1.5,0) -- (1.5,2.5) -- (8,2.5);
\filldraw (5.5,1.5) circle (.5ex);
\draw[line width = .2ex] (5.5,0) -- (5.5,1.5) -- (8,1.5);
\filldraw (6.5,0.5) circle (.5ex);
\draw[line width = .2ex] (6.5,0) -- (6.5,0.5) -- (8,0.5);

\end{tikzpicture}};

\node[right=1 of root] (temp) {\begin{tikzpicture}[scale=.5]

\draw (0,0) rectangle (8,8);

\draw (0,7) rectangle (1,8);
\draw (1,7) rectangle (2,8);
\draw (2,7) rectangle (3,8);
\draw (3,7) rectangle (4,8);

\draw (0,6) rectangle (1,7);
\draw (1,6) rectangle (2,7);

\draw (0,5) rectangle (1,6);
\draw (1,5) rectangle (2,6);

\draw (0,4) rectangle (1,5);
\draw (1,4) rectangle (2,5);

\draw (3,5) rectangle (4,6);

\draw (5,5) rectangle (6,6);

\draw (6,5) rectangle (7,6);

\node at (5.8,5.5) {$\mathbf{z}_w$};
\node at (-.2,-.75) {remove hook at $(2,3)$};

\filldraw (4.5,7.5) circle (.5ex);
\draw[line width = .2ex] (4.5,0) -- (4.5,7.5) -- (8,7.5);
\filldraw (7.5,5.5) circle (.5ex);
\draw[line width = .2ex] (7.5,0) -- (7.5,5.5) -- (8,5.5);
\filldraw (3.5,4.5) circle (.5ex);
\draw[line width = .2ex] (3.5,0) -- (3.5,4.5) -- (8,4.5);
\filldraw (0.5,3.5) circle (.5ex);
\draw[line width = .2ex] (0.5,0) -- (0.5,3.5) -- (8,3.5);
\filldraw (1.5,2.5) circle (.5ex);
\draw[line width = .2ex] (1.5,0) -- (1.5,2.5) -- (8,2.5);
\filldraw (5.5,1.5) circle (.5ex);
\draw[line width = .2ex] (5.5,0) -- (5.5,1.5) -- (8,1.5);
\filldraw (6.5,0.5) circle (.5ex);
\draw[line width = .2ex] (6.5,0) -- (6.5,0.5) -- (8,0.5);

\end{tikzpicture}};

\node[right=1 of temp] (ans) {
\begin{tikzpicture}[scale=.5]

\draw (0,0) rectangle (8,8);

\draw (0,7) rectangle (1,8);
\draw (1,7) rectangle (2,8);
\draw (2,7) rectangle (3,8);
\draw (3,7) rectangle (4,8);

\draw (0,6) rectangle (1,7);
\draw (1,6) rectangle (2,7);

\draw (0,5) rectangle (1,6);
\draw (1,5) rectangle (2,6);

\draw (0,4) rectangle (1,5);
\draw (1,4) rectangle (2,5);

\draw[fill=lightgray] (2,6) rectangle (3,7);
\draw[fill=lightgray] (3,6) rectangle (4,7);
\draw[fill=lightgray] (5,6) rectangle (6,7);

\node at (1,-.75) {$w''= 57341268$};

\filldraw (4.5,7.5) circle (.5ex);
\draw[line width = .2ex] (4.5,0) -- (4.5,7.5) -- (8,7.5);
\filldraw (6.5,6.5) circle (.5ex);
\draw[line width = .2ex] (6.5,0) -- (6.5,6.5) -- (8,6.5);
\filldraw (2.5,5.5) circle (.5ex);
\draw[line width = .2ex] (2.5,0) -- (2.5,5.5) -- (8,5.5);
\filldraw (3.5,4.5) circle (.5ex);
\draw[line width = .2ex] (3.5,0) -- (3.5,4.5) -- (8,4.5);
\filldraw (0.5,3.5) circle (.5ex);
\draw[line width = .2ex] (0.5,0) -- (0.5,3.5) -- (8,3.5);
\filldraw (1.5,2.5) circle (.5ex);
\draw[line width = .2ex] (1.5,0) -- (1.5,2.5) -- (8,2.5);
\filldraw (5.5,1.5) circle (.5ex);
\draw[line width = .2ex] (5.5,0) -- (5.5,1.5) -- (8,1.5);
\filldraw (7.5,0.5) circle (.5ex);
\draw[line width = .2ex] (7.5,0) -- (7.5,0.5) -- (8,0.5);

\end{tikzpicture}};

\draw [->,
line join=round,
decorate, decoration={
    zigzag,
    segment length=9,
    amplitude=2,post=lineto,
    post length=2pt
}] (root) -- (temp);
\path (temp) edge[->] node[midway,above] {\ } (ans);

\end{tikzpicture}
\end{center}
The moved boxes during $D(w) \mapsto D(w'')$ are shaded gray. 
\end{example}

\begin{figure}[b]
\begin{tikzpicture}[sibling distance=2pt]
\node at (-2.6,-2.2) {$x_4$};

\node at (-6.5,-6.2) {$1$};
\node at (-4.8,-6.5) {$x_3$};
\node at (-2.5,-6.2) {$2$};

\node at (-6.6,-10.2) {$x_3$};
\node at (-4.7,-10.5) {$1$};
\node at (-2.4,-10.2) {$2$};

\node at (2.6,-2.2) {$2$};

\node at (6.5,-6.2) {$2$};
\node at (4.8,-6.5) {$x_3$};
\node at (2.5,-6.2) {$1$};

\node at (6.6,-10.2) {$2$};
\node at (4.7,-10.5) {$1$};
\node at (2.4,-10.2) {$x_3$};

\node (root) [scale=.35,sibling distance=2mm]{\begin{tikzpicture}
\draw (0,0) rectangle (8,8);

\draw (0,7) rectangle (1,8);
\draw (1,7) rectangle (2,8);
\draw (2,7) rectangle (3,8);
\draw (3,7) rectangle (4,8);

\draw (0,6) rectangle (1,7);
\draw (1,6) rectangle (2,7);

\draw (0,5) rectangle (1,6);
\draw (1,5) rectangle (2,6);

\draw (0,4) rectangle (1,5);
\draw (1,4) rectangle (2,5);
\draw (3,4) rectangle (4,5);

\draw (3,5) rectangle (4,6);

\draw (5,5) rectangle (6,6);

\draw (6,5) rectangle (7,6);

\filldraw (4.5,7.5) circle (.5ex);
\draw[line width = .2ex] (4.5,0) -- (4.5,7.5) -- (8,7.5);
\filldraw (2.5,6.5) circle (.5ex);
\draw[line width = .2ex] (2.5,0) -- (2.5,6.5) -- (8,6.5);
\filldraw (7.5,5.5) circle (.5ex);
\draw[line width = .2ex] (7.5,0) -- (7.5,5.5) -- (8,5.5);
\filldraw (5.5,4.5) circle (.5ex);
\draw[line width = .2ex] (5.5,0) -- (5.5,4.5) -- (8,4.5);
\filldraw (0.5,3.5) circle (.5ex);
\draw[line width = .2ex] (0.5,0) -- (0.5,3.5) -- (8,3.5);
\filldraw (1.5,2.5) circle (.5ex);
\draw[line width = .2ex] (1.5,0) -- (1.5,2.5) -- (8,2.5);
\filldraw (3.5,1.5) circle (.5ex);
\draw[line width = .2ex] (3.5,0) -- (3.5,1.5) -- (8,1.5);
\filldraw (6.5,0.5) circle (.5ex);
\draw[line width = .2ex] (6.5,0) -- (6.5,0.5) -- (8,0.5);

\node at (3.5,4.5) {\Huge $\mathbf{z}$};

\end{tikzpicture}};
\node[below left = 2 and 2 of root][scale=.25,sibling distance=5mm] (c1) {\begin{tikzpicture}

\draw (0,0) rectangle (8,8);

\draw (0,7) rectangle (1,8);
\draw (1,7) rectangle (2,8);
\draw (2,7) rectangle (3,8);
\draw (3,7) rectangle (4,8);

\draw (0,6) rectangle (1,7);
\draw (1,6) rectangle (2,7);

\draw (0,5) rectangle (1,6);
\draw (1,5) rectangle (2,6);

\draw (0,4) rectangle (1,5);
\draw (1,4) rectangle (2,5);

\draw (3,5) rectangle (4,6);

\draw (5,5) rectangle (6,6);

\draw (6,5) rectangle (7,6);

\node at (6.9,5.8) {\huge $\mathbf{z}$};

\filldraw (4.5,7.5) circle (.5ex);
\draw[line width = .2ex] (4.5,0) -- (4.5,7.5) -- (8,7.5);
\filldraw (2.5,6.5) circle (.5ex);
\draw[line width = .2ex] (2.5,0) -- (2.5,6.5) -- (8,6.5);
\filldraw (7.5,5.5) circle (.5ex);
\draw[line width = .2ex] (7.5,0) -- (7.5,5.5) -- (8,5.5);
\filldraw (3.5,4.5) circle (.5ex);
\draw[line width = .2ex] (3.5,0) -- (3.5,4.5) -- (8,4.5);
\filldraw (0.5,3.5) circle (.5ex);
\draw[line width = .2ex] (0.5,0) -- (0.5,3.5) -- (8,3.5);
\filldraw (1.5,2.5) circle (.5ex);
\draw[line width = .2ex] (1.5,0) -- (1.5,2.5) -- (8,2.5);
\filldraw (5.5,1.5) circle (.5ex);
\draw[line width = .2ex] (5.5,0) -- (5.5,1.5) -- (8,1.5);
\filldraw (6.5,0.5) circle (.5ex);
\draw[line width = .2ex] (6.5,0) -- (6.5,0.5) -- (8,0.5);

\end{tikzpicture}
};
\node[below left = 2  of c1] [scale=.22,sibling distance=2mm] (c2) {\begin{tikzpicture}

\draw (0,0) rectangle (8,8);

\draw (0,7) rectangle (1,8);
\draw (1,7) rectangle (2,8);
\draw (2,7) rectangle (3,8);
\draw (3,7) rectangle (4,8);
\draw[fill=lightgray]  (4,7) rectangle (5,8);
\draw[fill=lightgray]  (5,7) rectangle (6,8);

\draw (0,6) rectangle (1,7);
\draw (1,6) rectangle (2,7);

\draw (0,5) rectangle (1,6);
\draw (1,5) rectangle (2,6);

\draw (0,4) rectangle (1,5);
\draw (1,4) rectangle (2,5);

\draw (3,5) rectangle (4,6);

\node at (3.9,5.8) {\huge $\mathbf{z}$};

\filldraw (6.5,7.5) circle (.5ex);
\draw[line width = .2ex] (6.5,0) -- (6.5,7.5) -- (8,7.5);
\filldraw (2.5,6.5) circle (.5ex);
\draw[line width = .2ex] (2.5,0) -- (2.5,6.5) -- (8,6.5);
\filldraw (4.5,5.5) circle (.5ex);
\draw[line width = .2ex] (4.5,0) -- (4.5,5.5) -- (8,5.5);
\filldraw (3.5,4.5) circle (.5ex);
\draw[line width = .2ex] (3.5,0) -- (3.5,4.5) -- (8,4.5);
\filldraw (0.5,3.5) circle (.5ex);
\draw[line width = .2ex] (0.5,0) -- (0.5,3.5) -- (8,3.5);
\filldraw (1.5,2.5) circle (.5ex);
\draw[line width = .2ex] (1.5,0) -- (1.5,2.5) -- (8,2.5);
\filldraw (5.5,1.5) circle (.5ex);
\draw[line width = .2ex] (5.5,0) -- (5.5,1.5) -- (8,1.5);
\filldraw (7.5,0.5) circle (.5ex);
\draw[line width = .2ex] (7.5,0) -- (7.5,0.5) -- (8,0.5);

\end{tikzpicture}};
\node[below = 2  of c1][scale=.22,sibling distance=1mm](c3) {\begin{tikzpicture}

\draw (0,0) rectangle (8,8);

\draw (0,7) rectangle (1,8);
\draw (1,7) rectangle (2,8);
\draw (2,7) rectangle (3,8);
\draw (3,7) rectangle (4,8);

\draw (0,6) rectangle (1,7);
\draw (1,6) rectangle (2,7);

\draw (0,5) rectangle (1,6);
\draw (1,5) rectangle (2,6);

\draw (0,4) rectangle (1,5);
\draw (1,4) rectangle (2,5);

\draw (3,5) rectangle (4,6);

\draw (5,5) rectangle (6,6);

\node at (5.5,5.8) {\huge $\mathbf{z}$};

\filldraw (4.5,7.5) circle (.5ex);
\draw[line width = .2ex] (4.5,0) -- (4.5,7.5) -- (8,7.5);
\filldraw (2.5,6.5) circle (.5ex);
\draw[line width = .2ex] (2.5,0) -- (2.5,6.5) -- (8,6.5);
\filldraw (6.5,5.5) circle (.5ex);
\draw[line width = .2ex] (6.5,0) -- (6.5,5.5) -- (8,5.5);
\filldraw (3.5,4.5) circle (.5ex);
\draw[line width = .2ex] (3.5,0) -- (3.5,4.5) -- (8,4.5);
\filldraw (0.5,3.5) circle (.5ex);
\draw[line width = .2ex] (0.5,0) -- (0.5,3.5) -- (8,3.5);
\filldraw (1.5,2.5) circle (.5ex);
\draw[line width = .2ex] (1.5,0) -- (1.5,2.5) -- (8,2.5);
\filldraw (5.5,1.5) circle (.5ex);
\draw[line width = .2ex] (5.5,0) -- (5.5,1.5) -- (8,1.5);
\filldraw (7.5,0.5) circle (.5ex);
\draw[line width = .2ex] (7.5,0) -- (7.5,0.5) -- (8,0.5);
\end{tikzpicture}};
\node[below right= 2  of c1] [scale=.22,sibling distance=2mm] (c4) {\begin{tikzpicture}

\draw (0,0) rectangle (8,8);

\draw (0,7) rectangle (1,8);
\draw (1,7) rectangle (2,8);
\draw (2,7) rectangle (3,8);
\draw (3,7) rectangle (4,8);

\draw (0,6) rectangle (1,7);
\draw (1,6) rectangle (2,7);

\draw (0,5) rectangle (1,6);
\draw (1,5) rectangle (2,6);

\draw (0,4) rectangle (1,5);
\draw (1,4) rectangle (2,5);

\draw[fill=lightgray] (2,6) rectangle (3,7);

\draw[fill=lightgray] (3,6) rectangle (4,7);

\draw[fill=lightgray] (5,6) rectangle (6,7);

\node at (5.1,6.8) {\huge $\mathbf{z}$};

\filldraw (4.5,7.5) circle (.5ex);
\draw[line width = .2ex] (4.5,0) -- (4.5,7.5) -- (8,7.5);
\filldraw (6.5,6.5) circle (.5ex);
\draw[line width = .2ex] (6.5,0) -- (6.5,6.5) -- (8,6.5);
\filldraw (2.5,5.5) circle (.5ex);
\draw[line width = .2ex] (2.5,0) -- (2.5,5.5) -- (8,5.5);
\filldraw (3.5,4.5) circle (.5ex);
\draw[line width = .2ex] (3.5,0) -- (3.5,4.5) -- (8,4.5);
\filldraw (0.5,3.5) circle (.5ex);
\draw[line width = .2ex] (0.5,0) -- (0.5,3.5) -- (8,3.5);
\filldraw (1.5,2.5) circle (.5ex);
\draw[line width = .2ex] (1.5,0) -- (1.5,2.5) -- (8,2.5);
\filldraw (5.5,1.5) circle (.5ex);
\draw[line width = .2ex] (5.5,0) -- (5.5,1.5) -- (8,1.5);
\filldraw (7.5,0.5) circle (.5ex);
\draw[line width = .2ex] (7.5,0) -- (7.5,0.5) -- (8,0.5);
\end{tikzpicture}};
\node[below left = 2  of c3]  [scale=.22,sibling distance=2mm](c5) {\begin{tikzpicture}

\draw (0,0) rectangle (8,8);

\draw (0,7) rectangle (1,8);
\draw (1,7) rectangle (2,8);
\draw (2,7) rectangle (3,8);
\draw (3,7) rectangle (4,8);

\draw (0,6) rectangle (1,7);
\draw (1,6) rectangle (2,7);

\draw (0,5) rectangle (1,6);
\draw (1,5) rectangle (2,6);

\draw (0,4) rectangle (1,5);
\draw (1,4) rectangle (2,5);

\draw (3,5) rectangle (4,6);

\node at (3.9,5.8) {\huge $\mathbf{z}$};

\filldraw (4.5,7.5) circle (.5ex);
\draw[line width = .2ex] (4.5,0) -- (4.5,7.5) -- (8,7.5);
\filldraw (2.5,6.5) circle (.5ex);
\draw[line width = .2ex] (2.5,0) -- (2.5,6.5) -- (8,6.5);
\filldraw (5.5,5.5) circle (.5ex);
\draw[line width = .2ex] (5.5,0) -- (5.5,5.5) -- (8,5.5);
\filldraw (3.5,4.5) circle (.5ex);
\draw[line width = .2ex] (3.5,0) -- (3.5,4.5) -- (8,4.5);
\filldraw (0.5,3.5) circle (.5ex);
\draw[line width = .2ex] (0.5,0) -- (0.5,3.5) -- (8,3.5);
\filldraw (1.5,2.5) circle (.5ex);
\draw[line width = .2ex] (1.5,0) -- (1.5,2.5) -- (8,2.5);
\filldraw (6.5,1.5) circle (.5ex);
\draw[line width = .2ex] (6.5,0) -- (6.5,1.5) -- (8,1.5);
\filldraw (7.5,0.5) circle (.5ex);
\draw[line width = .2ex] (7.5,0) -- (7.5,0.5) -- (8,0.5);
\end{tikzpicture}};
\node[below = 2  of c3] [scale=.22,sibling distance=2mm] (c6) {\begin{tikzpicture}

\draw (0,0) rectangle (8,8);

\draw (0,7) rectangle (1,8);
\draw (1,7) rectangle (2,8);
\draw (2,7) rectangle (3,8);
\draw (3,7) rectangle (4,8);
\draw[fill=lightgray] (4,7) rectangle (5,8);

\draw (0,6) rectangle (1,7);
\draw (1,6) rectangle (2,7);

\draw (0,5) rectangle (1,6);
\draw (1,5) rectangle (2,6);

\draw (0,4) rectangle (1,5);
\draw (1,4) rectangle (2,5);

\draw (3,5) rectangle (4,6);

\node at (3.5,5.8) {\huge $\mathbf{z}$};

\filldraw (5.5,7.5) circle (.5ex);
\draw[line width = .2ex] (5.5,0) -- (5.5,7.5) -- (8,7.5);
\filldraw (2.5,6.5) circle (.5ex);
\draw[line width = .2ex] (2.5,0) -- (2.5,6.5) -- (8,6.5);
\filldraw (4.5,5.5) circle (.5ex);
\draw[line width = .2ex] (4.5,0) -- (4.5,5.5) -- (8,5.5);
\filldraw (3.5,4.5) circle (.5ex);
\draw[line width = .2ex] (3.5,0) -- (3.5,4.5) -- (8,4.5);
\filldraw (0.5,3.5) circle (.5ex);
\draw[line width = .2ex] (0.5,0) -- (0.5,3.5) -- (8,3.5);
\filldraw (1.5,2.5) circle (.5ex);
\draw[line width = .2ex] (1.5,0) -- (1.5,2.5) -- (8,2.5);
\filldraw (6.5,1.5) circle (.5ex);
\draw[line width = .2ex] (6.5,0) -- (6.5,1.5) -- (8,1.5);
\filldraw (7.5,0.5) circle (.5ex);
\draw[line width = .2ex] (7.5,0) -- (7.5,0.5) -- (8,0.5);
\end{tikzpicture}};
\node[below right= 2  of c3]  [scale=.22,sibling distance=2mm](c7) {\begin{tikzpicture}

\draw (0,0) rectangle (8,8);

\draw (0,7) rectangle (1,8);
\draw (1,7) rectangle (2,8);
\draw (2,7) rectangle (3,8);
\draw (3,7) rectangle (4,8);

\draw (0,6) rectangle (1,7);
\draw (1,6) rectangle (2,7);
\draw[fill=lightgray] (2,6) rectangle (3,7);
\draw[fill=lightgray] (3,6) rectangle (4,7);

\draw (0,5) rectangle (1,6);
\draw (1,5) rectangle (2,6);

\draw (0,4) rectangle (1,5);
\draw (1,4) rectangle (2,5);



\filldraw (4.5,7.5) circle (.5ex);
\draw[line width = .2ex] (4.5,0) -- (4.5,7.5) -- (8,7.5);
\filldraw (5.5,6.5) circle (.5ex);
\draw[line width = .2ex] (5.5,0) -- (5.5,6.5) -- (8,6.5);
\filldraw (2.5,5.5) circle (.5ex);
\draw[line width = .2ex] (2.5,0) -- (2.5,5.5) -- (8,5.5);
\filldraw (3.5,4.5) circle (.5ex);
\draw[line width = .2ex] (3.5,0) -- (3.5,4.5) -- (8,4.5);
\filldraw (0.5,3.5) circle (.5ex);
\draw[line width = .2ex] (0.5,0) -- (0.5,3.5) -- (8,3.5);
\filldraw (1.5,2.5) circle (.5ex);
\draw[line width = .2ex] (1.5,0) -- (1.5,2.5) -- (8,2.5);
\filldraw (6.5,1.5) circle (.5ex);
\draw[line width = .2ex] (6.5,0) -- (6.5,1.5) -- (8,1.5);
\filldraw (7.5,0.5) circle (.5ex);
\draw[line width = .2ex] (7.5,0) -- (7.5,0.5) -- (8,0.5);
\end{tikzpicture}};

\node[below right = 2 and 2 of root]  [scale=.25,sibling distance=2mm](d1) {\begin{tikzpicture}
\draw (0,0) rectangle (8,8);

\draw (0,7) rectangle (1,8);
\draw (1,7) rectangle (2,8);
\draw (2,7) rectangle (3,8);
\draw (3,7) rectangle (4,8);

\draw (0,6) rectangle (1,7);
\draw (1,6) rectangle (2,7);

\draw (0,5) rectangle (1,6);
\draw (1,5) rectangle (2,6);

\draw (0,4) rectangle (1,5);
\draw (1,4) rectangle (2,5);
\draw[fill=lightgray] (2,5) rectangle (3,6);

\draw[fill=lightgray] (2,6) rectangle (3,7);

\draw (5,5) rectangle (6,6);

\draw (6,5) rectangle (7,6);

\node at (6.1,5.8) {\huge $\mathbf{z}$};

\filldraw (4.5,7.5) circle (.5ex);
\draw[line width = .2ex] (4.5,0) -- (4.5,7.5) -- (8,7.5);
\filldraw (3.5,6.5) circle (.5ex);
\draw[line width = .2ex] (3.5,0) -- (3.5,6.5) -- (8,6.5);
\filldraw (7.5,5.5) circle (.5ex);
\draw[line width = .2ex] (7.5,0) -- (7.5,5.5) -- (8,5.5);
\filldraw (2.5,4.5) circle (.5ex);
\draw[line width = .2ex] (2.5,0) -- (2.5,4.5) -- (8,4.5);
\filldraw (0.5,3.5) circle (.5ex);
\draw[line width = .2ex] (0.5,0) -- (0.5,3.5) -- (8,3.5);
\filldraw (1.5,2.5) circle (.5ex);
\draw[line width = .2ex] (1.5,0) -- (1.5,2.5) -- (8,2.5);
\filldraw (5.5,1.5) circle (.5ex);
\draw[line width = .2ex] (5.5,0) -- (5.5,1.5) -- (8,1.5);
\filldraw (6.5,0.5) circle (.5ex);
\draw[line width = .2ex] (6.5,0) -- (6.5,0.5) -- (8,0.5);

\end{tikzpicture}};
\node[below left = 2  of d1] [scale=.22,sibling distance=2mm] (d2) {\begin{tikzpicture}
\draw (0,0) rectangle (8,8);

\draw (0,7) rectangle (1,8);
\draw (1,7) rectangle (2,8);
\draw (2,7) rectangle (3,8);
\draw (3,7) rectangle (4,8);
\draw[fill=lightgray] (4,7) rectangle (5,8);
\draw[fill=lightgray] (5,7) rectangle (6,8);

\draw (0,6) rectangle (1,7);
\draw (1,6) rectangle (2,7);

\draw (0,5) rectangle (1,6);
\draw (1,5) rectangle (2,6);

\draw (0,4) rectangle (1,5);
\draw (1,4) rectangle (2,5);

\draw (2,5) rectangle (3,6);
\draw (2,6) rectangle (3,7);

\filldraw (6.5,7.5) circle (.5ex);
\draw[line width = .2ex] (6.5,0) -- (6.5,7.5) -- (8,7.5);
\filldraw (3.5,6.5) circle (.5ex);
\draw[line width = .2ex] (3.5,0) -- (3.5,6.5) -- (8,6.5);
\filldraw (4.5,5.5) circle (.5ex);
\draw[line width = .2ex] (4.5,0) -- (4.5,5.5) -- (8,5.5);
\filldraw (2.5,4.5) circle (.5ex);
\draw[line width = .2ex] (2.5,0) -- (2.5,4.5) -- (8,4.5);
\filldraw (0.5,3.5) circle (.5ex);
\draw[line width = .2ex] (0.5,0) -- (0.5,3.5) -- (8,3.5);
\filldraw (1.5,2.5) circle (.5ex);
\draw[line width = .2ex] (1.5,0) -- (1.5,2.5) -- (8,2.5);
\filldraw (5.5,1.5) circle (.5ex);
\draw[line width = .2ex] (5.5,0) -- (5.5,1.5) -- (8,1.5);
\filldraw (7.5,0.5) circle (.5ex);
\draw[line width = .2ex] (7.5,0) -- (7.5,0.5) -- (8,0.5);
\end{tikzpicture}};
\node[below = 2  of d1]  [scale=.22,sibling distance=2mm](d3) {\begin{tikzpicture}
\draw (0,0) rectangle (8,8);

\draw (0,7) rectangle (1,8);
\draw (1,7) rectangle (2,8);
\draw (2,7) rectangle (3,8);
\draw (3,7) rectangle (4,8);

\draw (0,6) rectangle (1,7);
\draw (1,6) rectangle (2,7);

\draw (0,5) rectangle (1,6);
\draw (1,5) rectangle (2,6);

\draw (0,4) rectangle (1,5);
\draw (1,4) rectangle (2,5);
\draw (2,5) rectangle (3,6);

\draw (2,6) rectangle (3,7);

\draw (5,5) rectangle (6,6);
\node at (5.5,5.8) {\huge $\mathbf{z}$};

\filldraw (4.5,7.5) circle (.5ex);
\draw[line width = .2ex] (4.5,0) -- (4.5,7.5) -- (8,7.5);
\filldraw (3.5,6.5) circle (.5ex);
\draw[line width = .2ex] (3.5,0) -- (3.5,6.5) -- (8,6.5);
\filldraw (6.5,5.5) circle (.5ex);
\draw[line width = .2ex] (6.5,0) -- (6.5,5.5) -- (8,5.5);
\filldraw (2.5,4.5) circle (.5ex);
\draw[line width = .2ex] (2.5,0) -- (2.5,4.5) -- (8,4.5);
\filldraw (0.5,3.5) circle (.5ex);
\draw[line width = .2ex] (0.5,0) -- (0.5,3.5) -- (8,3.5);
\filldraw (1.5,2.5) circle (.5ex);
\draw[line width = .2ex] (1.5,0) -- (1.5,2.5) -- (8,2.5);
\filldraw (5.5,1.5) circle (.5ex);
\draw[line width = .2ex] (5.5,0) -- (5.5,1.5) -- (8,1.5);
\filldraw (7.5,0.5) circle (.5ex);
\draw[line width = .2ex] (7.5,0) -- (7.5,0.5) -- (8,0.5);
\end{tikzpicture}};
\node[below right= 2  of d1]  [scale=.22,sibling distance=2mm](d4) {\begin{tikzpicture}
\draw (0,0) rectangle (8,8);

\draw (0,7) rectangle (1,8);
\draw (1,7) rectangle (2,8);
\draw (2,7) rectangle (3,8);
\draw (3,7) rectangle (4,8);

\draw (0,6) rectangle (1,7);
\draw (1,6) rectangle (2,7);
\draw[fill=lightgray] (3,6) rectangle (4,7);
\draw[fill=lightgray] (5,6) rectangle (6,7);

\draw (0,5) rectangle (1,6);
\draw (1,5) rectangle (2,6);

\draw (0,4) rectangle (1,5);
\draw (1,4) rectangle (2,5);
\draw (2,5) rectangle (3,6);

\draw (2,6) rectangle (3,7);

\node at (5.1,6.8) {\huge $\mathbf{z}$};

\filldraw (4.5,7.5) circle (.5ex);
\draw[line width = .2ex] (4.5,0) -- (4.5,7.5) -- (8,7.5);
\filldraw (6.5,6.5) circle (.5ex);
\draw[line width = .2ex] (6.5,0) -- (6.5,6.5) -- (8,6.5);
\filldraw (3.5,5.5) circle (.5ex);
\draw[line width = .2ex] (3.5,0) -- (3.5,5.5) -- (8,5.5);
\filldraw (2.5,4.5) circle (.5ex);
\draw[line width = .2ex] (2.5,0) -- (2.5,4.5) -- (8,4.5);
\filldraw (0.5,3.5) circle (.5ex);
\draw[line width = .2ex] (0.5,0) -- (0.5,3.5) -- (8,3.5);
\filldraw (1.5,2.5) circle (.5ex);
\draw[line width = .2ex] (1.5,0) -- (1.5,2.5) -- (8,2.5);
\filldraw (5.5,1.5) circle (.5ex);
\draw[line width = .2ex] (5.5,0) -- (5.5,1.5) -- (8,1.5);
\filldraw (7.5,0.5) circle (.5ex);
\draw[line width = .2ex] (7.5,0) -- (7.5,0.5) -- (8,0.5);
\end{tikzpicture}};
\node[below left = 2  of d3] (d5) [scale=.22,sibling distance=2mm] {\begin{tikzpicture}
\draw (0,0) rectangle (8,8);

\draw (0,7) rectangle (1,8);
\draw (1,7) rectangle (2,8);
\draw (2,7) rectangle (3,8);
\draw (3,7) rectangle (4,8);

\draw (0,6) rectangle (1,7);
\draw (1,6) rectangle (2,7);

\draw (0,5) rectangle (1,6);
\draw (1,5) rectangle (2,6);

\draw (0,4) rectangle (1,5);
\draw (1,4) rectangle (2,5);
\draw (2,5) rectangle (3,6);

\draw (2,6) rectangle (3,7);


\filldraw (4.5,7.5) circle (.5ex);
\draw[line width = .2ex] (4.5,0) -- (4.5,7.5) -- (8,7.5);
\filldraw (3.5,6.5) circle (.5ex);
\draw[line width = .2ex] (3.5,0) -- (3.5,6.5) -- (8,6.5);
\filldraw (5.5,5.5) circle (.5ex);
\draw[line width = .2ex] (5.5,0) -- (5.5,5.5) -- (8,5.5);
\filldraw (2.5,4.5) circle (.5ex);
\draw[line width = .2ex] (2.5,0) -- (2.5,4.5) -- (8,4.5);
\filldraw (0.5,3.5) circle (.5ex);
\draw[line width = .2ex] (0.5,0) -- (0.5,3.5) -- (8,3.5);
\filldraw (1.5,2.5) circle (.5ex);
\draw[line width = .2ex] (1.5,0) -- (1.5,2.5) -- (8,2.5);
\filldraw (6.5,1.5) circle (.5ex);
\draw[line width = .2ex] (6.5,0) -- (6.5,1.5) -- (8,1.5);
\filldraw (7.5,0.5) circle (.5ex);
\draw[line width = .2ex] (7.5,0) -- (7.5,0.5) -- (8,0.5);
\end{tikzpicture}};
\node[below = 2  of d3] (d6)  [scale=.22,sibling distance=2mm]{\begin{tikzpicture}
\draw (0,0) rectangle (8,8);

\draw (0,7) rectangle (1,8);
\draw (1,7) rectangle (2,8);
\draw (2,7) rectangle (3,8);
\draw (3,7) rectangle (4,8);
\draw[fill=lightgray] (4,7) rectangle (5,8);

\draw (0,6) rectangle (1,7);
\draw (1,6) rectangle (2,7);

\draw (0,5) rectangle (1,6);
\draw (1,5) rectangle (2,6);

\draw (0,4) rectangle (1,5);
\draw (1,4) rectangle (2,5);
\draw (2,5) rectangle (3,6);

\draw (2,6) rectangle (3,7);

\filldraw (5.5,7.5) circle (.5ex);
\draw[line width = .2ex] (5.5,0) -- (5.5,7.5) -- (8,7.5);
\filldraw (3.5,6.5) circle (.5ex);
\draw[line width = .2ex] (3.5,0) -- (3.5,6.5) -- (8,6.5);
\filldraw (4.5,5.5) circle (.5ex);
\draw[line width = .2ex] (4.5,0) -- (4.5,5.5) -- (8,5.5);
\filldraw (2.5,4.5) circle (.5ex);
\draw[line width = .2ex] (2.5,0) -- (2.5,4.5) -- (8,4.5);
\filldraw (0.5,3.5) circle (.5ex);
\draw[line width = .2ex] (0.5,0) -- (0.5,3.5) -- (8,3.5);
\filldraw (1.5,2.5) circle (.5ex);
\draw[line width = .2ex] (1.5,0) -- (1.5,2.5) -- (8,2.5);
\filldraw (6.5,1.5) circle (.5ex);
\draw[line width = .2ex] (6.5,0) -- (6.5,1.5) -- (8,1.5);
\filldraw (7.5,0.5) circle (.5ex);
\draw[line width = .2ex] (7.5,0) -- (7.5,0.5) -- (8,0.5);
\end{tikzpicture}};
\node[below right= 2  of d3] [scale=.22,sibling distance=2mm] (d7) {\begin{tikzpicture}
\draw (0,0) rectangle (8,8);

\draw (0,7) rectangle (1,8);
\draw (1,7) rectangle (2,8);
\draw (2,7) rectangle (3,8);
\draw (3,7) rectangle (4,8);

\draw (0,6) rectangle (1,7);
\draw (1,6) rectangle (2,7);

\draw (0,5) rectangle (1,6);
\draw (1,5) rectangle (2,6);

\draw (0,4) rectangle (1,5);
\draw (1,4) rectangle (2,5);
\draw (2,5) rectangle (3,6);

\draw (2,6) rectangle (3,7);
\draw[fill=lightgray] (3,6) rectangle (4,7);

\filldraw (4.5,7.5) circle (.5ex);
\draw[line width = .2ex] (4.5,0) -- (4.5,7.5) -- (8,7.5);
\filldraw (5.5,6.5) circle (.5ex);
\draw[line width = .2ex] (5.5,0) -- (5.5,6.5) -- (8,6.5);
\filldraw (3.5,5.5) circle (.5ex);
\draw[line width = .2ex] (3.5,0) -- (3.5,5.5) -- (8,5.5);
\filldraw (2.5,4.5) circle (.5ex);
\draw[line width = .2ex] (2.5,0) -- (2.5,4.5) -- (8,4.5);
\filldraw (0.5,3.5) circle (.5ex);
\draw[line width = .2ex] (0.5,0) -- (0.5,3.5) -- (8,3.5);
\filldraw (1.5,2.5) circle (.5ex);
\draw[line width = .2ex] (1.5,0) -- (1.5,2.5) -- (8,2.5);
\filldraw (6.5,1.5) circle (.5ex);
\draw[line width = .2ex] (6.5,0) -- (6.5,1.5) -- (8,1.5);
\filldraw (7.5,0.5) circle (.5ex);
\draw[line width = .2ex] (7.5,0) -- (7.5,0.5) -- (8,0.5);
\end{tikzpicture}};
  \begin{scope}[nodes = {draw = none}]
    \path (root) edge[->] (c1)
     (root) edge[->] (d1)
    (c1) edge[->] (c2) 
    (c1) edge[->] (c3)
    (c1) edge[->] (c4)
    (c3) edge[->] (c5)
    (c3) edge[->] (c6)
    (c3) edge[->] (c7)
    (d1) edge[->] (d2)
    (d1) edge[->] (d3)
    (d1) edge[->] (d4)
    (d3) edge[->] (d5)
    (d3) edge[->] (d6)
    (d3) edge[->] (d7)
      ;
  \end{scope}

\end{tikzpicture}

\caption{$\mathcal{T}(w)$ for $w = 53861247$ where the accessible boxes are marked with $\mathbf{z}$ and those boxes of the parent which moved are shaded gray.}
\label{fig:transTree}
\end{figure}
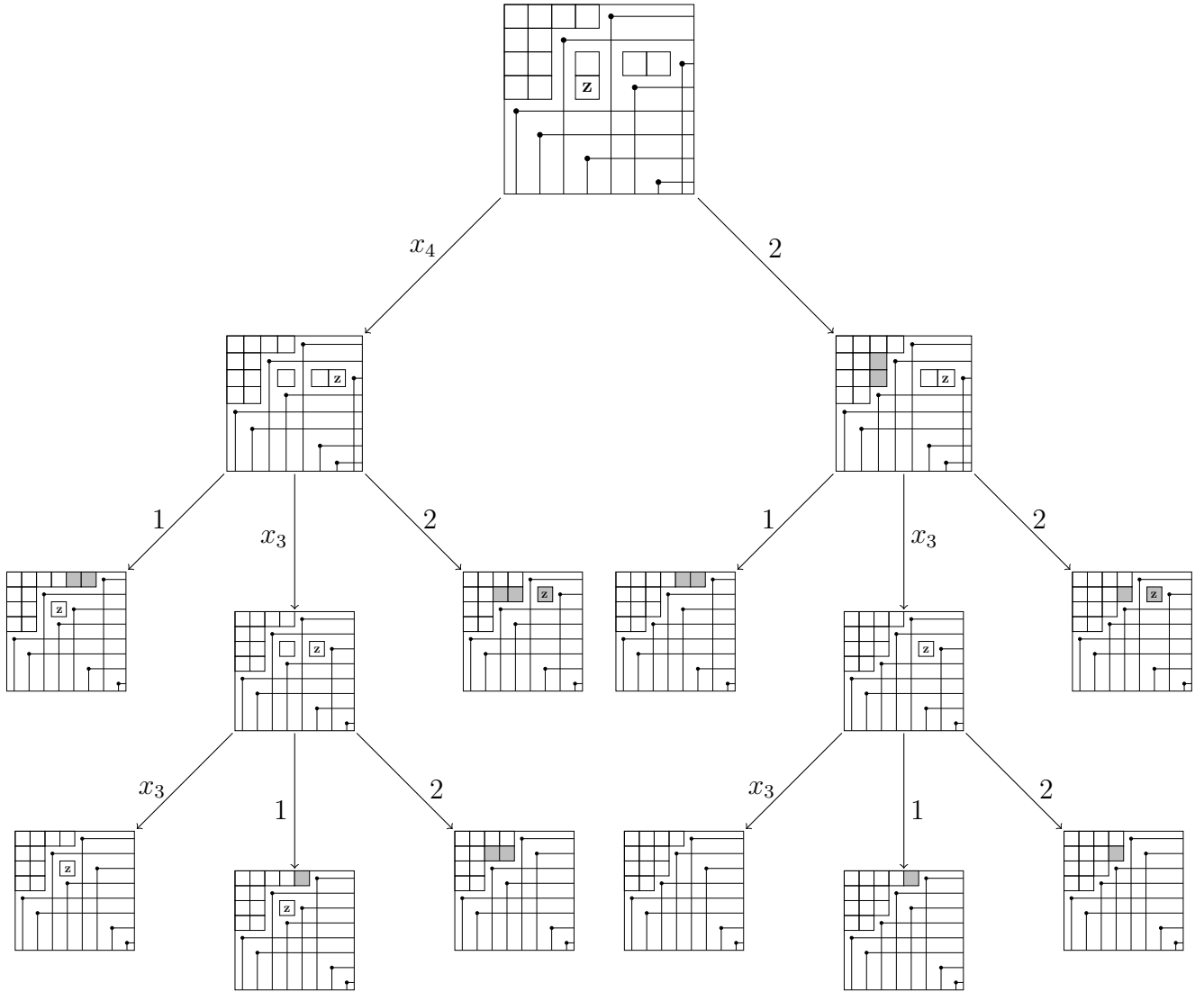

\begin{example}
Let $w=53861247$.
Using ${\mathcal T}(w)$ from Figure~\ref{fig:transTree}, we compute 
\begin{multline}\nonumber
{\mathfrak S}_{w}=x_4\cdot{\mathfrak S}_{73541268}
+x_4\cdot{\mathfrak S}_{57341268}
+x_{3}^2x_4\cdot{\mathfrak S}_{53641278}
+x_{3}x_4\cdot{\mathfrak S}_{63541278}
+x_{3}x_4\cdot{\mathfrak S}_{56341278}\\
+{\mathfrak S}_{74531268}
+{\mathfrak S}_{57431268}
+x_{3}^2\cdot{\mathfrak S}_{54631278}
+x_{3}\cdot{\mathfrak S}_{64531278}
+x_{3}\cdot{\mathfrak S}_{56431278}.
\end{multline}
For instance, $c_{(4,2,5,3),w}:=[x_1^4 x_2^2 x_3^5 x_4^3]{\mathfrak S}_{w}=1$ is witnessed by
\begin{itemize}
\item the path $w\stackrel{x_4}{\longrightarrow}\bullet\stackrel{x_3}{\longrightarrow}\bullet
\stackrel{x_3}{\longrightarrow} u=53641278$, and
\item the semistandard tableau
 \[T=\young(1111,223,33,44) \text{\ of shape $\lambda(u)$, flagged by $\phi(u)=(1,3,4,4)$.}\]
\end{itemize}
Proposition~\ref{lemma:transbasic} below formalizes a rule for $c_{\alpha,w}$ in terms of such pairs.
\end{example}

\subsection{Proof of $\#{\sf P}$-ness} 

The technical core of our proof of Theorem~\ref{thm:second} is to show:

\begin{theorem}\label{prop:schubinsharpP}
The problem of computing $c_{\alpha,w}$, given input $\alpha$ and ${\sf code}(w)$, is in $\#{\sf P}$.
\end{theorem}

Define $X$ to be the set consisting of pairs $(S,R)$ where:
\begin{itemize}
\item[(X.1)] $S=(s_1,\ldots,s_h)$, $s_t\in [L]\cup \{(x_k,m_t) \ :  \ k\in[L], m_t\in\mathbb{Z}_{>0}\}$ such that 
if $s_t=(x_k,m_t)$ then $s_{t+1}\neq(x_k,m_{t+1})$ for $t<h$, and 
\item[(X.2)] $R=(r_{ij})_{1\leq i,j\leq L}$, where $r_{ij}\in\mathbb{Z}_{\geq 0}$.
\end{itemize}

Fix $w\in S_{\infty}$ and a vexillary permutation $v\in S_{\infty}$. 
A \emph{$(w,v)$-transition string} is a sequence $S=(s_1,\ldots,s_h)$ 
satisfying (X.1) such that if we interpret $i$ as $\bullet \stackrel{i}{\longrightarrow}\bullet$ and
 $(x_k,m_t)$ as $\bullet \stackrel{x_k}{\longrightarrow}\bullet \cdots \bullet\stackrel{x_k}{\longrightarrow} \bullet$ ($m_t$-times) then $S$ describes a path from $w$ to (a leaf labelled by) $v$ in ${\mathcal T}(w)$. 
 Let ${\sf Trans}(w,v)$ be the set of such sequences.

The \emph{deletion weight} of $S\in {\sf Trans}(w,v)$ is
\[{\sf delwt}(S)=\displaystyle\sum m_t\cdot{\vec{e_r}},\]
where the summation is over $1\leq t\leq h$ such that $s_t=(x_r,m_t)\in S$ for some $r\in [L]$ (depending on $t$).  
Here ${\vec{e_r}} \in \mathbb{Z}_{\geq 0}^L$ is the $r$-th standard basis vector and $L$ is the length of ${\sf code}(w)=(c_1,c_2,\ldots,c_L)$.

\begin{example}
In Figure \ref{fig:transTree} we read the 
$(w=53861247, v=54631278)$-transition string $S=(2,(x_3,2))$ as the path
$w\stackrel{2}{\longrightarrow}\bullet \stackrel{x_3}{\longrightarrow}\bullet \stackrel{x_3}{\longrightarrow}
v$. Here, ${\sf delwt}(S)=(0,0,2,0)$.
\end{example}

Suppose $T$ is a tableau of shape $\lambda=(\lambda_1\geq \lambda_2\geq\ldots\geq \lambda_L\geq 0)$, with entries in $[L]$ and weakly increasing along rows.   Define 
\[R(T)=(r_{ij})_{1\leq i,j\leq L}\]
to be the $L\times L$ matrix where $r_{ij}$ is the number of $j$'s in row $i$ of $T$.
$R(T)$ encodes $T$. As pointed out in (a preprint version of) \cite{Narayanan}, $T$ might have exponentially many (in $L$) boxes, whereas $R(T)$ is a $O(L^2)$ description of $T$.

\begin{example}
If $\lambda=(4,3,1,0,0)$ and
\[T=\young(1123,245,4) \longleftrightarrow
R(T)=\left(\begin{matrix}
2 & 1 & 1 & 0& 0\\
0 & 1 & 0 & 1& 1\\
0& 0 & 0 &1 & 0\\
0 & 0 & 0 & 0 & 0\\
0 & 0 & 0 & 0 & 0
\end{matrix}
\right)
\]
\end{example}

Let $X_{\alpha,w}=\{(S,R(T))\}\subseteq X$ such that the following hold:
\begin{itemize}
    \item[(X.1')] $S \in {\sf Trans}(w,v)$, 
    \item[(X.2')] $T\in{\sf SSYT}(\lambda(v),\phi(v))$, and
    \item[(X.3')] ${\sf delwt}(S)+{\sf content}(T) =\alpha$.
\end{itemize}

\begin{proposition} 
\label{lemma:transbasic} $c_{\alpha,w}=\#X_{\alpha,w}$.
\end{proposition}
\begin{proof}
Iterating (\ref{eqn:thetrans}), 
\[{\mathfrak S}_w=\sum_{\text{vexillary  } v\in S_{\infty}}  \sum_{S\in {\sf Trans}(w,v)} x^{{\sf delwt}(S)} {\mathfrak S}_{v}.\]
Hence
\begin{equation}
\label{eqn:Jul16abc}
c_{\alpha,w}=
\sum_{\text{vexillary  } v\in S_{\infty}}  \sum_{S\in {\sf Trans}(w,v)} [x^{\alpha}] x^{{\sf delwt}(S)} {\mathfrak S}_{v}.
\end{equation}
The result then follows from
by (\ref{eqn:Sept28abc}), (\ref{eqn:flaggen}), and (\ref{eqn:Jul16abc}) combined.\end{proof}

\begin{proposition}[cf.~\cite{LS:transition}\label{prop:childCodePoly}]
Let ${\sf code}(w)=(c_1,\ldots,c_L)$. Suppose $D(w')$ is obtained from $D(w)$ using move \emph{(T.1)} and $D(w'')$ is obtained from $D(w)$ with move \emph{(T.2)} for a pivot in row $i$. There is an $O(L^{2})$-time algorithm to compute
\begin{itemize}
\item[\emph{(I)}] ${\sf code}(w')=(c_1,\ldots,c_{r-1},c_{r}-1,c_{r+1},\ldots,c_L)$ and 
\item[\emph{(II)}] ${\sf code}(w'')=(c_1,\ldots,c_{i-1},c_{i}+b,c_{i+1},\ldots,c_{r-1},c_r-b,c_{r+1},\ldots, c_L)$,
for some $b\in\mathbb{Z}_{>0}$.
\end{itemize}
\end{proposition}
\begin{proof}
By Proposition~\ref{prop:accBoxPoly}, determine $\mathbf{z}_w:=(r,c)$ in $O(L^2)$-time.

For (I), $D(w')$ is obtained from $D(w)$ by deleting $\mathbf{z}_w$; so the expression in (I) is clear.  

For (II), using Proposition \ref{prop:wPoly}, we can find $\mathbf{x}=(i,w(i))$, in $O(L^{2})$-time; this is our (T.2) pivot.
Notice that row $r$ of $D(w)\cap {\mathcal R}$ is nonempty (it contains $\mathbf{z}_w=(r,c)$); let $b$ be the number of boxes in this row.
It is straightforward from the graphical description of ${\mathcal R}$ 
in terms of Rothe diagrams that each row of
$D(w)\cap {\mathcal R}$ either has zero boxes or $b>0$ boxes. Moreover, the $d$-th box (say, from the left) of each row are in the same column. 

Suppose $j_1,\ldots,j_m\in [i+1,r]$ index the rows where $D(w)\cap {\mathcal R}\neq\emptyset$ (and thus has $b$ boxes). (T.2) moves the $b$ boxes of $j_1$ to
row $i$ and moves the $b$ boxes of $j_q$ to row $j_{q-1}$ for $q=2,\ldots,m$. As explained above $j_m=r$, so (T.2) moves no boxes into row $r$. 
Thus row $r$ of $D(w'')\cap {\mathcal R}$
has zero boxes. 

It remains to compute $b$ in $O(L^{2})$-time.
Using Proposition~\ref{prop:wPoly} compute, in $O(L^2)$-time, 
\[m:=\#\{h<r \ : \ w(h)<w(i)\}.\] 
Clearly $b=c_r-[(w(i)-1)-m]$.
\end{proof}

Let $s_t=(x_r,m_t)$, as in (X.1), be a valid (multi)-deletion move on $u\in \TT(w)$. Let $u^{\langle m\rangle}\in\TT(w)$ be defined by 
$u \stackrel{x_k}{\longrightarrow}\bullet \cdots \bullet\stackrel{x_k}{\longrightarrow} u^{\langle m_t\rangle}$ ($m_t$-times). 

\begin{proposition}\label{prop:movePoly}
Suppose $u\in \TT(w)$ where ${\sf code}(u)=({\widetilde c}_1,\ldots,{\widetilde c}_{L'})$.  
Let $s_t=(x_k,m_t)$ or $s_t=i$ be as in (X.1). 
Given input ${\sf code}(u)$ and $s_t$,
there is an $O(L^2)$ algorithm to 
respectively determine if $u \stackrel{x_k}{\longrightarrow}\bullet \cdots \bullet\stackrel{x_k}{\longrightarrow} u^{\langle m_t\rangle}$ ($m_t$-times) or $u\stackrel{i}{\longrightarrow} u''$ occurs in
${\mathcal T}(w)$ and (if yes) to compute
\begin{itemize} 
\item  ${\sf code}(u^{\langle m_t\rangle})$ in the case $s_t=(x_k,m_t)$ (a multi-deletion move (T.1)), or 
\item ${\sf code}(u'')$ in the case $s_t=i$ (a march move (T.2)). 
\end{itemize}
\end{proposition}

\begin{proof}
By Proposition \ref{prop:childCodePoly}, $L'\leq L$. Thus in our run-time analysis, we replace $L'$ by $L$.

Proposition \ref{prop:accBoxPoly} finds 
$\mathbf{z}_u:=(r,c)$ (or determines it does not exist) in $O(L^{2})$-time. If $\mathbf{z}_u$ does not exist then
$u$ is dominant and thus vexillary; output $s_t$ is invalid. Thus we assume henceforth that $\mathbf{z}_u$ exists.

\noindent
{\sf Case 1:}  ($s_t=(x_k,m_t)$.)   
Proposition \ref{prop:wPoly} finds $u(1),\ldots,u(L')$ in $O(L^2)$-time.
Determine (taking $O(L^2)$ time) if 
\begin{equation}
\label{eqn:Oct1xyz}
c_r-\left(\left(\min_{i\in[r]}{u(i)}\right)-1\right)\geq m_t,
\end{equation}
holds. We claim that $s_t$ is valid if and only if (\ref{eqn:Oct1xyz}) holds and $k=r$.
Indeed, observe
\begin{equation}
\label{eqn:Oct1abc}
\#\{\text{boxes in row $r$ of ${\sf Dom}(u)$}\}=
\left(\min_{i\in[r]}{u(i)}\right)-1.
\end{equation} 
Thus, (\ref{eqn:Oct1xyz}) is equivalent to the existence of 
$m_t$ boxes in row $r$ of $D(u)\smallsetminus {\sf Dom}(u)$. By (T.1), if $k=r$
this is equivalent to being able to apply $\bullet \stackrel{x_r}{\longrightarrow}\bullet$ successively $m_t$-times. 

Finally, if $s_t$ is valid, by $m_t$ applications of Proposition \ref{prop:childCodePoly} (I),
\begin{equation}\label{eq:delCode}
{\sf code}(u^{\langle m_t\rangle})=({\widetilde c}_1,\ldots,{\widetilde c}_{r-1},{\widetilde c}_{r}-m_t,{\widetilde c}_{r+1},\ldots, {\widetilde c}_{L'}).
\end{equation}
Hence we can output
(\ref{eq:delCode}) in $O(L^2)$-time. 

\noindent
{\sf Case 2:} ($s_t=i$.) By Proposition \ref{prop:wPoly}, determine $u(1),\ldots,u(L')$ from ${\sf code}(u)$ in $O(L^{2})$-time. In particular this 
computes $\mathbf{x}:=(i,u(i))$ in $O(L^{2})$-time. 
To decide if $s_t$ is valid we must determine if $\mathbf{x}\in {\sf Piv}(\mathbf{z}_u)$. 
 To do this, first calculate (in $O(L)$-time) 
 \[u_{NW}(\mathbf{z}_u):=\{(j,u(j)) \ : \ j<r, u(j)<c\}.\]
By definition,
 \[{\sf Piv}(\mathbf{z}_u)=\{(j,u(j))\in u_{NW}(\mathbf{z}_u) \ : \ \nexists (h,u(h))\in u_{NW}(\mathbf{z}_u) \mbox{ with } h>j, u(h)>u(j)\}.\]
${\sf Piv}(\mathbf{z}_u)$
takes $O(L)$-time to compute since $\#u_{NW}(\mathbf{z}_u)\leq r-1\leq L-1$. Hence
 we check if $\mathbf{x}\in {\sf Piv}(\mathbf{z}_u)$ in $O(L)$-time. If this is false, we output a rejection. Otherwise,
 Proposition \ref{prop:childCodePoly}
 outputs ${\sf code}(u'')$ in $O(L^2)$-time.
\end{proof}

\begin{proposition}\label{prop:lengthWitPoly} If $S=(s_1,\ldots,s_h)\in {\sf Trans}(w,v)$ then $h\leq L^2$.
\end{proposition}

\begin{proof}
Let $w:=w_0\stackrel{s_1}{\longrightarrow} w_1\stackrel{s_2}{\longrightarrow}\ldots \stackrel{s_{h}}{\longrightarrow} w_h=v$ be the path in $\TT(w)$ associated to $S$. By (T.1) and (T.2), ${\mathbf z}_{w_{t+1}}$ is weakly northwest of ${\mathbf z}_{w_{t}}$. Hence,
for any fixed $r$, those $t\in[0,h-1]$ with $\mathbf{z}_{w_t}$ in row $r$ form an interval $I^{(r)}\subseteq [0,h-1]$.  Since $1\leq r\leq L$, it suffices to prove
\begin{equation}
\label{eqn:July31abc}
\#I^{(r)}\leq 2(r-1).
\end{equation}

By (X.1) the transition moves acting on row $r$ alternate between multi-(T.1) moves
$(x_r,m_t)$ and (T.2) moves. Thus to show (\ref{eqn:July31abc}), it is enough to prove
\begin{equation}
\label{eqn:marchMax} 
\#\{t\in I^{(r)}:w_{t-1}\to w_t \text{ \ is a (T.2) move}\}\leq r-1.
\end{equation}
Consider a march move $i$ with $\mathbf{z}_{w_{t-1}}=(r,c)$ and $\mathbf{x}=(i,w_{t-1}(i))\in {\sf Piv}(\mathbf{z}_{w_{t-1}})$. 
By (T.2), if $(r,c')\in D({w_{t-1}})$ is in the same connected component as $\mathbf{z}_{w_{t-1}}$, the move $i$ takes $(r,c')$ strictly north of row $r$.
Thus, each march move strictly reduces the number of components in row $r$. Let $t_0=\min\{t\in I^{(r)}\}$.
Since there are at most $r$ $\bullet$'s weakly above row $r$, $D(w_{t_0})$ has at most $r-1$ (non-dominant) components in row $r$. Hence (\ref{eqn:marchMax}) holds, as desired.
\end{proof}

\begin{proposition}\label{prop:TabPoly}
Let $v$ be vexillary with ${\sf code}(v)=(c_1,\ldots,c_{L'})$ and $L'\leq L$.
There exists an $O(L^{2})$-time algorithm to check if $R=(r_{ij})_{1\leq i,j\leq L'}$ is $R=R(T)$ 
for some $T\in{\sf SSYT}(\lambda(v),\phi(v))$.
\end{proposition}
\begin{proof}
Since $L'\leq L$, it is $O(L^2)$-time to calculate $\phi(v), \lambda(v)$.
Let
\[\lambda_i:=\sum_{j=1}^{L'} r_{ij}, \text{\ for $1\leq i\leq L'$.}\]
First verify (in $O(L)$-time) that  $\lambda_i\geq\lambda_{i+1}$ for $1\leq i\leq L'-1$.
Then $R=R(T)$ where $T$ is the (unique) row weakly increasing tableau of shape $\lambda$
with $r_{ij}$ many $j$'s in row $i$. 

To verify $T\in {\sf SSYT}(\lambda(v),\phi(v))$ we must check that it is 
(i) is flagged by $\phi(v)$, (ii) has shape $\lambda(v)$, and (iii) is semistandard.
For (i), we need 
\begin{equation}
\label{eqn:needabc1}
r_{ij}=0 \mbox{ if } j> \phi(v)_i, \mbox{ for all }i,j\in[L'].
\end{equation}
 For (ii), we need 
\begin{equation}
\label{eqn:needabc2}
\lambda_i=\lambda(v)_i \mbox{ for each } i\in[L'].
\end{equation}
 For (iii), 
it remains to ensure that $T$ is column strict, i.e.,
\begin{equation}
\label{eqn:needabc3}
\sum_{j'\leq j }r_{i+1,j'}\leq \sum_{j'< j} r_{i,j'
}  \mbox{ for each } i\in[L'-1], j\in[L'].
\end{equation}
We found the inequalities (\ref{eqn:needabc2}) and (\ref{eqn:needabc3}) from a (preprint) version of \cite{Narayanan}.
 The inequalities (\ref{eqn:needabc1}), (\ref{eqn:needabc2}),
  and (\ref{eqn:needabc3}) can be checked in $O(L^{2})$-time since $i,j\in[L']\subseteq[L]$. 
\end{proof}

The following completes our proof that we can check that $(S,R)\in X_{\alpha,w}$ in $L^{O(1)}$-time.

\begin{proposition}\label{prop:NPwitnessPoly}
Given $(S,R)\in X$ and $({\sf code}(w),\alpha)$, one can determine if $(S,R)\in X_{\alpha,w}$ in $L^{O(1)}$-time.
\end{proposition}
\begin{proof}
By Propositions \ref{prop:movePoly} and \ref{prop:lengthWitPoly} combined, one determines in
$O(L^4)$-time if $S$ encodes a path 
$w:=w_0\stackrel{s_1}{\longrightarrow} w_1\stackrel{s_2}{\longrightarrow}\cdots \stackrel{s_h}{\longrightarrow} w_h=v$ in $\TT(w)$. If so, the length of 
${\sf code}(v)$ is at most $L$. Thus, using Theorem \ref{thm:vexCode}, one checks $v$ is vexillary in $O(L^{3})$-time. This decides if $S$ satisfies (X.1'). 
Proposition \ref{prop:TabPoly} checks $R$ satisfies (X.2') in $O(L^{2})$-time. Finally since $h\leq L^2$, computing ${\sf delwt}(S)$ takes $O(L^2)$-time.
Hence (X.3') is checkable in $O(L^2)$ time.
\end{proof}

\noindent
\emph{Proof of Theorem~\ref{prop:schubinsharpP}:}
By Proposition \ref{lemma:transbasic}, $\#X_{\alpha,w}=c_{\alpha,w}$. 
By Proposition \ref{prop:lengthWitPoly}, $(S,R)\in \#X_{\alpha,w}$ only if the list $S$ has at most
$L^2$ elements. Assuming this, we check $(S,R)$ satisfies (X.1) and (X.2) in $O(L^2)$-time. Using
Proposition \ref{prop:NPwitnessPoly},  we can verify $(S,R)\in X_{\alpha,w}$ in $L^{O(1)}$-time. Thus, given input $\alpha$ and ${\sf code}(w)$, computing $c_{\alpha,w}$ is in $\#{\sf P}$.
\qed

\subsection{Hardness, and the conclusion of the proof of Theorem \ref{thm:second}} 
\label{sec:schubinsharpPcomp}
\emph{Schur polynomials} are an important
basis of the vector space of symmetric polynomials. 
The {Schur polynomial} 
$s_{\lambda}=a_{\lambda+\delta}/a_{\delta} \text{ \ \ where $\lambda=(\lambda_1\geq \lambda_2\geq \cdots\geq \lambda_n\geq 0)$, $a_{\gamma}:=\det(x_{i}^{\gamma_j})_{i,j=1}^n$}$,
and $\delta=(n-1,n-2,\ldots,2,1,0)$. The flagged Schur function of Section \ref{sec:vex} is a generalization of the Schur polynomial.

A permutation $w$ is \emph{grassmannian} if it has at most one \emph{descent} $i$, i.e., where $w(i)>w(i+1)$.
	Given a partition $(\lambda_1\geq \lambda_2\geq \cdots\geq \lambda_L\geq 0)$ define a grassmannian permutation
	$w_{\lambda}$ by setting 
	\[w_{\lambda}(i)=i+\lambda_{L-i+1} \text{\ for $1\leq i\leq L$.}\] 
	For $w_{\lambda}$ grassmannian, it is well-known (see, e.g., \cite{Manivel}) that
\begin{equation}
\label{eqn:Jul17fff}
{\sf code}(w_{\lambda})=(\lambda_L,\lambda_{L-1},\ldots, \lambda_1).
\end{equation}  
 
Moreover, 
\begin{equation}
\label{eqn:Jul17abc}
{\mathfrak S}_{w_{\lambda}}=s_{\lambda}(x_1,\ldots,x_L)=\sum_{\alpha\in\mathbb{Z}_{\geq0}^L} K_{\lambda,\alpha} x^{\alpha},
\end{equation}
where $K_{\lambda,\alpha}$ is 
the \emph{Kostka coefficient}. This number counts semistandard
tableaux of shape $\lambda$ with content $\alpha$. 

By (\ref{eqn:Jul17abc}),
\begin{equation}
\label{eqn:Jul17xyz}
c_{\alpha,w_{\lambda}}=K_{\lambda,\alpha}.
\end{equation}
By Theorem~\ref{prop:schubinsharpP}, counting 
$c_{\alpha,w}$ is in $\#{\sf P}$.  Suppose there is an oracle to compute $c_{\alpha,w}$
in polynomial time in the input length of $({\sf code}(w),\alpha)$. This
input length is the same as for the input $\lambda,\alpha$ for $K_{\lambda, \alpha}$. Hence (\ref{eqn:Jul17fff}) 
and (\ref{eqn:Jul17xyz}) combined imply a polynomial-time counting reduction from $\{c_{\alpha,w}\}$ to Kostka coefficients. Now H.~Narayanan \cite{Narayanan} proved that counting $K_{\lambda,\alpha}$ is a 
$\#{\sf P}$-complete problem. Thus counting $c_{\alpha,w}$ is a $\#{\sf P}$-{\sf complete} problem. \qed

\begin{remark}
\label{remark:July17}
Suppose the input for counting $c_{\alpha,w}$ is $(\alpha,w)$ where $w\in S_n$ (in one-line notation). Then the above counting reduction is 
not polynomial time in the input length of the Kostka problem.
For example, suppose $\lambda=\alpha=(2^L,2^L,\ldots, 2^L)$ ($L$-many). Then the input length of this instance of the Kostka problem is $2L^2\in O(L^2)$. On the other hand, $w_{\lambda}\in S_{L+2^L}$. 
Therefore, a polynomial time algorithm for the Schubert coefficient problem in $n$ would have $\Omega(2^L)$ run time for the Kostka problem.

It seems unlikely that there is a polynomial-time reduction under this input assumption. This is our justification to encode $w$
via ${\sf code}(w)$ rather than one line notation. \qed
\end{remark}

\section*{Acknowledgments}
AY was supported by an NSF grant, a Simons Collaboration grant and a UIUC Campus Research Board grant. CR was supported by the National Science Foundation
Graduate Research Fellowship Program under Grant No. DGE – 1746047. We acknowledge Mathoverflow, 
S.~Kintali's blog ``My brain is open'' and R.~O'Donnell's Youtube videos (from his class at Carnegie Mellon)
for background. We thank Philipp Hieronymi, Alexandr Kostochka, Cara Monical, Erik Walsberg, Douglas West, Alexander Woo,  for helpful comments and conversations. We also thank the anonymous referee for their insightful suggestions that improved
the clarity of this paper.


\begin{thebibliography}{99}
\bibitem{FPSAC} 
A.~Adve, C.~Robichaux, and A.~Yong, \emph{Computational complexity, Newton polytopes, and Schubert polynomials}, Proceedings
of the $31$st Conference on Formal Power Series and Algebraic Combinatorics (Ljubljana), S\'{e}m. Lothar. Combin. 82B (2020), Art. 52, 12 pp. 
\bibitem{earlierversion}
\bysame, \emph{Complexity, combinatorial positivity, and Newton polytopes}, preprint, 2018. \textsf{arXiv:1810.10361v1}
 
\bibitem{BJS}
S.~Billey, W.~Jockusch and
R.~P.~Stanley, \emph{Some combinatorial properties of Schubert polynomials}, 
J.~Algebraic Combin.~{\bf 2}(1993), no. 4, 345--374.
\bibitem{Fink} A.~Fink, K.~M\'{e}sz\'aros, and A.~St.~Dizier, \emph{Schubert polynomials as integer point transforms of generalized permutahedra}, Adv. Math. {\bf 332} (2018), 465--475.
\bibitem{FGRS} S.~Fomin, C.~Greene, V.~Reiner, and M.~Shimozono, \emph{Balanced labellings and Schubert polynomials}, European J. Combin. 18 (1997), no. 4, 373--389.
\bibitem{Fulton:YT} W.~Fulton, \emph{Young tableaux. With applications to representation theory and geometry}. London Mathematical Society Student Texts, 35. Cambridge University Press, Cambridge, 1997.
\bibitem{Knutson.Yong} A.~Knutson and A.~Yong, \emph{A formula for $K$-theory truncation Schubert calculus}, Int. Math. Res. Not. 2004, no. 70, 3741--3756. 
\bibitem{LS:transition} A.~Lascoux and M.~-P. Sch\"{u}tzenberger, \emph{Schubert polynomials and the Littlewood-Richardson
rule}, Letters in Math. Physics 10 (1985), 111--124.
\bibitem{LS1} \bysame, \emph{Polyn\^{o}mes de Schubert}, C.~R.~Acad.~Sci.~Paris S\'{e}r.~I Math.
{\bf 294} (1982), 447--450.
\bibitem{Manivel}
 L. Manivel, \emph{Symmetric functions, Schubert polynomials and degeneracy loci}. Translated from the
1998 French original by John R. Swallow. SMF/AMS Texts and Monographs, American Mathematical
Society, Providence, 2001.
\bibitem{MTY} C.~Monical, N.~Tokcan and A.~Yong, \emph{Newton polytopes in algebraic combinatorics}, Sel. Math. 25(5) (2019), no. 66, 37 pp.
\bibitem{Narayanan} H. Narayanan, \emph{On the complexity of
computing Kostka numbers and Littlewood-Richardson coefficients},
J.~Alg.~Comb., Vol. 24, N. 3, 2006, 347--354.
\bibitem{Combopt} C.~H.~Papadimitriou and K.~Steiglitz, Kenneth, \emph{Combinatorial optimization: algorithms and complexity.} Corrected reprint of the 1982 original. Dover Publications, Inc., Mineola, NY, 1998. xvi+496 pp. 
\bibitem{Schrijver} A.~Schrijver, \emph{Theory of Linear and Integer Programming}, John Wiley  \& sons, 1998.
\bibitem{Shanan} R.~P.~Stanley, \emph{Some Schubert shenanigans}, preprint, 2017. \textsf{arXiv:1704.00851}
\bibitem{Valiant}
L.~G.~Valiant, \emph{The complexity of computing the permanent}, Theoret. Comput. Sci., 8(2):189--201, 1979.
\bibitem{Weigandt}
A.~Weigandt, \emph{Schubert polynomials, $132$-patterns, and Stanley's conjecture}, 
 Algebr. Comb. 1 (2018), no. 4, 415--423. 
\end{thebibliography}
\end{document}